\documentclass[12pt,reqno]{amsart}
\usepackage[non-sorted-cites]{amsrefs}
\usepackage[utf8]{inputenc}
\usepackage[margin=2.9cm,headheight=1000pt,textheight=30cm,marginpar=2cm]{geometry}
\usepackage{color, amssymb}
\usepackage{comment}
\usepackage{stmaryrd}
\usepackage{enumerate}
\usepackage{dsfont}
\usepackage{thmtools, thm-restate}

\usepackage{hyperref}
\usepackage{fancyhdr}

\newtheorem{corollary}{Corollary}
\newtheorem{theorem}{Theorem}

\newtheorem{lemma}{Lemma}
\newtheorem*{theorem*}{Theorem}
\newtheorem*{lemma*}{Lemma}
\newtheorem{prop}{Proposition}

\renewcommand{\implies}{\Rightarrow}
\newcommand{\ii}{\mathrm{i}} 
 
\newcommand{\re}{{\rm Re } }
\newcommand{\E}{\mathbb E} 
\newcommand{\PP}{\mathbb P}
\newcommand{\1}{{\mathbf 1}}
\newcommand{\OO}{{\rm O}}
\newcommand{\oo}{{\rm o}}
\newcommand{\rd}{{\rm d}}

\newcommand{\e}{\varepsilon}
\newcommand{\ee}{\epsilon}

\newcommand{\ttL}{n_{\mathcal L}}

\makeatletter
\def\@tocline#1#2#3#4#5#6#7{\relax
  \ifnum #1>\c@tocdepth 
  \else
    \par \addpenalty\@secpenalty\addvspace{#2}%
    \begingroup \hyphenpenalty\@M
    \@ifempty{#4}{%
      \@tempdima\csname r@tocindent\number#1\endcsname\relax
    }{%
      \@tempdima#4\relax
    }%
    \parindent\z@ \leftskip#3\relax \advance\leftskip\@tempdima\relax
    \rightskip\@pnumwidth plus4em \parfillskip-\@pnumwidth
    #5\leavevmode\hskip-\@tempdima
      \ifcase #1
       \or\or \hskip 2em \or \hskip 3em \else \hskip 4em \fi%
      #6\nobreak\relax
    \dotfill\hbox to\@pnumwidth{\@tocpagenum{#7}}\par
    \nobreak
    \endgroup
  \fi}
\makeatother

\usepackage[normalem]{ulem} 

\begin{document}

\begin{abstract}
We prove a lower bound on the maximum of the Riemann zeta function in a typical short interval on the critical line. Together with the upper bound from \cite{ArgBouRad2020}, this implies
tightness of
$$
\max_{|h|\leq 1}|\zeta(\tfrac 12+\ii \tau+\ii h)|\cdot \frac{(\log\log T)^{3/4}}{\log T},
$$
for large $T$, where $\tau$ is uniformly distributed on $[T,2T]$.
The techniques are also applied to bound the right tail of the maximum, proving the distributional decay $\asymp y e^{-2y}$ for $y$ positive. This confirms the Fyodorov-Hiary-Keating conjecture, which states that the maximum of $\zeta$ in short intervals lies in the universality class of logarithmically correlated fields. 
\end{abstract}

\title{The Fyodorov-Hiary-Keating Conjecture. II.}

\author{Louis-Pierre Arguin}
\address{Department of Mathematics, Baruch College and Graduate Center, City University of New York, USA}
\email{louis-pierre.arguin@baruch.cuny.edu}
\author{Paul Bourgade}
\address{Courant Institute, New York University, USA}
\email{bourgade@cims.nyu.edu}
\author{Maksym Radziwi\l\l}
\address{Department of Mathematics, Caltech,  Department of Mathematics, U.T. Austin, USA}
\email{maksym.radziwill@gmail.com}

\begingroup
\def\uppercasenonmath#1{} 
\let\MakeUppercase\relax 
~\vspace{-0.1cm}
\maketitle
\endgroup

~\vspace{-0.9cm}

\setcounter{tocdepth}{1}
\tableofcontents

\section{Introduction}
\fancyhead{}
\pagestyle{fancy}
\renewcommand{\headrulewidth}{0pt}
\fancyhead[LE,RO]{\thepage}
\fancyhead[CO]{L.-P. Arguin, P. Bourgade and M. Radziwi\l\l}
\fancyhead[CE]{The Fyodorov-Hiary-Keating Conjecture. II.}
 \fancyfoot{}

\noindent 

\noindent The distribution of the Riemann zeta function on the critical line is conjecturally related to random matrices,  a fact discovered  by Montgomery \cite{Mon1972} for local statistics of the zeros.
It was extended in many directions including to distributions for families of $L$-functions \cite{KatSar1999} and their moments \cite{KeaSna2000}.
Fyodorov, Hiary \& Keating \cite{FyoHiaKea2012} and Fyodorov \& Keating \cite{FyoKea2014} proposed to further expand the scope of this analogy at the level of extreme values. Based on a similar conjecture for random unitary matrices,  they 
put forward the very precise asymptotics   $$
   \frac{1}{T} \cdot \text{\rm meas} \Big \{ T \leq t \leq 2T : \max_{|h| \leq 1} |\zeta(\tfrac 12 + \ii  t + \ii  h)| > e^y \cdot \frac{\log T}{(\log\log T)^{3/4}} \Big \} \rightarrow F(y), 
   $$
   as $T \rightarrow \infty$, where the  limiting distribution function $F$ satisfies $F(y)\sim C ye^{-2y}$ for large $y$.
While the explicit form of $F$ is not expected to be universal,  the exponent $3/4$ and the tail asymptotics $ ye^{-2y}$ characterize the universality class of logarithmically correlated fields. 

\noindent In the first part of this series \cite{ArgBouRad2020},  we showed the upper bound of this conjecture,  $F(y) \ll y e^{-2y}$. 
 The main goal of this paper is to complete this work and establish tightness in the Fyodorov-Hiary-Keating conjecture, 
by showing $F(y) \to 1$ as $y \to - \infty$.  The following is the main result.

 \begin{theorem}
 \label{thm: left tail}
There exists $c>0$ such that for any $T\geq 100$ and $0\leq y\leq (\log\log T)^{1/10}$ we have 
$$
\frac{1}{T} \cdot \text{\rm meas} \Big \{ T \leq t \leq 2T : \max_{|h|\leq 1} |\zeta(1/2+\ii t+\ii h)|< e^{-y}\frac{\log T}{(\log\log T)^{3/4}} \Big \} \leq c^{-1} y^{-c}.
$$
 \end{theorem}

A direct consequence of the above result and \cite[Theorem 1]{ArgBouRad2020} is the expected tightness of maxima on short intervals, and existence of subsequential limits.

 \begin{corollary}\label{eqn:tightness}
   For every $\varepsilon > 0$ there exists $C > 0$ such that for any $T\geq 100$,  for $t \in [T, 2T]$ in a set of measure larger than $(1 - \varepsilon) T$ we have
   $$
   \left|\max_{|h| \leq 1} \log |\zeta(\tfrac 12 + \ii t + \ii h)| - (\log\log T - \frac{3}{4} \log\log\log T)\right| \leq C.
   $$
In particular, there exists a subsequence $T_{\ell} \to\infty$ and a distribution function $F$  such that
   $$
   \frac{1}{T_{\ell}}\cdot  \text{\rm meas} \Big \{ t \in [T_{\ell}, 2 T_{\ell}]: \max_{|h| \leq 1} |\zeta(\tfrac 12 + \ii t + \ii h)| > e^y \cdot \frac{\log T_\ell}{(\log\log T_\ell)^{3/4}} \Big \} \rightarrow F(y),
   $$
   uniformly in $y \in \mathbb{R}$ outside of a countable set.
\end{corollary}

Previous results in the direction of Theorem \ref{thm: left tail} were limited to the first order $\log T$,  conditionally on the Riemann Hypothesis by Najnudel \cite{Naj2018} and unconditionally by the authors with Belius and Soundararajan \cite{ArgBelBouRadSou2019}.  This contrasts with the developments on the upper bound, starting with the first order $\log T$ proved in \cite{Naj2018, ArgBelBouRadSou2019}, then the second order by Harper \cite{Har2019},  and finally the optimal upper bound with the tail distribution \cite{ArgBouRad2020}.  In fact, the present work builds on many techniques developed for the upper bound in \cite{ArgBouRad2020},  as well as new inputs as we now explain.\\

Progress towards the Fyodorov-Hiary-Keating conjecture has relied on the observation that the maxima of $|\zeta|$ on a short interval are related to extremes of branching processes. Indeed,
 the emergence of large values of $|\zeta|$  follows a scenario first identified by Bramson \cite{Bra1978} in the setting of branching Brownian motion.  As explained in the introduction of \cite{ArgBouRad2020},
the explicit branching structure behind $\zeta$ comes from the Dirichlet polynomials $(S_k(h), k\geq 1)$, $|h|\leq 1$, defined in (\ref{eqn: Sk}).  
These polynomials behave similarly to correlated random walks, the time index $k$ corresponding to primes in the loglog scale.
Bramson's scenario translates into the ballistic behavior of  
$(S_k(h), k\leq n_{\mathcal L})$ conditioned not to cross an upper barrier.
Estimating the maximum of $\zeta$ with a precision of order one is a delicate task because the final index $n_{\mathcal L}$ needs to be $y$-dependent and very large, i.e., the sum must include primes very close to $T$.

The proof of Theorem \ref{thm: left tail} is decomposed into two parts. First, it is shown that large values of $S_{n_{\mathcal L}}$ indeed imply large values of $\log|\zeta|$, cf.~Proposition \ref{prop:zetalocal}. Second, we prove that large values of $S_{n_{\mathcal L}}$ of the claimed size are achieved, see Proposition \ref{prop:almostsuregood}.
Proposition \ref{prop:almostsuregood} builds on two techniques from  \cite{ArgBouRad2020}, namely the introduction of a lower barrier ensuring that large deviations of the increments of $S_k$ can be obtained even for large primes,  and the precise encoding through Dirichlet sums of the event that $S$ remains in the corridor defined by an upper barrier and lower barrier.

To justify that large values of $S_{n_{\mathcal L}}$ imply large values of $\log|\zeta|$,  the first order asymptotics from \cite{ArgBelBouRadSou2019} relied on working on the right of the critical line.  However,  implementing  this method for the much finer tightness would be considerably more involved.  Instead,  Proposition \ref{prop:zetalocal} uses a new,  simpler argument allowing to work directly on the critical line,  through an integral approximation of $\zeta$ by a finite Euler product (Lemma \ref{le:Analytic}),  and a control of the regularity in $h$ of $S_{n_{\mathcal L}}$ on high points    (Proposition \ref{prop:Low}). 

With these methods developed for Theorem \ref{thm: left tail}, we can also complement the upper bound $F(y) \ll y e^{-2 y}$ from \cite[Theorem 1]{ArgBouRad2020}, and show that $F(y) \asymp y e^{-2y}$ for positive $y$. 

 \begin{theorem}
 \label{thm: right tail} For any $C>0$ there exists $c>0$ such that 
for any  $10\leq y\leq C\frac{\log\log T}{\log\log\log T}$,  we have
\begin{equation}
\label{eqn: LB right tail}
\PP\Big(\max_{|h|\leq 1} |\zeta(1/2+\ii\tau+\ii h)|>e^y\frac{\log T}{(\log\log T)^{3/4}}\Big)\geq c\, y e^{-2y} e^{-y^2/\log\log T}.
\end{equation}
 \end{theorem}

This proves the matching lower bound of the upper tail not only in the in the exponential regime $y\leq \sqrt{\log\log T}$ but also in the Gaussian regime $\sqrt{\log\log T}\leq y\leq C\frac{\log\log T}{\log\log\log T}$,  because the proof of \cite[Theorem 1]{ArgBouRad2020} implies the Gaussian decay in this range.

The estimate \eqref{eqn: LB right tail} essentially
matches the range $y=\oo(t)$ proved by Bramson  \cite{Bra1983}  for the branching Brownian motion up to time $t$.
(The time $t$ corresponds to $\log\log T$ in our problem.)
It is weaker by a logarithmic factor as it would corresponds to $y\leq C\, t/\log t$ in the branching Brownian motion case.
We are not aware of other examples of log-correlated processes where the order of the right tail of the maximum is known to this level of precision.  In fact, any form of decay has only been proved for a few models in this universality class.
Notably 
for the branching random walk,  the best known range is  $y=\OO(\sqrt{t})$ \cite{BraDinZei2016bis},  which matches the known precision for the two dimensional discrete Gaussian free field on the $N\times N$ square grid,  $y=\OO(\sqrt{\log N})$ \cite{Din2013,DinZei2014,BraDinZei2016}.
A finer control of the contributions from small primes in the random walk would improve this range of $y$ in Theorem \ref{thm: right tail} to match Bramson's.  \\

The distributional limit obtained in Corollary \ref{eqn:tightness} is presumably unique but we believe this is out of reach with current number theory techniques. Moreover, no explicit formula for $F$ was conjectured.
Indeed,  denoting $U_n$ a Haar-distributed $n\times n$  unitary matrix, \cite{FyoHiaKea2012} proposed a very precise limiting distribution for 
\begin{equation}\label{eqn:RMT}
\sup_{|z|=1} \big(\log|{\rm det}(z-U_n)|-\log n + \frac{3}{4} \log \log n\big),
\end{equation}
but as explained in \cite{FyoKea2014} this limit is not expected to coincide with $F$: It
primarily suggested 
the characteristic exponent $3/4$ and the tail distribution $y e^{-2y}$ for $\zeta$, which are the prominent signatures of extremal statistics in log-correlated fields \cite{CarLed01}.
Progress on a limit for \eqref{eqn:RMT} culminated in the breakthrough proofs of 
tightness \cite{ChhMadNaj2018}  and uniqueness \cite{PaqZei2022} of a limiting distribution for the more general circular beta ensembles,  after initial steps verifying the first  \cite{ArgBelBou2017} and second order terms \cite{PaqZei2018}.

The exact form of the limiting distribution of  \eqref{eqn:RMT},  and universality of its right tail, remain open.\\

\noindent {\bf Acknowledgment.} L.-P.~A. is supported by the grants NSF CAREER 1653602 and NSF DMS 2153803,  P.~B.  is supported by the NSF grant DMS 2054851, and M.~R.  is supported by the NSF grant DMS 1902063.\\

\noindent {\bf Notation.} Throughout the paper, $\tau$ will denote a random variable uniformly distributed in $[T, 2T]$, and $T$ will be some large parameter that is usually taken to go to infinity. 
With this notation, for any measurable function $f$ on $[T,2T]$ and event $A$, we have
 $$
 \mathbb{P}( f(\tau) \in A) := \frac{1}{T} \cdot \text{\rm meas} \Big \{ T \leq t \leq 2T : f(t) \in A \Big \}.  
 $$

 \section{Proof of Theorem \ref{thm: left tail}}
 \label{sect: proof thm 1}

 Let
 $$
 n_0 := \lfloor y \rfloor \text{ and } n := \lfloor \log\log T \rfloor \text{ and } n_{\mathcal{L}} := n - n_0. 
 $$
 For $n_0 \leq k \leq n_{\mathcal{L}}$ and $|h|\leq 1$, we consider the partial sums
 \begin{equation}
\label{eqn: Sk}
S_k(h)  = \sum_{n_0<\log\log  p \leq k} \re \Big( p^{-(1/2 + \ii \tau + \ii h)} + \frac 12 \cdot p^{- 2(1/2 + \ii \tau + \ii h)} \Big ).
 \end{equation}
 Essentially one can think of $S_k(h)$ as an approximation to
 $$
\int_{\mathbb{R}} \log |\zeta(\tfrac 12 + \ii  \tau + \ii h + \ii x)| f ( e^k x)  \ e^k \rd x. 
 $$
 for some choice of smoothing with $\widehat{f}$ compactly supported. 

 We will show that with high probability the local maxima of $S_{n_{\mathcal{L}}}(h)$ arise at those $h$ at which the partial sums $S_k(h)$ evolve in a predictable manner as $k$ runs from $n_0$ to $n_{\mathcal{L}}$. More precisely, the partial sums $S_k(h)$ of maximizing $h$'s are constrained between $L_k$ and $U_k$ (defined below) for all $n_0 \leq k \leq n_{\mathcal{L}}$. Once $k$ reaches $n_{\mathcal{L}}$ there are only $\OO(1)_y$ well-spaced (i.e, $1/\log T$ spaced) values of $h$ that can satisfy all those constraints, thus identifying the maximum almost uniquely.  

 In order to define $L_k$ and $U_k$ we introduce the \textit{slope},
 \begin{equation}
 \label{eqn: alpha}
 \alpha=1-\frac{3}{4}\frac{\log n}{n},
 \end{equation}
 Furthermore given a function $f$, we define a symmetrized version,
 $$
 \mathcal{S}_{\mathcal{L}}(f)(k) := \begin{cases}
   f(k - n_0) & \text{ for } n_0 < k \leq \frac{n}{2}, \\ 
   f(n_{\mathcal{L}} - k)  & \text{ for } \frac{n}{2} < k < n_{\mathcal{L}}, \\
     0 & \text{ for } k \geq n_{\mathcal{L}} \text{ or } k \leq n_0.
   \end{cases}
 $$
 Then, the so-called \textit{barriers} (i.e., values $L_k$ and $U_k$) are defined as
\begin{align}
  \label{eq:barup}U_k & = \frac{y}{10} +\alpha (k-n_0) - 10 \mathcal{S}_{\mathcal{L}}(x \mapsto \log(x))(k), \\
\label{eq:bardown} L_k & = - 10 y +\alpha (k-n_0) - \mathcal{S}_{\mathcal{L}}(x \mapsto x^{3/4})(k).
\end{align}
 We now introduce the set of \textit{good points} $G_{\mathcal{L}}$, defined more generally for
 $n_0 \leq \ell \leq n_{\mathcal{L}}$ as
 \begin{equation}
 \label{eqn: G}
 \begin{aligned}
  G_0&=[-\tfrac 12,\tfrac 12]\cap e^{-(n_\mathcal L - n_0)}\mathbb Z,\\
 G_{\ell} &= \{ h\in G_0 : S_{k}(h) \in[L_k,  U_k] \text{ for all } k \leq \ell \}. 
  \end{aligned}
 \end{equation}
 We will show that with high probability the local maximum belongs to $G_{\mathcal{L}}$.  

We first comment on the above choices of barriers and discrete sets.
The interval $[-\tfrac 12,\tfrac 12]$ defining $G_0$ needs to be strictly included in  the original interval $[-1,1]$, as it will be apparent in the proof of  Proposition \ref{prop:max}.  Moreover,  the discretization step $e^{-(n_\mathcal L - n_0)}$ will be convenient for the proof of Proposition \ref{prop: comp second} as it corresponds to the number of steps of the random walk \eqref{eqn: Sk}, but it is not essential and any step in $[e^{-n_\mathcal L },e^{-(n_\mathcal L - n_0)}]$ would work.  However,  contrary to \cite{ArgBouRad2020},  it is essential that the upper barrier is convex and not concave,  as we will see in the proof of Proposition \ref{prop: comp second}.

 The proof of the main theorem reduces now to two main propositions. In the first proposition, we show how the local maxima of the zeta function arise from the good points $h \in G_{\mathcal{L}}$. 

 \begin{prop} \label{prop:zetalocal}
   There exists an absolute constant $C>0$ such that uniformly in $T\geq 100$ and $0\leq y\leq (\log\log T)^{1/10}$ we have 
   \begin{align*}
\mathbb{P} \Big ( \max_{|h| \leq 1} & \log |\zeta(\tfrac 12 + \ii \tau + \ii  h)| \geq n - \frac{3}{4} \log n - 100 y - C \Big ) \geq \mathbb{P} \Big ( \exists h \in G_{\mathcal{L}} \Big ) + \OO (e^{-y}).
   \end{align*}
 \end{prop}

 In the second proposition, we then show that good points exist with high probability.
 
 \begin{prop} \label{prop:almostsuregood}
   There exists $c>0$ such that uniformly in $T\geq 100$ and $0\leq y\leq (\log\log T)^{1/10}$ we have 
   \begin{align*}
   \mathbb{P} \Big ( \exists h \in G_{\mathcal{L}} \Big ) = 1 + \OO(y^{-c}).
   \end{align*}
 \end{prop}

 Combining Proposition \ref{prop:zetalocal} and Proposition \ref{prop:almostsuregood} yields Theorem \ref{thm: left tail}. We now describe the proofs of Proposition \ref{prop:almostsuregood} and Proposition \ref{prop:zetalocal}

 \subsection{Proof of Proposition \ref{prop:zetalocal}}

 The proof of Proposition \ref{prop:zetalocal} breaks down into two propositions.

 \begin{prop} \label{prop:max}
There exists $C>0$ such that for any $1000 < y < n^{1/10}$
   $$
\mathbb{P} \Big ( \max_{|h| \leq 1} \log |\zeta(\tfrac 12 + \ii  \tau + \ii h)| \geq \max_{h \in G_0} \min_{|u| \leq 1} (S_{n_{\mathcal{L}}}(h + u) + \sqrt{|u| e^{n_{\mathcal{L}}}}) - 2 C - 20 y \Big ) \geq 1 - \OO(e^{-y}).
$$
 \end{prop}
 
 We then show that with high probability for all $h \in G_{\mathcal{L}}$ and all $|u| \leq 1$, 
 \begin{align*}
   |S_{n_{\mathcal{L}}}(h + u) - S_{n_{\mathcal{L}}}(h)| \leq  20y + \sqrt{|u| e^{n_{\mathcal{L}}}}. 
 \end{align*}
 
 \begin{prop} \label{prop:Low}
For any $1000 < y < n^{1/10}$ we have
   $$
   \mathbb{P} \Big (\forall h \in G_{\mathcal{L}} \ \forall |u| \leq 1 : |S_{n_{\mathcal{L}}}(h + u) - S_{n_{\mathcal{L}}}(h)| \leq 20 y + \sqrt{|u| e^{n_{\mathcal{L}}}} \Big ) = 1 - \OO(e^{-y}). 
   $$
 \end{prop}
 On the event that there exists a $h \in G_{\mathcal{L}}$, Proposition \ref{prop:Low} now implies
 \begin{align*}
 \max_{v \in G_0} \min_{|u| \leq 1} (S_{n_{\mathcal{L}}}(v + u) + \sqrt{|u|e^{n_{\mathcal{L}}}}) & \geq \min_{|u| \leq 1} (S_{n_{\mathcal{L}}}(h + u) + \sqrt{|u| e^{n_{\mathcal{L}}}}) \\ & \geq S_{n_{\mathcal{L}}}(h) - 20y \geq n - \frac{3}{4} \log n - 50 y
 \end{align*}
 outside of a set of probability $\OO(e^{-y})$. Proposition \ref{prop:max} then yields that outside of a set of $\tau$ of probability $\ll e^{-y}$,
 $$
 \max_{|h| \leq 1} \log |\zeta(\tfrac 12 + \ii \tau + \ii h)| > n - \frac{3}{4} \log n - 100 y - 2C.
 $$
 In other words,
 $$
 \mathbb{P} (\exists h \in G_{\mathcal{L}} ) \leq \mathbb{P} \Big (\max_{|h| \leq 1} \log |\zeta(\tfrac 12 + \ii  \tau + \ii h)| > n - \frac{3}{4} \log n - 100 y - 2C \Big ) + \OO(e^{-y})
 $$
 and Proposition \ref{prop:zetalocal} follows. 
 
 \subsection{Proof of Proposition \ref{prop:almostsuregood}}
Cauchy-Schwarz inequality readily implies
 \begin{equation} \label{eq:pz}
 \mathbb{P} \Big ( \exists h \in G_{\mathcal{L}} \Big ) \geq \frac{\mathbb{E}[\#G_{\mathcal{L}}]^2}{\mathbb{E}[(\# G_{\mathcal{L}})^2]}.
 \end{equation}

 For fixed $k$ and $h,h'\in G_0$, we posit that the random variables $(S_k(h), S_k(h'))$ can be well approximated by two correlated Gaussian random variables $(\mathcal{G}_k(h), \mathcal{G}_k(h'))$ with,
 \begin{align} \label{eqn: Gk}
   \mathcal{G}_k(h) := \sum_{n_0 \leq j \leq k} \mathcal{N}_j \text{ and } \mathcal{G}_k(h') := \sum_{n_0 \leq j \leq k} \mathcal{N}_j',
 \end{align}
 where the increments $\mathcal{N}_{j}$ and $ \mathcal{N}_j'$ are Gaussian random variables with mean $0$, equal variance
 \begin{equation}\label{eqn:sk}
\E[\mathcal{N}_{k}^2]=\E[{\mathcal{N}'_{k}}^2]= \mathfrak{s}_k^2 := \sum_{e^{k-1} < \log p \leq e^{ k}} \Big ( \frac{1}{2 p} + \frac{1}{8 p^2} \Big ),
 \end{equation}
 and covariance
 \begin{equation}\label{eqn:rhok}
 \E[\mathcal{N}_k\mathcal{N}'_k]=\rho_k:=  \sum_{e^{ k-1} < \log p \leq e^{ k}}  \Big ( \frac{\cos(|h - h'| \log p)}{ 2 p} + \frac{\cos(2 |h - h'| \log p)}{ 8 p^2} \Big ).
 \end{equation}
The analog of the good sets \eqref{eqn: G} for the Gaussian random variables is
 $$
 \mathfrak{G}_{\mathcal{L}}^{\pm} := \# \Big \{ h \in G_0 : \mathcal{G}_k(h) \in [L_k \mp 1, U_k \pm 1] \text{ for all } n_0 \leq k \leq n_{\mathcal{L}} \Big \}. 
 $$
 We then show that in \eqref{eq:pz} we can replace the \textit{arithmetic} good set $G_{\mathcal{L}}$ by the purely \textit{probabilistic} good sets $\mathfrak{G}_{\mathcal{L}}^{\pm}$.
 \begin{prop} \label{prop:gauss}
Uniformly in  $T\geq 100$ and $100\leq y\leq n^{1/10}$, we have
   $$
   \frac{\mathbb{E}[\# G_{\mathcal{L}}]^2}{\mathbb{E}[(\# G_{\mathcal{L}})^2]} \geq \Big (1 + \OO(y^{-10}) \Big ) \frac{\mathbb{E}[\# \mathfrak{G}_{\mathcal{L}}^{+}]^2}{\mathbb{E}[(\# \mathfrak{G}_{\mathcal{L}}^{-})^2]}.
   $$
 \end{prop}
 
 \begin{proof}
     This result is an immediate consequence of the comparison with Gaussian random walks as stated in Propositions \ref{prop: 1 point} and  \ref{prop: 2 points} in the next Section \ref{sec:GaussAppr}.
\end{proof}

The problem is now reduced to a purely probabilistic computation. 
The proof of Proposition \ref{prop:almostsuregood} is concluded by the next proposition building on ideas of Bramson. 
    
 \begin{prop}
   \label{prop: comp second}
There is an absolute constant $c>0$ such that for any $T\geq 100$ and $c^{-1}\leq y\leq n^{1/10}$,
   $$
   \frac{\E[\#\mathfrak{G}_{\mathcal L}^{+}]^2}{\mathbb{E}[(\# \mathfrak{G}_{\mathcal{L}}^{-})^2]} \geq 1 - y^{- c}. 
   $$
 \end{prop}
 Combining the two above propositions with the lower bound from \eqref{eq:pz} yield Proposition \ref{prop:almostsuregood}. 

    \section{Approximations by Gaussian Random Walks}\label{sec:GaussAppr}
The proof of Proposition \ref{prop:gauss} relies on approximating one-point and two-point correlations in terms of correlations of Gaussian random variables, see Propositions \ref{prop: 1 point} and \ref{prop: 2 points} below. 
 Note that the one-point estimate contains an additional twist by a Dirichlet polynomial. This will be needed in the proof of Proposition \ref{prop:Low}.
The proofs of Propositions \ref{prop: 1 point} and \ref{prop: 2 points} are independent of the rest of the paper and  can be skipped on a first reading. 

\begin{restatable}{prop}{onepoint}
  \label{prop: 1 point}
  \label{lem: 1 point}
 Let $h\in[-1,1]$.
 Let $n_0 \leq \ell \leq n_{\mathcal{L}}$. 
 Let $(S_k(h), n_0 \leq k \leq n_{\mathcal{L}})$ and $(\mathcal G_k(h),n_0 \leq k \leq n_{\mathcal{L}})$
 be as in  Equations \eqref{eqn: Sk} and \eqref{eqn: Gk}. 
 Let $\mathcal{Q}$ be a Dirichlet polynomial of length $\leq \exp(\tfrac{1}{100} e^n)$ and supported on integers such that
 all their prime factors are greater than $\exp(e^{\ell})$. 
 Then, we have for $n_0$ large enough,
 \begin{align}
   \label{eqn: LB P}
   \mathbb{E}  \Big [ & |\mathcal{Q}(\tfrac 12 + \ii  \tau + \ii  h)|^2 \mathbf{1} \Big ( S_{k}(h) \in [L_k, U_k], k \leq \ell \Big ) \Big ]
    \\ \nonumber & \geq (1 + n_0^{-10}) \mathbb{E} \Big [ |\mathcal{Q}(\tfrac 12 + \ii  \tau + \ii  h)|^2 \Big ] \cdot \mathbb{P} \Big (\mathcal{G}_k(h) \in [L_k + 1, U_k - 1] , n_0 \leq k \leq \ell \Big )
 \end{align}
 and
   \begin{align}
 \label{eqn: UB P}
 \mathbb{E} \Big [ & |\mathcal{Q}(\tfrac 12 + \ii  \tau + \ii h)|^2 \mathbf{1} \Big ( S_{k}(h)\in [L_k, U_k], k \leq \ell \Big ) \Big ] \\ \nonumber & \leq (1+n_0^{-10}) \mathbb{E}\Big [ |\mathcal{Q}(\tfrac 12 + \ii  \tau + \ii  h)|^2 \Big ] \cdot \mathbb{P}(\mathcal G_{k}(h)\in [L_k-1, U_k+1], n_0 < k \leq \ell).
 \end{align}
 \end{restatable}

 \begin{restatable}{prop}{twopoints}
   \label{lem: 2 points}
   \label{prop: 2 points}
 Let $h,h'\in [-1,1]$. 
Consider $(S_k(h), S_k(h))$ and $(\mathcal G_k(h), \mathcal G_k'(h))$ for $n_0<k\leq n_\mathcal L$ as defined in Equations \eqref{eqn: Sk} and \eqref{eqn: Gk}. 
We have for $n_0$ large enough
   \begin{equation}
 \label{eqn: UB P 2}
 \begin{aligned}
 &\PP\big((S_{k}(h),S_k(h'))\in [L_k, U_k]^{2}, n_0< k\leq n_\mathcal L\big)\\
& \hspace{2cm} \leq (1+n_0^{-10}) \cdot \PP\big((\mathcal G_{k}(h), \mathcal G_k(h'))\in [L_k-1, U_k+1]^{2}, n_0< k\leq n_\mathcal L\big).
 \end{aligned}
 \end{equation}
   \end{restatable}
A similar lower bound can be proved, but is actually not needed in the proofs of Theorem \ref{thm: left tail} and \ref{thm: right tail}. \\

The proof of both propositions rely on an extension of  the techniques of \cite{ArgBouRad2020} to estimate the probability of events involving the partial sums \eqref{eqn: Sk} in terms of random walk estimates.
The first step is to approximate indicator functions in terms of explicit polynomials in Section \ref{sect: 1 to D}. 
The relations between the partial sums and the random walks are then established in Section \ref{sect: D to P} via Dirichlet polynomials.

\subsection{Approximation of Indicator Functions by Polynomials}
\label{sect: 1 to D}
First, we state a slight modification of \cite[Lemma 6]{ArgBouRad2020} that is more convenient when working with lower bounds.  Throughout the paper, the normalization for the Fourier transform is 
$$
\widehat{f}(u)=\int_{\mathbb{R}} e^{-\ii 2\pi u x} f(x)\rd x.
$$

\begin{lemma} \label{le:harmo}
  There exists an absolute constant $C > 0$ such that for any $\Delta, A \geq 3 $, there exist entire functions $G^-_{\Delta,A}$ and $G^+_{\Delta,A}(x) \in L^2(\mathbb{R})$ such that:
  \begin{enumerate}
  \item The Fourier transforms $\widehat{G^\pm_{\Delta,A}}$ are supported on $[-\Delta^{2A}, \Delta^{2A}]$.
  \item We have,
    $
    0 \leq G^-_{\Delta, A}(x)\leq  G^+_{\Delta, A}(x)\leq 1
    $
    for all $x \in \mathbb{R}$.
  \item We have
    $$
    \begin{aligned}
    \mathbf{1}(x \in [0, \Delta^{-1}]) &\leq G^+_{\Delta, A}(x) \cdot (1 + C e^{-\Delta^{A - 1}}), \\
      \mathbf{1}(x \in [0, \Delta^{-1}]) &\geq G^-_{\Delta, A}(x)-C e^{-\Delta^{A - 1}}.
      \end{aligned}
    $$
  \item We have
    $$
    \begin{aligned}
    G^+_{\Delta, A}(x) &\leq \mathbf{1}(x \in [-\Delta^{-A/2} , \Delta^{-1} + \Delta^{-A/2}]) + C e^{-\Delta^{A - 1}},\\
    G^-_{\Delta, A}(x) &\geq \mathbf{1}(x \in [\Delta^{-A/2} , \Delta^{-1} - \Delta^{-A/2}]) \cdot (1  - C e^{-\Delta^{A - 1}}).
    \end{aligned}
    $$
\item We have
  $
  \int_{\mathbb{R}} |\widehat{G^\pm_{\Delta, A}}(x)| \rd x \leq 2\Delta^{2A}.
  $
  \end{enumerate}
  \end{lemma}

  \begin{proof}
  This is proved the same way as  \cite[Lemma 6]{ArgBouRad2020} with
  $$
  G^-_{\Delta,A}(x)=\int_{\Delta^{- A/2}-\Delta^{-A}}^{\Delta^{-1}-\Delta^{-A/2}+\Delta^{-A}} \Delta^{2A}F(\Delta^{2A}(x-t))\rd t
  $$
  and
    $$
  G^+_{\Delta,A}(x)=\int_{-\Delta^{-A}}^{\Delta^{-1}+\Delta^{-A}} \Delta^{2A}F(\Delta^{2A}(x-t))\rd t
  $$
   with the approximate identity $F=F_0/\|F_0\|_1$, where the existence of $F_0$ is given by the following lemma.\end{proof}

   \begin{lemma}{\cite[Lemma 5]{ArgBouRad2020}}
    \label{le:ingham}
  There exists a smooth function $F_0$ such that 
  \begin{enumerate}
  \item For all $x \in \mathbb{R}$, we have $0 \leq F_0(x) \leq 1$ and $\widehat{F}_0(x) \geq 0$. 
  \item $\widehat{F}_0$ is compactly supported on $[-1,1]$.
  \item Uniformly in $x \in \mathbb{R}$, we have
    $$
    F_0(x) \ll e^{-|x| / \log^2 (|x| + 10)}.
    $$
  \end{enumerate}
\end{lemma}

With Lemma \ref{le:harmo}, we get the following estimate of indicator functions expressed in terms of polynomials.
\begin{lemma}
\label{lem: DP}
Let $ A\geq 3$ and $\Delta$ large enough. There exist polynomials $ \mathcal D^-_{\Delta, A}(x)$ and $ \mathcal D^+_{\Delta, A}(x)$  of degree at most $\Delta^{10A}$ with $\ell$-th coefficient bounded by $2\Delta^{2A(\ell+1)}$
such that for all $|x|\leq \Delta^{6A}$
 \begin{equation}
\label{eqn: D+}
\begin{aligned}
\1(x \in [0,\Delta^{-1}])&\leq  (1 +C e^{-\Delta^{A - 1}}) |\mathcal D^+_{\Delta, A}(x)|^2\\
 |\mathcal D^+_{\Delta, A}(x)|^2&\leq \mathbf{1}(x \in [-\Delta^{-A/2}, \Delta^{-1}+\Delta^{-A/2}]) +Ce^{-\Delta^{A-1}},
\end{aligned}
\end{equation}
and
 \begin{equation}
 \label{eqn: D-}
 \begin{aligned}
\1(x \in [0,\Delta^{-1}])&\geq |\mathcal D^-_{\Delta, A}(x)|^2-Ce^{-\Delta^{A-1}}\\
  |\mathcal D^-_{\Delta, A}(x)|^2&\geq (1 -C e^{-\Delta^{A - 1}})  \mathbf{1}(x \in [\Delta^{-A/2}, \Delta^{-1}-\Delta^{-A/2}]),
  \end{aligned}
 \end{equation}
   for some absolute constant $C>0$.
\end{lemma}
\begin{proof}
We prove the inequalities \eqref{eqn: D-} for $\mathcal D^-_{\Delta, A}$. The ones for $\mathcal D^+_{\Delta, A}(x)$ were proved in \cite{ArgBouRad2020} using the function $G^+_{\Delta, A}$, cf.~Equations (32), (33) and (41), (42) there.
The treatment is very similar to the one below.

For the first inequality in \eqref{eqn: D-}, item (3) of Lemma \ref{le:harmo} ensures the existence of a function $G^-_{\Delta, A}(x)$ in $L^2$ such that
\begin{equation}
\label{eqn: 1 to H}
\1(x\in [0,\Delta^{-1}])\geq G^-_{\Delta, A}(x)-Ce^{-\Delta^{A-1}}.
\end{equation}
For $\nu=\Delta^{10A}$, we write $G^-_{\Delta,A}(x)$ as 
\begin{equation}
\label{eqn: H to D}
G^-_{\Delta,A}(x)= \int_{\mathbb R} e^{2\pi \ii \xi x} \widehat{G^-_{\Delta,A}}(\xi){\rm d} \xi=  \mathcal D^-_{\Delta, A}(x) + \sum_{\ell >\nu} \frac{(2\pi \ii x)^{\ell}}{\ell!} \int_{\mathbb{R}} \xi^{\ell} \widehat{G^-_{\Delta, A}}(\xi) {\rm d}\xi,
\end{equation}
where
   \begin{equation}
     \label{eqn: D}
     \mathcal D^-_{\Delta, A}(x)= \sum_{\ell \leq \nu} \frac{(2\pi \ii x)^{\ell}}{\ell!} \int_{\mathbb{R}} \xi^{\ell} \widehat{G^-_{\Delta, A}}(\xi) {\rm d}\xi .
   \end{equation}
Clearly,  the degree of $ \mathcal D^-_{\Delta, A}$ is $\nu=\Delta^{10A}$,and 
\begin{equation}
\label{eqn: bound G}
\int_{\mathbb{R}} |\xi|^{\ell} |\widehat{G^-_{\Delta,A}}(\xi)| {\rm d} \xi \leq \Delta^{2A \ell}\int_{\mathbb{R}} |\widehat{G^-_{\Delta,A}}(\xi)| {\rm d} \xi \leq 2 \Delta^{2A (\ell+1)},
\end{equation}
by properties {\it (1)} and {\it (5)} of Lemma \ref{le:harmo}. Thus, the coefficients of $\mathcal D^-_{\Delta, A}(x)$ are bounded by $\ll \Delta^{2A(\ell+1)}$.

Assuming that $|x|\leq \Delta^{6A}$, then the error term in Equation \eqref{eqn: H to D} is smaller than
  \begin{equation}
  \label{eqn: error lemma4}
  \frac{(2\pi)^{\nu}}{\nu!} |x|^{\nu} \int_{\mathbb{R}} |\xi^{\nu}| |\widehat{G^-_{\Delta, A}}(\xi)| {\rm d}\xi \leq \frac{10^{\nu}}{\nu !} \Delta^{6A\nu}\,\Delta^{2A(\nu + 1)} \leq \frac{10^{\nu}}{\nu !} \Delta^{9A\nu}.
  \end{equation}
 This is $\leq e^{-\Delta^A}$ for the choice $\nu=\Delta^{10A}$.
 This shows that whenever $|x|\leq \Delta^{6A}$
\begin{equation}
\label{eqn: H D}
G^-_{\Delta, A}(x)=\mathcal D^-_{\Delta, A}(x)+\OO^*(e^{-\Delta^A}),
\end{equation}
where the $\OO^\star$ means that the implicit constant is smaller than $1$. 
If $x\notin [0,\Delta^{-1}]$, then Equations \eqref{eqn: 1 to H} and \eqref{eqn: H D} with the fact that $G^-_{\Delta, A}\geq 0$ imply
$$
-e^{-\Delta^A}\leq \mathcal D^-_{\Delta, A}(x)\leq 2 C e^{-\Delta^{A-1}},
$$
for $\Delta$ large enough (depending on $C$).
Therefore, in this case, the following holds
$$
\begin{aligned}
\1(x\in [0,\Delta^{-1}])&\geq  |\mathcal D^-_{\Delta, A}(x)|^2+(\mathcal D^-_{\Delta, A}(x)-|\mathcal D^-_{\Delta, A}(x)|^2)-2Ce^{-\Delta^{A-1}}\\
&\geq  |\mathcal D^-_{\Delta, A}(x)|^2-2|\mathcal D^-_{\Delta, A}(x)|-2Ce^{-\Delta^{A-1}}\\
&\geq  |\mathcal D^-_{\Delta, A}(x)|^2-6Ce^{-\Delta^{A-1}}.
\end{aligned}
$$
If $x\in [0,\Delta^{-1}]$, then the fact that $G^-_{\Delta, A}\leq 1$ implies instead.
$$
-e^{-\Delta^A}\leq \mathcal D^-_{\Delta, A}(x)\leq 1+2 C e^{-\Delta^{A-1}}.
$$
We deduce that:
$$
\begin{aligned}
\1(x\in [0,\Delta^{-1}])&\geq  |\mathcal D^-_{\Delta, A}(x)|^2+(\mathcal D^-_{\Delta, A}(x)-|\mathcal D^-_{\Delta, A}(x)|^2)-2Ce^{-\Delta^{A-1}}\\
&=  |\mathcal D^-_{\Delta, A}(x)|^2+|\mathcal D^-_{\Delta, A}(x)|\Big(\text{sgn}\mathcal D^-_{\Delta, A}(x) -|\mathcal D^-_{\Delta, A}(x)|\Big)-2Ce^{-\Delta^{A-1}}\\
&\geq  |\mathcal D^-_{\Delta, A}(x)|^2-6Ce^{-\Delta^{A-1}}.
\end{aligned}
$$
This establishes the first inequality in \eqref{eqn: D-} by redefining $C$.

For the second inequality in \eqref{eqn: D-}, item (4) of Lemma \ref{le:harmo} and Equation \eqref{eqn: H D} give
$$
  |\mathcal D^-_{\Delta, A}(x)+\OO^*(e^{-\Delta^{A-1}})|^2\geq (1 -C e^{-\Delta^{A - 1}})\cdot \mathbf{1}(x \in [\Delta^{-A/2}, \Delta^{-1}-\Delta^{-A/2}]).
 $$
Since the constant in $\OO^*$ is $\leq 1$, the dominant term on the left-hand side is $\mathcal D^-_{\Delta, A}(x)$, and we can absorb the additive error in a multiplicative factor to get
the second inequality in \eqref{eqn: D-}.
 \end{proof}

 \subsection{Proof of Propositions \ref{prop: 1 point} and \ref{prop: 2 points}.} \label{sect: D to P}
For these proofs, we need two preliminary steps. First, the constraints for the random walk $(S_k)_k$ (\ref{eqn: G}) are re-expressed in terms of its increments. 
Second, this allows to write the probabilities for the Dirichlet sums $S_k$ in terms of a probabilistic model. \\

\noindent{\it Constraints and increments.}\  First,  the polynomial approximation of indicator functions from Lemma \ref{lem: DP} will be related to events involving the partial sums $S_k$, $n_0<k\leq n_\mathcal L$. Fix $h\in [-1,1]$.
 Consider the increments
 $$
 Y_j(h)=S_j(h)-S_{j-1}(h), \quad n_0< j\leq n_\mathcal L.
 $$
 To shorten the notation, we consider the set of times
 \begin{equation}
 \label{eqn: J}
 \mathcal J_{\ell}=\{n_0+1, n_0+2, \dots, \ell -1 , \ell\}.
 \end{equation}
 with $n_0 \leq \ell \leq n_{\mathcal{L}}$. 
 We will partition the intervals of values taken by $Y_j$, $j\in \mathcal J_{\ell}$ into sub-intervals of length  $\Delta_j^{-1}$ where 
 $$
 \Delta_j=(j \wedge (n-j))^4.
 $$
 The exponent $4$ is chosen to ensure summability.  In particular we will simply use that for $y$ chosen large enough we have
\begin{equation}
\label{eqn:step}
 \sum_{j \in \mathcal{J}_{\ell}} \Delta_j^{-1} \leq  \sum_{j\geq n_0} \Delta_j^{-1} \leq 1.
\end{equation}
We consider events for the partial sums of the form
 $$
\{S_{j}(h)\in [L_j, U_j], j\in \mathcal J_{\ell}\}, \quad h\in [-1,1].
 $$
 We would like to decompose the above in terms of events for the increments
 \begin{equation}
 \label{eqn: Y}
 \{Y_j(h)\in [u_j,u_j + \Delta_j^{-1}],j\in \mathcal J_{\ell}\},\quad h\in [-1,1],
 \end{equation}
 for a given tuple $(u_{j}, j\in \mathcal J_{\ell})$. Note that such events are disjoint for two distinct tuples. 
On an event of the form \eqref{eqn: Y},  from (\ref{eqn:step}) we have
\begin{equation}
\label{eqn: u S}
\begin{aligned}
\sum_{i\leq j} u_i \leq S_{j}(h) &\leq \sum_{i\leq j} (u_i +\Delta_i^{-1})&\leq  \sum_{i\leq j} u_i +1, \quad \text{ for all $j\in \mathcal J_{\ell}$,}
\end{aligned}
\end{equation}
This means that we have the following inclusions
 \begin{align}
  \{S_{j}(h)\in [L_j, U_j+1], j\in \mathcal J_{\ell}\}&\supset \bigcup_{\mathbf u \in \mathcal I} \{Y_j(h)\in [u_j,u_j + \Delta_j^{-1}], j\in \mathcal J_{\ell}\} \label{eqn: inclusion LB},\\
\{S_{j}(h)\in [L_j+1, U_j], j\in \mathcal J_{\ell}\}&\subset \bigcup_{\mathbf u \in \mathcal I}  \{Y_j(h)\in [u_j,u_j + \Delta_j^{-1}], j\in \mathcal J_{\ell}\}\label{eqn: inclusion UB},
  \end{align}
  where $\mathcal I$ is the set of tuples $\mathbf u=(u_j, j\in \mathcal J_{\ell})$, $u_j\in \Delta_j^{-1}\mathbb Z$, such that $\sum_{i\leq j}u_i\in [L_i, U_i]$ for all $j\in \mathcal J_{\ell}$.
The definition of $\mathcal I$ imposes restrictions on the $u_j$'s. 
Indeed, we must have
$$
u_j\leq U_j-L_{j-1} \leq 10\Delta_j^{1/4} \qquad u_j \geq L_j-U_{j-1}\geq -10\Delta_j^{1/4}.
$$
In all cases, we  have the following bound which will be repeatedly used:
\begin{equation}
\label{eqn: u bound}
|u_j|\leq 100 \Delta_j^{1/4}, \quad j\in \mathcal J_{\ell}. 
\end{equation}
 
\noindent {\it Probabilistic model for the increments.} Additionally to the original random walk (\ref{eqn: Sk}) and its Gaussian counterpart (\ref{eqn: Gk}), as an intermediate we now consider another probabilistic model needed for the proofs of Propositions \ref{prop: 1 point} and \ref{prop: 2 points}.
For $h\in[-1,1]$,  let
\begin{equation}
\label{eqn: steinhaus walk}
\mathcal S_{k}(h) = \sum_{n_0\leq \log\log p\leq  k} \re \Big ( e^{\ii \theta_p} \, p^{-(1/2 +\ii h)} + \tfrac 12 \,  e^{2\ii \theta_p} \, p^{-(1 + 2 \ii h)} \Big ),\quad k\leq n_{\mathcal L},
\end{equation}
where $(\theta_p, p \text{ prime})$ are i.i.d.  random variables distributed uniformly on $[0,2\pi]$, and define the corresponding increments
\begin{equation}
\label{eqn: Y_k}
\mathcal Y_k(h)=\mathcal S_{k}(h)-\mathcal S_{k-1}(h), \quad k\leq n_{\mathcal L}. 
\end{equation}
It is easy to see that $\mathcal S_k$ and $\mathcal Y_k$ have mean $0$. The variance of the increments $\mathcal Y_k$ coincides with (\ref{eqn:sk}) and by a quantitative version of the Prime Number Theorem (see \cite[Equation (74)]{ArgBouRad2020}) they satisfy
\begin{equation}\label{eqn:skasymp}
\mathfrak{s}_j^2=\frac{1}{2}+{\rm O}(e^{-{c \sqrt{j}}}).
\end{equation}
for some universal $c>0$.  These precise asymptotics are not used in the comparison with the Gaussian model,  i.e.  in the proof of Proposition \ref{prop: 1 point} below,  and they will be used only for convenience in the first and second moment for the Gaussian model, Proposition \ref{prop: comp second}. In fact to apply the Ballot theorem from Proposition \ref{prop:barrier} we will only rely on $\mathfrak{s}_j^2\in[\kappa,\kappa^{-1}]$ for some fixed $\kappa>0$.

 \begin{proof}[Proof of Proposition \ref{prop: 1 point}] 
   We prove \eqref{eqn: LB P}. The upper bound \eqref{eqn: UB P} is proved in a similar way, see Proposition \ref{lem: 2 points}. We define the weighted expectation, 
   $$
   \mathbb{E}_{\mathcal{Q}}[X] := \mathbb{E} \Big [ |\mathcal{Q}(\tfrac 12 + \ii  \tau)|^2 \cdot X(\tau) \Big ] \cdot \mathbb{E} \Big [ |\mathcal{Q}(\tfrac 12 + \ii  \tau)|^2 \Big ]^{-1},
   $$
   and the corresponding measure $\mathbb{P}_{\mathcal{Q}}(A) := \mathbb{E}_{\mathcal{Q}}[\mathbf{1}(\tau \in A)]$.
   
In what follows, we drop the dependence on $h$ as it plays no role. 
Equation \eqref{eqn: inclusion LB} directly implies (by taking $U_k$ instead of $U_{k}+\ee$):
 \begin{equation}
 \label{eqn: sum over u}
  \mathbb{P}_{\mathcal{Q}}(S_{k}\in [L_k, U_k],k\in \mathcal J_{\ell})\geq \sum_{\mathbf u \in \mathcal I} \mathbb{P}_{\mathcal{Q}}(Y_k-u_k\in [0,\Delta_k^{-1}], k\in \mathcal J_{\ell}),
 \end{equation}
 where $\mathcal I$ is now the set of tuples $\mathbf u=(u_j, j\in \mathcal J_{\ell})$, $u_j\in \Delta_j^{-1}\mathbb Z$, such that $\sum_{i\leq j}u_i\in [L_i, U_i-1]$ for all $j\in \mathcal J_{\ell}$.
 By introducing the indicator functions $\prod_k \1(|Y_k-u_k|\leq \Delta_k^{6A})$, Equation \eqref{eqn: D-} of Lemma \ref{lem: DP} can be applied with $A=10$ (say), thanks to the bound \eqref{eqn: u bound}.
 This yields
 \begin{equation}
 \label{eqn: P to D}
 \mathbb{P}_{\mathcal{Q}}(Y_k-u_k\in [0,\Delta_k^{-1}],k\in \mathcal J_{\ell})\geq \mathbb{E}_{\mathcal{Q}} \Big[\prod_{k}\Big(|\mathcal D^-_{\Delta_k,A}(Y_k-u_k)|^2-Ce^{-\Delta_k^{A-1}}\Big)\1(|Y_k-u_k|\leq \Delta_k^{6A})\Big].
 \end{equation}
 The tricky part is to get rid of the indicator function. 
 For simplicity, let's write $\mathcal D_k$ for $|\mathcal D^-_{\Delta_k,A}(Y_k-u_k)|^2-Ce^{-\Delta_k^{A-1}}$. 
Since $\1(|Y_k-u_k|\leq \Delta_k^{6A})=1-\1(|Y_k-u_k|> \Delta_k^{6A})$, we can rewrite the above as
\begin{equation}
\label{eqn: indicator D}
 \mathbb{E}_{\mathcal{Q}} \Big[\prod_{k\in  \mathcal J_{\ell}} \mathcal D_k\Big]
 +\sum_{J\subseteq \mathcal J_{\ell}, J\neq \emptyset}(-1)^{|J|}\mathbb{E}_{\mathcal{Q}} \Big[\prod_{k\in  \mathcal J_{\ell}} \mathcal D_k\prod_{j\in J}\1(|Y_j-u_j|>\Delta_j^{6A})\Big].
\end{equation}
We start with the first term, which will be dominant.
Each $Y_j$ is a Dirichlet polynomial of length at most $\exp(2e^{j})$. Therefore,
from Lemma \ref{lem: DP},  for any subset $\mathcal M\subset \mathcal J_\ell$ the Dirichlet polynomial $\prod_{j\in \mathcal M}\mathcal D^-_{\Delta_j,A}(Y_j-u_j)$ is of length at most 
\begin{equation}
\label{eqn: length D}
\exp(2e^{\ttL}\Delta_{\ttL}^{100})
\leq \exp \Big ( \frac{1}{100} e^n \Big )
\end{equation}
for $y$ large enough.
Therefore,  Lemma \ref{lem: Transition}  applies to compare with the random model with increments $\mathcal Y_k$ given in \eqref{eqn: Y_k}:
\begin{equation}
\label{eqn: D to random}
 \mathbb{E}_{\mathcal{Q}} \Big[\prod_{k\in  \mathcal M} |\mathcal D^-_{\Delta_k, A}(Y_k-u_k)|^2\Big]
 = (1+\OO(T^{-99/100}))\prod_{k\in  \mathcal M} \mathbb{E} \Big [ |\mathcal{D}^-_{\Delta_k,A}(\mathcal Y_k - u_k)|^2 \Big ],
 \end{equation}
where we have split the expectation $\E_{\mathcal Q}$ and used $\E=\E_{\mathcal Q}$ for the probabilistic model, thanks to the independence of the $\mathcal Y_k$'s.
 Moreover,   for each $k$, we have
\begin{equation}
\label{eqn: D LB}
\begin{aligned}
 \E\Big[ |\mathcal D^-_{\Delta_k,A}(\mathcal Y_k-u_k)|^2\Big]&\geq  \E\Big[ |\mathcal D^-_{\Delta_k,A}(\mathcal Y_k-u_k)|^2\1(|\mathcal Y_k-u_k|\leq\Delta_k^{6A})\Big]\\
 &\geq (1-Ce^{-\Delta_k^{A-1}})\cdot \PP(\mathcal Y_k-u_k \in [\Delta_k^{-A/2}, \Delta_k^{-1}-\Delta_k^{-A/2}]),
 \end{aligned}
\end{equation}
where the second inequality follows from  \eqref{eqn: D-},
noting that the condition $|\mathcal Y_k-u_k|\leq\Delta_k^{6A}$ is implied by $\mathcal Y_k-u_k \in [\Delta_k^{-A/2}, \Delta_k^{-1}-\Delta_k^{-A/2}]$, and thus can be dropped.
We now rewrite this probability in terms of Gaussian increments.
Lemma \ref{lem:compProb} in Appendix \ref{appendixA} gives
\begin{equation}
\label{eqn: Y_k estimate}
\PP(\mathcal Y_k-u_k \in [\Delta_k^{-A/2}, \Delta_k^{-1}-\Delta_k^{-A/2}])=\PP(\mathcal N_k-u_k\in [\Delta_k^{-A/2},\Delta_k^{-1}-\Delta_k^{-A/2}])+\OO(e^{-ce^{k/2}}).
\end{equation}
The overspill $\Delta_j^{-A/2}$ can be removed at no cost:  from  \eqref{eqn: u bound} and  $\mathfrak{s}_k\asymp 1$,  uniformly in $x,y\in u_k+[\Delta_k^{-A/2},\Delta_k^{-1}-\Delta_k^{-A/2}]$ the density $f_k$ of $\mathcal{N}_k$ satisfies $f_k(x)\asymp f_k(y)$, so
\begin{equation}
\label{eqn: dropping delta}
 \PP(\mathcal N_k -u_k\in [\Delta_k^{-A/2},\Delta_k^{-1}-\Delta_k^{1-A/2}]) =(1+\OO(\Delta_k^{-A/2}))\cdot \PP(\mathcal N_k-u_k\in [0,\Delta_k^{-1}]).
\end{equation}
Moreover, 
\begin{equation}
\label{eqn: LB prob}
\PP(\mathcal N_k-u_k\in [0,\Delta_k^{-1}])\gg \Delta_k^{-1}e^{-2 u_k^2}\gg \Delta_k^{-1}e^{-100^2\Delta_k^{1/2}}.
\end{equation}
This is much larger than the additive error term $\OO(e^{-ce^{k/2}})$ in \eqref{eqn: Y_k estimate},  which can therefore be replaced by a multiplicative error. Both multiplicative errors together give
for $k\leq n_{\mathcal L}$
\begin{equation}\label{eqn:GaussianDens}
\PP(\mathcal Y_k-u_k \in [\Delta_k^{-A/2}, \Delta_k^{-1}-\Delta_k^{-A/2}])=\Big(1+\OO \Big ((k\wedge (n-k))^{-2A} \Big ) \cdot \PP(\mathcal N_k-u_k\in [0,\Delta_k^{-1}]).
\end{equation}
The product over $k\in \mathcal J_{\ell}$ of the error terms above is $(1+\OO(n_0^{-A}))$.
Going back to Equations \eqref{eqn: D to random} and \eqref{eqn: D LB}, we have established that
\begin{equation}
\label{eqn: prod random}
\E_{\mathcal Q}\Big[\prod_{k\in\mathcal M}|\mathcal D_{\Delta_k, A}^-(Y_k-u_k)|^2\Big]\geq  (1+\OO(n_0^{-A}))\prod_{k\in\mathcal M}  \PP(\mathcal N_k-u_k\in [0,\Delta_k^{-1}]).
\end{equation}
Remember that we aim at a similar estimate for $\mathcal D_k=|\mathcal D^-_{\Delta_k, A}(Y_j-u_j)|^2-Ce^{-\Delta_k^{A-1}}$.  From \eqref{eqn: LB prob},  $\PP(\mathcal N_k-u_k\in [0,\Delta_k^{-1}])\gg e^{-\Delta_k}$ and (\ref{eqn: prod random}) holds for arbitrary $\mathcal M\subset \mathcal J_\ell$, so that by a simple expansion we have 
\begin{equation}
\label{eqn: D2}
\E_{\mathcal Q}\Big[\prod_{k\in \mathcal J_{\ell}}\mathcal D_{k}\Big]\geq  (1+\OO(n_0^{-A}))\prod_{k\in \mathcal J_{\ell}}  \PP(\mathcal N_k-u_k\in [0,\Delta_k^{-1}]).
\end{equation}
We now bound the second term in \eqref{eqn: indicator D}. Let's fix the non-empty subset $J\subseteq \mathcal J_{\ell}$ in the sum.
Since 
$\1(|X|>\lambda)\leq \frac{|X|^{2q}}{\lambda^{2q}},$
 we have
\begin{equation}
\label{eqn: D 1}
\mathbb{E}_{\mathcal{Q}} \Big[\prod_{k\in  \mathcal J_{\ell}} \mathcal D_k\prod_{j\in J}\1(|Y_j-u_j|>\Delta_j^{6A})\Big]\leq \mathbb{E}_{\mathcal{Q}} \Big[\prod_{k\in \mathcal J_{\ell}}\mathcal D_k\prod_{j\in J} \frac{|Y_j-u_j|^{2q_j}}{\Delta_j^{12Aq_j}}\Big],
\end{equation}
where we pick $q_j=\lfloor\Delta_j^{6A}\rfloor$, $A=10$.
As for the first term, we need to handle the error $Ce^{-\Delta_k^{A-1}}$ in $\mathcal D_k$.
For this we abbreviate $d_k(x)= D^-_{\Delta_k,A}(x-u_k)$, $\varepsilon_k=Ce^{-\Delta_k^{A-1}}$, and expand
\begin{equation}\label{eqn:intermed}
\mathbb{E}_{\mathcal{Q}} \Big[\prod_{k\in \mathcal J_{\ell}}\mathcal D_k\prod_{j\in J} \frac{|Y_j-u_j|^{2q_j}}{\Delta_j^{12Aq_j}}\Big]\leq\sum_{B\subset \mathcal{J}_\ell}\mathbb{E}_{\mathcal{Q}} \Big[\prod_{k\in B}|d_k(Y_k)|^2\prod_{k\in\mathcal{J}_{\ell}\setminus B}\varepsilon_k\prod_{j\in J} \frac{|Y_j-u_j|^{2q_j}}{\Delta_j^{12Aq_j}}\Big].
\end{equation}
From Lemma \ref{lem: DP}, the Dirichlet polynomial $d_j$ is of length at most $\exp(2e^{j}\Delta_j^{100})$. 
The choice of $q_j$ implies that the Dirichlet polynomial
$
\prod_{k\in A} d_k \prod_{j\in J} (Y_j-u_j)^{q_j}$
has length at most $\exp(2e^{\ttL} \Delta_{\ttL}^{100})\leq \exp(\tfrac{1}{100} e^n)$ as in \eqref{eqn: length D}.
Therefore, we can use Lemma \ref{lem: Transition} again, and work with the random model term by term.
Again, the fact that $\mathcal{Q}$ is supported on integers with primes $p$ with $\log p > e^{\ell}$ means
that for the random model the expectation with respect to $\mathbb{E}_{\mathcal{Q}}$ is equal to the
expectation with respect to $\mathbb{E}$.  
We start with the case $j\in B\cap J$. We have
\begin{equation}\label{eqn:sch}
\E\Big[|d_j(\mathcal Y_j)|^2|\mathcal Y_j-u_j|^{2q_j}\Big]\ll \E \Big [|d_j(\mathcal Y_j)|^4 \Big ]^{1/2}\cdot \E\Big[|\mathcal Y_j-u_j|^{4q_j}\Big]^{1/2}.
\end{equation}
The definition of $\mathcal D^-_{\Delta_j, A}$ in Equations \eqref{eqn: D} and \eqref{eqn: bound G} implies the following bound on all $2k$-moments, $k\in \mathbb N$,
\begin{multline}
\label{eqn: 4th D bound}
 \E[|d_j(\mathcal Y_j)|^{2k}]
  \leq \E\Big[ \Big ( \sum_{\ell \leq \Delta_j^{10 A}} \frac{(2\pi )^{\ell}}{\ell!}  2\Delta_j^{2A(\ell+1)}  (|\mathcal Y_j|+100\Delta_j^{1/4})^{\ell} \Big )^{2k}\Big ]  \\
  \ll  \Delta_j^{4k A} \, \E[\exp( 4\pi k \Delta_j^{2A}  (| \mathcal Y_j|+100 \Delta_j^{1/4}))] \ll_k e^{\Delta_j^{5A}},
\end{multline}
where the third inequality follows from Lemma \ref{lem:basicFour}. By Lemma \ref{lem:basicFour} and the inequality $\frac{x^{4q}}{q^{4q}}\leq \frac{(4q)!}{(\lambda q)^{4q}}\cdot (e^{\lambda x}+e^{-\lambda x})$
with the choice $q=q_j=\lfloor\Delta_j^{6A}\rfloor$,  $\lambda=10$,  we have  for any $j\leq n_\mathcal L$
\begin{equation}
\label{eqn: Y moment}
\E\Big[\frac{|\mathcal Y_j-u_j|^{4q_j}}{\Delta_j^{24Aq_j}}\Big]
\ll e^{-2\Delta_j^{6A}},
\end{equation}
by Stirling's formula and the fact that $|u_j|\leq 100 \Delta_j^{1/4}$.
From equations (\ref{eqn: D 1}) (\ref{eqn:intermed}) (\ref{eqn:sch}) \eqref{eqn: 4th D bound} and \eqref{eqn: Y moment} we have proved
\begin{multline*}
\mathbb{E}_{\mathcal{Q}} \Big[\prod_{k\in  \mathcal J_{\ell}} \mathcal D_k\prod_{j\in J}\1(|Y_j-u_j|>\Delta_j^{6A})\Big]\ll \sum_{B\subset\mathcal{J}_\ell}\prod_J e^{-\Delta_j^{6A}}\prod_{B\backslash J}\E[|d_j({\mathcal Y}_j)|^2]\prod_{\mathcal{J}_\ell\backslash B}\varepsilon_j\\
=\prod_{j\in {\mathcal J}_\ell}\E[|d_j({\mathcal Y}_j)|^2] \sum_{B\subset\mathcal{J}_\ell}\prod_J \frac{e^{-\Delta_j^{6A}}}{\E[|d_j({\mathcal Y}_j)|^2] }\prod_{\mathcal{J}_\ell\backslash (B\cup J)}\frac{\varepsilon_j}{\E[|d_j({\mathcal Y}_j)|^2] }\prod_{J\backslash B}\varepsilon_j.
\end{multline*}
Moreover, from  \eqref{eqn: D LB} with the estimates \eqref{eqn: LB prob}, \eqref{eqn:GaussianDens},  we have 
$\E[|d_j({\mathcal Y}_j)|^2] \gg \Delta_j^{-1}e^{-100^2\Delta_j^{1/2}}$. We have obtained
\begin{multline*}
\left|\sum_{J\subseteq \mathcal J_{\ell}, J\neq \emptyset}(-1)^{|J|}\mathbb{E}_{\mathcal{Q}} \Big[\prod_{k\in  \mathcal J_{\ell}} \mathcal D_k\prod_{j\in J}\1(|Y_j-u_j|>\Delta_j^{6A})\Big]\right| \\ \ll
\prod_{j\in {\mathcal J}_\ell}\E[|d_j({\mathcal Y}_j)|^2]\sum_{J\subset\mathcal{J}_\ell,J\neq\emptyset} \sum_{B\subset\mathcal{J}_\ell}\prod_J e^{-\tfrac{1}{2}\Delta_j^{6A}}\prod_{\mathcal{J}_\ell\backslash B}\varepsilon_j^{1/2}\\=\prod_{j\in {\mathcal J}_\ell}\E[|d_j({\mathcal Y}_j)|^2]\big(\prod_{\mathcal{J}_\ell}(1+e^{-\tfrac{1}{2}\Delta_j^{6A}})-1\big)\prod_{\mathcal{J}_\ell}(1+\sqrt{\varepsilon_j})\ll e^{-n_0^{100}}\prod_{j\in {\mathcal J}_\ell}\E[|d_j({\mathcal Y}_j)|^2].
\end{multline*}
The above product is $\ll \prod_{j\in \mathcal J_{\ell}}\PP(\mathcal N_j-u_j\in [0,\Delta_j^{-1}])$ as easily proved by combining \eqref{eqn: D-} and \eqref{eqn: Y moment}.  (A similar bound in the more general case of joint increments is detailed in (\ref{eqn:Dbound}).)

Equations \eqref{eqn: P to D},\eqref{eqn: indicator D} and \eqref{eqn: D2} with the above finally yield

$$
 \mathbb{P}_{\mathcal{Q}}(Y_j-u_j \in [0,\Delta_j^{-1}], j\in \mathcal J_{\ell})
 \geq (1+\OO(n_0^{-10})) \prod_{j\in \mathcal J_{\ell}}\PP(\mathcal N_j-u_j\in [0,\Delta_j^{-1}]).
$$
The claim \eqref{eqn: LB P} follows by summing over $\mathbf u\in \mathcal I$ as in Equation \eqref{eqn: sum over u}, and by applying the inclusion \eqref{eqn: inclusion UB} for the Gaussian random walk with increments $\mathcal N_j$.
 \end{proof}

For the proof Proposition \ref{prop: 2 points} below, we will also consider the partial sums  at $h$ and $h'$ jointly,  i.e., $S_k(h)$ and $S_k(h')$,  $n_0<k\leq n_\mathcal L$, as well as the joint increments  $\mathcal Y_j(h)$ and $\mathcal Y_j(h')$. These increments  have 
covariance and correlations identical to those of $\mathcal{N}_{j}$ and $ \mathcal{N}_j'$, i.e.,   they are given by  \eqref{eqn:rhok}, which satisfies the asymptotics
\begin{equation}
\label{eqn: cov estimate}
\begin{aligned}
\rho_j&=
\begin{cases}
\mathfrak{s}_j^2+\OO((e^j|h-h'|)^2) \ & \text{if $j\leq \log |h-h'|^{-1}$},\\
\OO((e^j|h-h'|)^{-1}) \ & \text{if $j\geq \log |h-h'|^{-1}$},
\end{cases}
\end{aligned}
\end{equation}
as is easily proved using the Prime Number Theorem as in \cite[Lemma 2.1]{ArgBelHar2017}.
We also define $\varepsilon_j=\varepsilon_j(h,h')$ by 
\begin{align} \label{eq:epsilonjdef}
\rho_j = \begin{cases}
  \mathfrak{s}_j^2 - \varepsilon_j & \text{ if } j \leq \log |h - h'|^{-1}, \\
  \varepsilon_j & \text{ if } j > \log |h - h'|^{-1}. 
  \end{cases}
\end{align}
The precise asymptotics of the covariances in \ref{eq:epsilonjdef} will not play a role in  the  proof of  Proposition \ref{prop: 2 points} below.
However, it will be crucial in the proof of Proposition \ref{prop: comp second}.

 \begin{proof}[Proof of Proposition \ref{lem: 2 points}] 
We write $(S_k,S_k')$ for $(S_k(h),S_k(h'))$ for conciseness, and similarly for the increments. 
 The event on the left-hand side of \eqref{eqn: UB P} is decomposed using the increments as in Equation \eqref{eqn: inclusion UB}.
Then, Equation \eqref{eqn: D+} can be used to bound the indicator functions for both points. We take $A=10$ (say). This gives that the left-hand side of \eqref{eqn: UB P 2} is
 \begin{equation}
\label{eqn: eqn1 2points}
\leq (1+\OO(e^{-n_0^{10}}))\sum_{\mathbf u, \mathbf u'\in \mathcal I}\E\Big[\prod_{j\in \mathcal J_{\mathcal{L}}} |\mathcal D^+_{\Delta_j, A}(Y_j-u_j) \mathcal D^+_{\Delta_j, A}(Y'_j-u'_j)|^2\Big],
\end{equation}
where we write $\mathcal J_{\ell}$ as in \eqref{eqn: J}. 
We proceed as in Equation \eqref{eqn: D to random}.
From Lemma \ref{lem: DP}, the Dirichlet polynomial $\prod_j\mathcal D^+_{\Delta_j,A}(Y_j-u_j)$ is of length at most $\exp(2e^{n_\mathcal L}\Delta_{n_\mathcal L}^{100})$. So the product of the polynomials for $h$ and $h'$ has length smaller than $\exp(4e^{n_\mathcal L}\Delta_{n_\mathcal L}^{100})\leq T^{1/100}$,  as in \eqref{eqn: length D}.
Lemma \ref{lem: Transition} then implies
\begin{equation}
\label{eqn: eqn2 2points}
\begin{aligned}
&\E\Big[ \prod_{j\in \mathcal J_{\mathcal{L}}}|\mathcal D^+_{\Delta_j, A}(Y_j-u_j)\mathcal D^+_{\Delta_j, A}( Y'_j-u'_j)|^2\Big]\\
&=(1+\OO(T^{-99/100}))\prod_{j\in \mathcal J_{\mathcal{L}}}\E\Big[|\mathcal D^+_{\Delta_j, A}(\mathcal Y_j-u_j)\mathcal D^+_{\Delta_j, A}( \mathcal Y'_j-u'_j)|^2\Big].
\end{aligned}
\end{equation}

We estimate the expectation for each $j$. 
Write for short $\mathcal D^+_{\Delta_j, A}(\mathcal Y_j-u_j)=\mathcal D_j$ and similarly for $\mathcal D_j'$. 
We would like to introduce the indicator functions $\1(|\mathcal Y_j-u_j|\leq \Delta_j^{6A})$ and $\1(|\mathcal Y'_j-u'_j|\leq \Delta_j^{6A})$. 
For this, note first that
$$
\E\Big[|\mathcal D_j \mathcal D'_j |^2\1(|\mathcal Y_j-u_j|> \Delta_j^{6A})\Big]
\leq \E \Big [|\mathcal D_j|^6 \Big]^{1/3}\cdot \E \Big [|\mathcal D'_j|^6 \Big ]^{1/3}\cdot \PP \Big (|\mathcal Y_j-u_j|> \Delta_j^{6A} \Big )^{1/3}\ll e^{-\Delta_j^{6A}},
$$
by Equation \eqref{eqn: 4th D bound} (with $a=3$) and Markov's inequality using \eqref{eqn: Y moment}.
This observation implies that
\begin{align}
\E \Big [|\mathcal D_j \mathcal D'_j |^2 \Big ]&=\E\Big[|\mathcal D_j \mathcal D'_j |^2\1(|\mathcal Y_j-u_j|\leq \Delta_j^{6A}, |\mathcal Y'_j-u'_j| \leq \Delta_j^{6A}) \Big ] +\OO \Big (e^{-\Delta_j^{6A}} \Big )\notag\\
&\leq \PP \Big ((\mathcal Y_j-u_j, \mathcal Y_j'-u_j')\in [-\Delta_j^{-A/2},\Delta_j^{-1}+\Delta_j^{-A/2}]^{2} \Big ) +\OO \Big (e^{-\Delta_j^{A-1}} \Big ),\label{eqn:Dbound}
\end{align}
by Equation \eqref{eqn: D+} applied to both $\mathcal D_j$ and $\mathcal D_j'$. 
The Berry-Esseen approximation of Lemma \ref{lem:compProb} can now be applied:
$$
\begin{aligned}
\E\Big[|\mathcal D_j \mathcal D'_j |^2 \Big ]
&\leq (1+\Delta_j^{-A/2}) \PP \Big ((\mathcal N_j-u_j, \mathcal N_j'-u_j')\in [-\Delta_j^{-A/2},\Delta_j^{-1}+\Delta_j^{-A/2}]^{2} \Big ) +\OO \Big (e^{-\Delta_j^{A-1}} \Big ).
\end{aligned}
$$
The overspill $\Delta_j^{-A/2}$ can be also removed as in \eqref{eqn: dropping delta}.
We conclude that the above is
$$
\begin{aligned}
&=  (1+{\rm O}(\Delta_j^{-A/2}))  \PP \Big ((\mathcal N_j-u_j, \mathcal N_j'-u_j')\in [0,\Delta_j^{-1}]^{2} \Big ) +\OO \Big (e^{-\Delta_j^{A-1}} \Big )\\
&=(1+{\rm O}(\Delta_j^{-A/2}))  \PP \Big ((\mathcal N_j-u_j, \mathcal N_j'-u_j')\in [0,\Delta_j^{-1}]^{2} \Big ),
\end{aligned}
$$
since $ \PP((\mathcal N_j-u_j, \mathcal N_j'-u_j')\in [0,\Delta_j^{-1}]^{2})\gg e^{-cu_j^2 -c {u_j'}^2}\gg e^{-2 c\Delta_j^{1/2}}$ by the bound on $u_j$ and $u_j'$.
It remains to use the above bound in \eqref{eqn: eqn2 2points} and then  \eqref{eqn: eqn1 2points}. 
The claim then follows from Equation \eqref{eqn: inclusion UB} for the Gaussian random walks. 
 \end{proof}

 \section{Proof of Proposition \ref{prop:max}}

 We first need preliminary bounds on the size of $\zeta$ and Dirichlet sums. We will use the notation
\begin{equation}\label{eqn:Pn0}
  P_{n_0}(h) =\sum_{\log\log p\leq n_0} \re \Big( p^{-(1/2 + \ii \tau + \ii h)} + \frac 12 \cdot p^{- 2(1/2 + \ii \tau + \ii h)} \Big ).
\end{equation}

 \begin{lemma} \label{le:Easy}
   We have, for $1000 < y < n / 10$, 
   \begin{align}\label{eqn:41}
  & \mathbb{P} \Big (\forall m \geq 1: \max_{|u| \leq 2^{m}} |\zeta(\tfrac 12 + \ii  \tau + \ii u)| \leq 2^{2m} e^{2n_{\mathcal{L}}} \Big ) = 1 - \OO(e^{-n}),\\
\label{eqn:42}
   &\mathbb{P} \Big (\forall m \geq 1: \max_{|u| \leq 2^{m}} |S_{n_{\mathcal{L}}}(u)| \leq 2^{m/100} e^{n_{\mathcal{L}} / 100} \Big ) = 1 - \OO(e^{-n}),\\
\label{eqn:43}
   &\mathbb{P} \Big ( \forall m \geq 1: \max_{|u| \leq 2^m} | P_{n_0}(u)| \leq 2^{m / 100} \cdot 10 y \Big ) = 1 - \OO(e^{-y}). 
    \end{align}
 \end{lemma}
 \begin{proof}
   By a union bound, the probability of the complement of the first event is
   $$
   \sum_{m \geq 1} 2^{- 4 m} e^{-4n_{\mathcal{L}}} \mathbb{E} \Big [ \max_{|u| \leq 2^{m}} |\zeta(\tfrac 12 + \ii  \tau + \ii  u)|^2 \Big ] \ll \sum_{m \geq 1} 2^{-4m} e^{-4n_{\mathcal{L}}} \cdot 2^{m} e^{2n} \ll e^{-n},
   $$
   as claimed,  where  the first inequality above relies on the same subharmonicity argument as \cite[Lemma 28]{ArgBouRad2020}. For the second claim, we similarly have that the probability of the complement is bounded by, 
   $$
   \sum_{m \geq 1} 2^{-4m} e^{-4n_{\mathcal{L}}} \cdot \mathbb{E} \Big [ \max_{|u| \leq 2^{m}} |S_{n_{\mathcal{L}}}(u)|^{400} \Big ] \ll \sum_{m \geq 1} 2^{-4m} e^{-4n_{\mathcal{L}}} \cdot 2^{m} e^n n^{200} \ll e^{-2n},
   $$
   where we used \eqref{eqn:mom2} in Lemma \ref{lem: Gaussian moments}.
   Finally, the last bound is proved in exactly the same way, using that, for $v = \lfloor 100 y \rfloor$,
\begin{multline*}
   \sum_{m \geq 1} 2^{-2v m / 100} (10 y)^{-2 v} \cdot \mathbb{E} \Big [ \max_{|u| \leq 2^{m}}  | P_{n_0}(u) |^{2v} \Big ]  \\
\ll \sum_{m \geq 1} 2^{-2v m / 100} (10 y)^{-2v} \cdot 2^{m} e^{n_0} \cdot v^{1/2}\frac{(2v)!}{2^v v!} \cdot (C y)^{v} \ll e^{-90 y},
\end{multline*}
where the moments calculation is now based on \eqref{eqn:mom3} in Lemma \ref{lem: Gaussian moments}.
 \end{proof}
 
 The main analytic input is the next lemma. 

 \begin{lemma} \label{le:Analytic}
   Let $100 \leq T \leq t \leq 2T$ and $|h| \leq 1$. 
   Let $f$ be a smooth function with $\widehat{f}$ compactly supported in $[-\frac{1}{2\pi},\frac{1}{2\pi}]$ and such that $\widehat{f}(0) = 1$. Then,
   \begin{align} \label{eq:shift}
  \log X \int_{\mathbb{R}} \zeta(\tfrac 12 + \ii t + \ii h + \ii x) \prod_{p \leq X} \Big (1 - \frac{1}{p^{1/2 + \ii t + \ii h + \ii x}} \Big ) f (x \log X) \rd x = 1 + \OO(T^{-1}).
   \end{align}
 \end{lemma}
 \begin{proof}
For $z\in\mathbb{R}$,  we have $f(z)=\int_{\mathbb{R}} \widehat f(u)e^{\ii 2\pi z u }\rd u$. As $\widehat{f}$ is compactly supported, by Paley-Wiener this defines for $z\in\mathbb{C}$  an entire function of rapid (faster than polynomial) decay as $|\re\ z| \rightarrow \infty$ inside any fixed strip. We can therefore shift the contour of integration in \eqref{eq:shift} and see that it is equal to
   $$
  \log X \int_{2 - i \infty}^{2 + i \infty} \zeta(s + \ii t + \ii  h) \prod_{p \leq X} \Big ( 1 - \frac{1}{p^{s + \ii t + \ii h}} \Big ) f \Big ( \frac{s - \tfrac 12}{\ii } \cdot \log X \Big ) \frac{\rd s}{\ii} + \OO(T^{-1})
   $$
   where $T^{-1}$ is the contribution of the pole at $s = 1 - it - ih$ of $\zeta$. On the line $\re \ s = 2$ we can write pointwise
   $$
   \zeta(s + \ii t + \ii h) \prod_{p \leq X} \Big ( 1 - \frac{1}{p^{s + \ii t + \ii h}} \Big ) = 1 + \sum_{\substack{n > 1 \\ p | n \implies p > X}} \frac{1}{n^{s + \ii t + \ii h}}.
   $$
  After nterchanging the sum and integral,  the task reduces to estimating
   $$
   \frac{\log X}{\ii} \int_{2 - \ii  \infty}^{2 + \ii  \infty} n^{-s - \ii t - \ii  h} f \Big ( \frac{s - \tfrac 12}{\ii } \cdot \log X \Big ) \rd s.
   $$
   Shifting the contour back to the line $\re\ s = \tfrac 12$, this is equal to
   $$
   \log X \int_{\mathbb{R}} \frac{1}{n^{1/2 + \ii t + \ii h + \ii x}} \cdot f(x \log X) \rd x = \widehat{f} \Big ( -\frac{\log n}{2\pi\log X} \Big ) \cdot \frac{1}{n^{1/2 + \ii t + \ii h}}.
   $$
   If $n = 1$, then this is equal to $\widehat{f}(0) = 1$. On the other hand if $n \neq 1$ then $n > X$ and then by assumption $\widehat{f}(-\log n / (2\pi\log X)) = 0$. This gives the claim. 
 \end{proof}

 We are now ready to prove Proposition \ref{prop:max}.
 \begin{proof}[Proof of Proposition \ref{prop:max}]

From Lemma \ref{le:ingham},  there exists a smooth function $f \geq 0$ such that $\widehat{f}(0) = 1$, $\widehat{f}$ is compactly supported in $[-\frac{1}{2\pi},\frac{1}{2\pi}]$ and
   $$
   |f(x)| \ll e^{- 2 \sqrt{|x|}}. 
   $$
Applying Lemma \ref{le:Analytic} with this choice for $f$ and $X=\exp(e^{n_{\mathcal L}})$, we find by the mean-value theorem that for every $\tau$ and $h \in G_0$, there exists a $k \geq 0$ and $\frac{1}{4} \cdot (2^{k} - 1) \leq |u| \leq \tfrac{1}{4} \cdot (2^{k + 1} - 1)$ such that
   \begin{equation} \label{eq:opoi} \log |\zeta(\tfrac 12 + \ii \tau + \ii h + \ii u)| - S_{n_{\mathcal{L}}}(h + u) - P_{n_0}(h + u) - \sqrt{|u| e^{n_{\mathcal{L}}}} \geq -C \end{equation}
  with $C > 0$ an absolute constant and where we remind the definition \eqref{eqn:Pn0}.

  By (\ref{eqn:41}) and (\ref{eqn:42}) in Lemma \ref{le:Easy}, the probability (in $\tau$) that there exists an $|h| \leq 1$ and $k \geq 1$ for which \eqref{eq:opoi} holds is $\ll e^{-n}$. Moreover, by (\ref{eqn:43}) in Lemma \ref{le:Easy}, we also know that $$\max_{\substack{|h| \leq 1 \\ |u| \leq 1/4}} |P_{n_0}(h + u)| \leq 20 y$$ for all $\tau$ outside of a set of probability $\ll e^{-y}$. Therefore, for all $\tau$ outside of a set of probability $\ll e^{-y}$ we find that for all $h \in G_0$ there exists a $|u| \leq 1/4$ such that
 $$ \log |\zeta(\tfrac 12 + \ii\tau + \ii h + \ii u)| - S_{n_{\mathcal{L}}}(h + u) - \sqrt{|u| e^{n_{\mathcal{L}}}} \geq -C - 20 y.$$
  Since $G_0 \subset [-\tfrac 12, \tfrac 12]$, it follows that for all $\tau$ outside of a set of measure $\ll e^{-y}$, for all $h \in G_0$, there exists an $|u| \leq 1/4$ such that
   \begin{align*}
   \max_{|v| \leq 1} \log |\zeta(\tfrac 12 + \ii  \tau + \ii v)| & > S_{n_{\mathcal{L}}}(h + u) + \sqrt{|u| e^{n_{\mathcal{L}}}} - 2C - 20 y \\ & \geq \min_{|u| \leq 1} (S_{n_{\mathcal{L}}}(h + u) + \sqrt{|u| e^{n_{\mathcal{L}}}}) - 2C - 20 y. 
   \end{align*}
   We now take an $h \in G_0$ that maximizes the right-hand side, and the claim follows. 
 \end{proof}
  
   \section{Proof of Proposition \ref{prop:Low}}

   The following lemma will be important.
   \begin{lemma} \label{le:Bound}
     Let $n_0 \leq \ell \leq n_{\mathcal{L}}$.
     Let $v \geq 1$ and $0 \leq k \leq n$ be given.
     Let $\mathcal{Q}$ be a Dirichlet polynomial supported on primes $p$ or their squares $p^2$, such that $e^{\ell} \leq \log p \leq e^{n_{\mathcal{L}}}$ and of length $\leq \exp(\tfrac{1}{200 v} e^n)$:
\begin{equation}\label{eqn:Qdef}
\mathcal{Q}(s)=\sum_{e^{\ell} \leq \log p \leq e^{n_{\mathcal{L}}}}\left(\frac{a(p)}{p^s}+\frac{b(p)}{p^{2s}}\right),
\end{equation}
where we also assume $|b(p)|\leq 1$.
Then 
     \begin{align} \label{eq:to bound 101}
       \mathbb{E} \Big [ & \sup_{\substack{|h| \leq 1 \\ |u| \leq e^{-k + 1}}} |\mathcal{Q}(\tfrac 12 + \ii  \tau + \ii  h + \ii  u)  - \mathcal{Q}(\tfrac 12 + \ii  \tau + \ii  h)|^{2v} \cdot \mathbf{1}_{h \in G_{\ell}} \Big ] \\ \nonumber & \ll e^{n_{\mathcal{L}} - n_0 - \ell + 10 ((\ell-n_0) \wedge (n_{\mathcal{L}} - \ell))^{3/4} + 20 y} \\ & \ \ \ \  \times 100^{v} v! \cdot \Big ( \Big ( e^{-2k + 4} \sum_{e^{\ell} \leq \log p \leq e^k} \frac{|a(p)|^2 \log^2 p}{p} \Big )^{v}  + \Big ( 16 \sum_{e^k \leq \log p} \frac{|a(p)|^2}{p} \Big )^{v} \cdot e^{n_{\mathcal{L}} - k} +1\Big )\label{eqn:bounded}.
       \end{align}
   \end{lemma}
   \begin{proof} To simplify the exposition we first assume that $b(p)= 0$ for all $p$.
     Since $G_{\mathcal{\ell}} \subset G_{0} = e^{-(n_{\mathcal{L}} - n_0)} \mathbb{Z} \cap [-1,1]$ we have,
     \begin{align*}
     \sup_{\substack{|h| \leq 1 \\ |u| \leq e^{-k + 1}}} & | \mathcal{Q}(\tfrac 12 + \ii  \tau + \ii  h + \ii  u) - \mathcal{Q}(\tfrac 12 + \ii  \tau + \ii  h) |^{2v} \cdot \mathbf{1}_{h \in G_{\ell}} \\ & \leq \sum_{h \in G_0} \sup_{|u| \leq e^{-k + 1}} | \mathcal{Q}(\tfrac 12 + \ii  \tau + \ii  h + \ii  u) - \mathcal{Q}(\tfrac 12 + \ii  \tau + \ii  h) |^{2v} \cdot \mathbf{1}_{h \in G_{\ell}}.
     \end{align*}
     Taking the expectation we find that \eqref{eq:to bound 101} is
     \begin{align} \label{eq:to bound 102}
     & \leq e^{n_{\mathcal{L}} - n_0} \cdot \mathbb{E} \Big [ \sup_{|u| \leq e^{-k + 1}} |\mathcal{Q}(\tfrac 12 + \ii \tau + \ii u) - \mathcal{Q}(\tfrac 12 + \ii \tau)|^{2v} \cdot \mathbf{1}_{0 \in G_{\ell}} \Big ].
     \end{align}
     We now split the Dirichlet polynomial $\mathcal{Q}(\tfrac 12 + \ii  \tau + \ii u) - \mathcal{Q}(\tfrac 12 + \ii \tau)$ into two parts. One part $\mathcal{Q}_{\leq k}(\tfrac 12 + \ii  \tau + \ii  u) - \mathcal{Q}_{\leq k}(\tfrac 12 + \ii  \tau)$ composed of primes $p$ with $\log p \leq e^k$ and another part supported on primes $p$ with $\log p > e^k$, denoted $\mathcal{Q}_{> k}(\tfrac 12 + i \tau + i u) - \mathcal{Q}_{> k}(\tfrac 12 + i \tau)$. For the first part, for $|u| \leq e^{-k + 1}$, 
     \begin{align*}
     |\mathcal{Q}_{\leq k}(\tfrac 12 + \ii  \tau + \ii  u) - \mathcal{Q}_{\leq k}(\tfrac 12 + \ii  \tau)|^{2v} & \leq \Big (\int_{0}^{e^{-k + 1}} |\mathcal{Q}_{\leq k}'(\tfrac 12 + \ii  \tau + \ii x)| \rd x \Big )^{2v} \\ & \leq e^{- (2v - 1) (k - 1)} \int_{0}^{e^{-k + 1}} |\mathcal{Q}_{\leq k}'(\tfrac 12 + \ii  \tau + \ii x)|^{2v} \rd x.
     \end{align*}
     Then
     $$
     \mathbb{E} \Big [ |\mathcal{Q}_{\leq k}'(\tfrac 12 + \ii \tau + \ii x)|^{2v} \cdot \mathbf{1}_{0  \in G_{\ell}} \Big ] \ll v! \cdot \Big ( \sum_{\log p \leq e^k} \frac{|a(p)|^2 \log^2 p}{p} \Big )^{v} \cdot e^{-\ell + 20y +  10 ((\ell-n_0) \wedge (n_{\mathcal{L}} - \ell))^{3/4})},
     $$
     using  Proposition \ref{prop: 1 point}, Lemma \ref{lem:Sound} and the Ballot theorem from Proposition \ref{prop:barrier}. Therefore
     \begin{align*}
     e^{n_{\mathcal{L}} - n_0} \cdot \mathbb{E} \Big [ \sup_{|u| \leq e^{-k + 1}} |& \mathcal{Q}_{\leq k}(\tfrac 12 + \ii \tau + \ii u)  - \mathcal{Q}_{\leq k}(\tfrac 12 + \ii  \tau)|^{2v} \cdot \mathbf{1}_{0 \in G_{\ell}} \Big ] \\ & \ll e^{n_{\mathcal{L}} - n_0 - \ell + 20y + 10 ((\ell-n_0) \wedge (n_{\mathcal{L}} - \ell))^{3/4}} \cdot v! \cdot \Big ( e^{-2k + 2} \sum_{\log p \leq e^k} \frac{|a(p)|^2 \log^2 p}{p} \Big )^{v}. 
     \end{align*}
    For the second part, we bound the contribution of $\mathcal{Q}_{\geq k}(\tfrac 12 + \ii  \tau + \ii u) - \mathcal{Q}_{\geq k}(\tfrac 12 + \ii \tau)$ simply by the triangle inequality and the discretization Lemma (\ref{lem: discretization}) applied to $D=\mathcal{Q}_{\geq k}^v$,  followed by Proposition \ref{prop: 1 point}.  This gives
     \begin{align*}
     e^{n_{\mathcal{L}} - n_0} \cdot \mathbb{E} \Big [ \sup_{|u| \leq e^{-k + 1}} | & \mathcal{Q}_{\geq k}(\tfrac 12 + \ii \tau + \ii u) - \mathcal{Q}_{\geq k}(\tfrac 12 + \ii  \tau)|^{2v} \cdot \mathbf{1}_{0 \in G_{\ell}} \Big ] \\ & \ll e^{n_{\mathcal{L}} - n_0 - \ell + 20 y + 10 ((\ell-n_0) \wedge (n_{\mathcal{L}} - \ell))^{3/4}} \cdot e^{n_{\mathcal{L}} - k} \cdot 2^{2v}\, v! \cdot \Big ( 4 \sum_{\log p > e^k} \frac{|a(p)|^2}{p} \Big )^{v}. 
     \end{align*}
     Combining everything we obtain the claim when $b(p)=0$.  When $b$ is non-trivial,  the only difference is that we cannot dirrectly  apply Lemma \ref{lem:Sound}  to bound the moments of $\mathcal{Q}$: instead,  we just use $|X+Y|^{2v}\leq 2^{2v}(|X|^{2v}+|Y|^{2v})$
for $X=\sum \frac{a(p)}{p^s}$, $Y=\sum\frac{b(p)}{p^{2s}}$, and apply Lemma \ref{lem:Sound} separately to each term.  The assumption $|b(p)|\leq 1$ allows to absorb the contribution of $|Y|^{2v}$ into the $+1$ in \eqref{eqn:bounded}.
\end{proof}
   
   We are now ready to prove Proposition \ref{prop:Low}.

   \begin{proof}[Proof of Proposition \ref{prop:Low}]
If there exists an $h \in G_{\mathcal{L}}$ and $|u| \leq 1$ such that
     \begin{equation} \label{eq:stuff2}
     |S_{n_{\mathcal{L}}}(h + u) - S_{n_{\mathcal{L}}}(h)| > 20 y + \sqrt{|u| e^{n_{\mathcal{L}}}},
     \end{equation}
     then there exists a $0 \leq k < n_{\mathcal{L}}' := n_{\mathcal{L}} - \lfloor 2 \log y \rfloor$ such that
     $$
     \sup_{\substack{|h| \leq 1 \\ |u| \leq e^{-k + 1}}} |S_{n_{\mathcal{L}}}(h + u) - S_{n_{\mathcal{L}}}(h)| \cdot \mathbf{1}_{h \in G_{\mathcal{L}}} \geq e^{(n_{\mathcal{L}} - k) / 2}.
     $$
     Notice that we can stop at $k = n_{\mathcal{L}}' := n_{\mathcal{L}} - \lfloor 2 \log y \rfloor$ thanks to the term $20 y$. 
     Therefore it suffices to bound
     \begin{align} \label{eq:MAINBOUND}
     \sum_{0 \leq k < n'_{\mathcal{L}}} \mathbb{P} \Big ( \sup_{\substack{|h| \leq 1 \\ e^{-k} \leq |u| \leq e^{-k + 1}}} |S_{n_{\mathcal{L}}}(h + u) - S_{n_{\mathcal{L}}}(h)| \cdot \mathbf{1}_{h \in G_{\mathcal{L}}} \geq e^{(n_{\mathcal{L}} - k) / 2} \Big ).
     \end{align}
     Suppose now that $|u| \leq e^{-k + 1}$ for some $0 \leq k < n'_{\mathcal{L}}$.
     Notice that
     \begin{align*}
       |S_{n_{\mathcal{L}}}(h + u) - S_{n_{\mathcal{L}}}(h)| \mathbf{1}_{h \in G_{\mathcal{L}}} \leq & \sum_{n_0\leq j < k} |(S_{j + 1} - S_j)(h + u) - (S_{j + 1} - S_{j})(h)|\mathbf{1}_{h \in G_{j}} \\ & \ \ \ \  + |(S_{n_{\mathcal{L}}} - S_{k})(h + u) - (S_{n_{\mathcal{L}}} - S_k)(h)| \mathbf{1}_{h \in G_{k}},
     \end{align*}
because $S_{n_0}=0$.
     Therefore, by the union bound, for each $0 \leq k \leq n_{\mathcal{L}}'$, 
     \begin{align} \nonumber
       \mathbb{P} & \Big ( \sup_{\substack{|h| \leq 1 \\ e^{-k} \leq |u| \leq e^{-k + 1}}} |S_{n_{\mathcal{L}}}(h + u) - S_{n_{\mathcal{L}}}(h)| \cdot \mathbf{1}_{h \in G_{\mathcal{L}}} \geq e^{(n_{\mathcal{L}} - k) / 2} \Big ) \\ \nonumber & \leq \sum_{0 \leq j < k} \mathbb{P} \Big ( \sup_{\substack{|h| \leq 1 \\ |u| \leq e^{- k + 1}}} |(S_{j + 1} - S_j)(h + u) - (S_{j + 1} - S_{j})(h)|\mathbf{1}_{h \in G_{j}} \geq \frac{e^{(n_{\mathcal{L}} - k) / 2}}{4 (k - j)^2} \Big ) \\ \label{eq:finalTerm} & \ \ \ \ \ \ \ \ \ + \mathbb{P} \Big (      \sup_{\substack{|h| \leq 1 \\ |u| \leq e^{- k + 1}}} |(S_{n_{\mathcal{L}}} - S_{k})(h + u) - (S_{n_{\mathcal{L}}} - S_k)(h)| \mathbf{1}_{h \in G_{k}} \geq \frac{e^{(n_{\mathcal{L}} - k) / 2}}{4} \Big ).
     \end{align}
     We now estimate each of the above probabilities using Chernoff's bound.
     According to Lemma \ref{le:Bound} for $0 \leq j < k$, for $v\geq 1$, we have
     \begin{align} 
\notag       \mathbb{P} & \Big ( \sup_{\substack{|h| \leq 1 \\ |u| \leq e^{- k + 1}}} |(S_{j + 1} - S_j)(h + u) - (S_{j + 1} - S_{j})(h)|\mathbf{1}_{h \in G_{j}} \geq \frac{e^{(n_{\mathcal{L}} - k) / 2}}{4 (k - j)^2} \Big ) \\ & \ll\notag (4 (k - j))^{4v} \cdot \mathbb{E} \Big [ \sup_{\substack{|h| \leq 1 \\ |u| \leq e^{-k + 1}}} \frac{|(S_{j + 1} - S_j)(h + u) - (S_{j + 1} - S_{j})(h)|^{2v}}{e^{v (n_{\mathcal{L}} - k)}} \cdot \mathbf{1}_{h \in G_{j}} \Big ] \\ & \ll(k - j)^{4v} \cdot e^{n_{\mathcal{L}} - n_0 - j + 20 y + 10 ((j-n_0) \wedge (n_{\mathcal L} - j)^{3/4})} \cdot e^{- v (n_{\mathcal{L}} - k)} \cdot v! \cdot e^{\tilde Cv} \cdot e^{2v(j - k)}.
  \label{Balotprecis}   \end{align}
 The above $e^{2v(j-k)}$ factor is due to the contribution of  $\big(e^{-2k + 4} \sum_{e^{j} \leq \log p \leq e^{j+1}} \frac{|a(p)|^2 \log^2 p}{p} \big )^{v}  $ in Lemma \ref{le:Bound}.
We choose $v=\lfloor e^{n_{\mathcal L}-j-C}\rfloor$ for fixed $C>0$. Then the Dirichlet sum $S_{j+1}^v$ has length $\exp(e^j\cdot  e^{n_{\mathcal L}-j-C})\leq \exp(e^n/200)$ for large enough $C$, so Lemma \ref{le:Bound} can be applied. The above bound becomes, for some absolute positive constant $\tilde C$,
     \begin{align*}
     &\ll e^{n_{\mathcal{L}} - n_0 - j + 20 y + 10 ((j-n_0) \wedge (n_{\mathcal{L}} - j))^{3/4}}\exp \Big (v\log v- (n_{\mathcal L}+k-2j-4\log (k-j)-\tilde C) v \Big )\\
&\ll e^{n_{\mathcal{L}} - n_0 - j + 20 y + 10 ((j-n_0) \wedge (n_{\mathcal{L}} - j))^{3/4}}\exp \Big (-v(k-j-4\log(k-j)+C-\tilde C) \Big )\\
&\ll e^{n_{\mathcal{L}} - n_0 - j + 20 y + 10 ((j-n_0) \wedge (n_{\mathcal{L}} - j))^{3/4}}\exp \Big (-c e^{n_{\mathcal{L}}-j}(k-j)\Big ),
     \end{align*}
for some small constant $c>0$, by choosing $C$ large enough.
     Summing over $n_0 \leq j < k$ we see that the sum is dominated by the contribution of the last term $j = k - 1$. The full sum (over $j$ and $k$) is therefore bounded with
     $$
     \sum_{0\leq k<n'_{\mathcal L}}e^{n_{\mathcal{L}} - n_0 - k + 20 y + 10 ((k-n_0) \wedge (n_{\mathcal{L}} - k))^{3/4}} \exp( - ce^{n_{\mathcal{L}} - k}),$$
which is dominated by $k=n'_{\mathcal{L}}-1$ and gives a global bound $c^{c_1 y-c_2 y^2}$ for some absolute $c_1, c_2>0$.

     The second probability in \eqref{eq:finalTerm} is again by a Chernoff bound, 
     \begin{align}\notag 
       \mathbb{P} & \Big ( \sup_{\substack{|h| \leq 1 \\ |u| \leq e^{- k + 1}}} |(S_{n_{\mathcal{L}}} - S_k)(h + u) - (S_{n_{\mathcal{L}}} - S_{k})(h)|\mathbf{1}_{h \in G_{k}} \geq \frac{e^{(n_{\mathcal{L}} - k) / 2}}{4} \Big ) \\ & \ll 4^{4v} \cdot \mathbb{E} \Big [ \sup_{\substack{|h| \leq 1 \\ |u| \leq e^{-k + 1}}} \frac{|(S_{n_{\mathcal{L}}} - S_k)(h + u) - (S_{n_{\mathcal{L}}} - S_{k})(h)|^{2v}}{e^{v (n_{\mathcal{L}} - k)}} \cdot \mathbf{1}_{h \in G_{k}} \Big ] \notag\\ & \ll e^{n_{\mathcal{L}} - n_0 - k + 10 ((k-n_0) \wedge (n_{\mathcal{L}} - k))^{3/4}}\cdot e^{- v (n_{\mathcal{L}} - k)} \cdot v! \cdot e^{\tilde Cv} \cdot (n_{\mathcal{L}} - k)^{v}\cdot e^{n_{\mathcal L}-k}\label{eqn:bidule},
     \end{align}
for some absolute $\tilde C$.
     Choosing $v = \lfloor e^{n_{\mathcal{L}} - k-C}/(n_{\mathcal L}-k)^4\rfloor$, we see that this is also $\ll e^{v(\tilde C-C)}$. Therefore, for alarge enough absolute constant $C > 0$, the full contribution of 
\eqref{eq:finalTerm} after summation over $k$ is
     $$
     \ll \sum_{0 \leq k < n_{\mathcal{L}'}} e^{n_{\mathcal{L}} - n_0 - k + 10 ((k-n_0) \wedge (n_{\mathcal{L}} - k))^{3/4}} \cdot \exp(-e^{n_{\mathcal{L}} - k-C}/(n_{\mathcal L}-k)^4)\ll e^{\tilde c_1 y-\tilde c_2\frac{y^2}{(\log y)^4}},
     $$
 for some absolute $\tilde c_1,\tilde  c_2>0$,
where we used that the main contribution comes from $k=n'_{\mathcal L}$.  This concludes the proof. 
   \end{proof}

   \section{Proof of Proposition \ref{prop: comp second}}

We first need a lemma which precisely captures the coupling/decoupling of the Gaussian walks $\mathcal G_k(h)$ defined in \eqref{eqn: Gk} as a function of the distance $|h-h'|$. For this,  the following elementary lemma will be key in the decoupling regime $|h-h'|>e^{-j}$.

\begin{lemma} \label{le:gaussian} Let $|\rho|<\mathfrak{s}^2$.
Consider the following Gaussian vectors and their covariance matrices: 
\begin{align*}
&(\mathcal{N}_1,\mathcal{N}_1'),&\mathcal C_1=\left(\begin{matrix} 
\mathfrak{s}^2  & \rho\\
\rho &\mathfrak{s}^2 
\end{matrix}\right),\\
&(\mathcal{N}_2,\mathcal{N}_2'),&\mathcal C_2=\left(\begin{matrix} 
\mathfrak{s}^2+|\rho|  &0\\
0&\mathfrak{s}^2+|\rho| 
\end{matrix}\right).
\end{align*}
Then for any measurable set $A\subset\mathbb{R}^2$ we have
$$
\mathbb{P}((\mathcal{N}_1,\mathcal{N}'_1)\in A)\leq\sqrt{\frac{\mathfrak{s}^2+|\rho| }{\mathfrak{s}^2-|\rho|}}\cdot \mathbb{P}((\mathcal{N}_2,\mathcal{N}'_2)\in A).
$$
\end{lemma}

\begin{proof}
The proof is simply by expanding the density of $(\mathcal N_1, \mathcal N'_1)$, which is
\begin{equation}
\label{eqn: pdf}
\frac{1}{2\pi \sqrt{\mathfrak{s}^4-\rho^2}} \exp\left(-\frac{\mathfrak{s}^2 w^2+\mathfrak{s}^2 z^2-2\rho w z}{2(\mathfrak{s}^4-\rho^2)}\right).
\end{equation}
If $\rho\geq 0$ then  for any $w,z\in\mathbb{R}$ we have $\mathfrak{s}^2 w^2+\mathfrak{s}^2 z^2-2\rho w z\geq (\mathfrak{s}^2-\rho)(w^2+z^2)$ so that
$$
\frac{\mathfrak{s}^2 w^2+\mathfrak{s}^2 z^2-2\rho w z}{2(\mathfrak{s}^4-\rho^2)}\geq \frac{w^2+z^2}{2(\mathfrak{s}^2+\rho)},
$$
and the conclusion follows.
If $\rho\leq 0$ then from the previous case for any $B\subset\mathbb{R}^2$
$$
\mathbb{P}((\mathcal{N}_1,-\mathcal{N}'_1)\in B)\leq\sqrt{\frac{\mathfrak{s}^2-\rho}{\mathfrak{s}^2+\rho}}\cdot \mathbb{P}((\mathcal{N}_2,-\mathcal{N}'_2)\in B),
$$
which concludes the proof by choosing $B=\{(x,-y):(x,y)\in A\}$.
\end{proof}

\begin{proof}[Proof of  Proposition \ref{prop: comp second}]
We have
$\E[(\#\mathfrak{G}^{+}_{\mathcal L})^2]=
\sum_{h,h'}\PP\Big(\mathfrak S(h)\cap \mathfrak S(h')\Big)$ 
where
\begin{equation}
\label{eqn: S}
\mathfrak S(h)=\{\mathcal G_k(h)\in [L_k-1,U_k+1], \forall n_0<k\leq n_\mathcal L\}, \quad h\in [-1,1].
\end{equation}
In what follows, we fix $h$, $h'$ and simply write $(\mathcal G_k,\mathcal G_k')$ for $(\mathcal G_k(h),\mathcal G_k(h'))$.
We divide the above sum over pairs in three ranges of $|h-h'|$; this is necessary to achieve the precision $1+\oo(1)$ required by Proposition \ref{prop: comp second}.

\subsection{Case $|h-h'|>e^{-n_0/2}$}
This is the dominant term. 
We can express the events $\mathfrak S(h)$ in terms of the increments using \ref{eqn: inclusion UB}, and then in terms of independent increments using Lemma \ref{le:gaussian}. 
Under the product over $j$, the multiplicative error from Lemma \ref{le:gaussian} is
\begin{equation}\label{eqn:prefactor}
\prod_{n_0<j\leq n_\mathcal L}\sqrt{\frac{\mathfrak{s}_j^2+|\rho_j| }{\mathfrak{s}_j^2-|\rho_j|}}=\exp\Big(\OO(\sum_{n_0\leq j\leq  n_\mathcal L} \rho_j)\Big)=\exp\Big(\OO(\sum_{n_0\leq j\leq  n_\mathcal L} \frac{1}{e^j|h-h'|}\Big)\leq 1+ \OO(e^{-n_0/2}),
\end{equation}
therefore we obtain
$$
\sum_{|h-h'|>e^{-n_0/2}}\PP\Big(\mathfrak S(h)\cap \mathfrak S(h')\Big)\leq (1+\OO(n_0^{-10}))\cdot  \Big(\PP(\widetilde{\mathcal G}_{j}\in [L_j-1, U_j+2], n_0<j\leq n_\mathcal L)\Big)^2,
$$
where $\widetilde{\mathcal G}_{j}=\sum_{i\leq j}\widetilde{\mathcal N}_j$ and the independent  Gaussian centered $\widetilde{\mathcal N}_j$'s have variance $\mathfrak{s}_j^2+|\rho_j|$. Moreover  the change from the original interval $[L_j-1, U_j+1]$ to $[L_j-1, U_j+2]$ is due to (\ref{eqn: inclusion LB}) when transferring the constraint on increments back to the random walk itself.
From the Ballot theorem in Proposition \ref{prop:barrier} the barrier can be changed into $[L_j+1,U_j-1]$, and the $\widetilde{\mathcal G}_{j}$ can be replaced by ${\mathcal G}_{j}$ at a combined multiplicative cost of $1+\OO(y^{-c})$,  so that in particular
$$
\sum_{|h-h'|>e^{-n_0/2}}\PP\Big(\mathfrak S(h)\cap \mathfrak S(h')\Big)\leq (1+\OO(y^{-c}))(\E[\#\mathfrak{G}_{\mathcal L}^{-}])^2.
$$
All the other cases will be much smaller than $(\E[\#\mathfrak{G}_{\mathcal L}^{-}])^2$. 

\subsection{Case $e^{-n_0}< |h-h'|\leq e^{-n_0/2}$}
The same reasoning as above applies in this case. The multiplicative error term analogue to (\ref{eqn:prefactor}) is now $\OO(1)$,  and the precise estimate of this error is not necessary since there are only $\ll e^{2(n_\mathcal L-n_0)}e^{-n_0/2}$ pairs $(h,h')$ to consider. Therefore, we obtain
$$
\sum_{e^{-n_0}< |h-h'|\leq e^{-n_0/2}}\PP(\mathfrak S(h)\cap \mathfrak S(h'))\ll e^{-\frac{n_0}{2}} (\E[\#\mathfrak{G}_{\mathcal L}^{-}])^2.
$$

\subsection{Case $e^{-(n_{\mathcal{L}}-n_0)} \leq |h-h'| \leq e^{-n_0}$}
We start by writing,
$$
\sum_{\substack{h, h' \in \mathfrak{G}^{+}_{\mathcal{L}} \\ e^{-(n_{\mathcal{L}} - n_0)} < |h -h'| \leq e^{-n_0}}}  \mathbb{P} \Big ( \mathfrak{S}(h) \cap \mathfrak{S}(h') \Big ) =
\sum_{n_0 \leq j^{\star} \leq n_{\mathcal{L}}} \sum_{\substack{h, h' \in \mathfrak{G}^{+}_{\mathcal{L}} \\ j^{\star} = \lfloor \log |h - h'|^{-1} \rfloor}}  \mathbb{P} \Big ( \mathfrak{S}(h) \cap \mathfrak{S}(h') \Big ).
$$
In order to evaluate $\mathbb{P}(\mathfrak{S}(h) \cap \mathfrak{S}(h'))$,  we apply the Gaussian decorrelation Lemma \ref{le:gaussian} for the increments $j\geq j^*$.  
For the increments before $j^*$, it will be useful to consider the 
random variables
\begin{equation}\label{eqn:orth}
\overline{\mathcal G}_{j}=\frac{{\mathcal G}_{j}+\mathcal G'_{j}}{2},\qquad \mathcal G^\perp_{j}=\frac{{\mathcal G}_{j}-\mathcal G'_{j}}{2}, \quad n_0< j\leq n_\mathcal L.
\end{equation}
Note that $(\overline{\mathcal{G}}_j)_{j}$ and $(\mathcal{G}_{j}^{\perp})_{j}$ are independent and  $\mathcal{G}_j = \overline{\mathcal{G}}_j + \mathcal{G}_j^{\perp}$,  $\mathcal{G}_j' = \overline{\mathcal{G}}_j - \mathcal{G}_j^{\perp}.$

As before we can express the events $\mathfrak S(h)$ in terms of the increments using \eqref{eqn: inclusion UB}; here we only use such a decomposition for the  process $\mathcal{G}_{j^{\star},j} := \mathcal{G}_j - \mathcal{G}_{j^{\star}}$,  approximating its increments with independent ones through  Lemma \ref{le:gaussian}, up to a multiplicative error equal to
$$
\prod_{j^*<j\leq n_\mathcal L}\sqrt{\frac{\mathfrak{s}_j^2+|\rho_j| }{\mathfrak{s}_j^2-|\rho_j|}}=\OO(1).
$$
For $h, h'$ such that $\lfloor \log |h - h'|^{-1} \rfloor = j^{\star}$, this gives
\begin{align} \label{eqn: prob 3rd}
\mathbb{P}(\mathfrak{S}(h) \cap \mathfrak{S}(h'))\ll \sum_{L_{j^{\star}} - 1 \leq v-q, v+q \leq U_{j^{\star}}} \text{C}_{j^{\star}}(h, h', v, q) \ \text{D}_{j^{\star}}(h, v - q)\text{D}_{j^{\star}}(h', v + q),
\end{align}
where
\begin{align*}
  \text{C}_{j^{\star}}(h, h', v, q) & := \mathbb{P} \Big ( \mathcal{G}_{j}, \mathcal{G}'_j \in [L_j - 1, U_j + 1] \text{ for all } j < j^{\star},   \overline{\mathcal{G}_{j^{\star}}} \in [v, v + 1], \mathcal{G}_{j^{\star}}^{\perp} \in [q, q + 1] \Big ), \\
  \text{D}_{j^{\star}}(h, v)& := \mathbb{P} \Big ( \widetilde{\mathcal{G}}_{j^{\star}, j}(h) + v \in [L_j - 2, U_j + 2] \text{ for all } j > j^{\star} \Big),
\end{align*}
and $\widetilde{\mathcal{G}}_{j^{\star},j}=\widetilde{\mathcal{G}}_{j}-\widetilde{\mathcal{G}}_{j^*}$.
The proof now reduces to bounding the correlated (C) and decorrelated (D) terms. 

\subsubsection{The Correlated term} 

Note that if $\mathcal{G}_{j}, \mathcal{G}_j' \in [L_j - 1, U_j + 1]$ for all $j < j^{\star}$ then also
$
\overline{\mathcal G}_{j}\in [L_j - 1, U_j + 1]
$
for all $j < j^{\star}$.  Moreover,   $\mathcal{G}_{j^*}^{\perp}$ is independent of $(\overline{\mathcal{G}}_j)_{j\leq j^*}$. We can therefore bound
$$
\text{C}_{j^{\star}}(h,h', v, q) \leq \mathbb{P} \Big ( \overline{\mathcal G}_{j}\in [L_j - 1, U_j + 1] \text{ for all } j < j^{\star},  \overline{\mathcal{G}_{j^{\star}}} \in [v, v + 1]\Big)\cdot \mathbb{P}\big(\mathcal{G}_{j^{\star}}^{\perp} \in [q, q + 1) \big ). 
$$
The Gaussian $\mathcal G^\perp_{j^\star}$ is centered with variance $\ll \sum_{j\leq j^\star}\e_j^2\ll 1$ from \eqref{eqn: cov estimate} and \eqref{eq:epsilonjdef}. We thus have
$$
\PP( \mathcal G^\perp_{j} \in [q, q + 1))\ll e^{-cq^2},\quad \text{for some $c>0$.} 
  $$
Moreover, $(\overline{\mathcal{G}}_j)_{j\leq j^*}$ satisfies the assumptions of Proposition \ref{prop:barrier},  and $\overline{\mathcal{G}}_{j^*}$ has variance $\frac{1}{2}\sum_{j\leq j^*}(\mathfrak{s}_j^2+\rho_j)=\frac{j^*-n_0}{2}+\OO(1)$ from \eqref{eqn:skasymp} and \eqref{eqn: cov estimate}. Thus, uniformly in $|v|\leq 100(j^*-n_0)$, we have
  $$
  \text{C}_{j^{\star}}(h,h', v, q) \ll \frac{U_{n_0} \cdot (U_{j^{\star}} - v + 1)}{(j^{\star} - n_0)^{3/2}} \cdot e^{- \frac{v^2}{j^{\star} - n_0}  - c q^2 }. 
  $$

\subsubsection{The Decorrelated term}
We condition on $\mathcal{G}_{j, j^{\star}} \in [v_2, v_2 + 1]$ which implies that $v_1 + v_2 \in [U_{n_{\mathcal{L}}}, L_{n_{\mathcal{L}}}]$. Then by the Ballot theorem stated in Proposition \ref{prop:barrier}, $\text{D}_{j^{\star}}(h, v_1 - q)$ is
$$\ll \sum_{\substack{L_{n_{\mathcal{L}}} - 2 \leq v_1 + v_2 \leq U_{n_{\mathcal{L}}} + 2}} \frac{(U_{j^{\star}} - v_1 + q + 1) (U_{n_{\mathcal{L}}} - v_1 - v_2 + q + 1)}{(n_{\mathcal{L}} - j^{\star})^{3/2}} \cdot e^{ - \frac{(v_2 - q)^2}{n_{\mathcal{L}} - j^{\star}}},
$$
where we have used from \eqref{eqn:skasymp} and \eqref{eqn: cov estimate} to obtain $\E[(\widetilde{\mathcal{G}}_{j^{\star},n_{\mathcal L}})^2]=\sum_{j^*<j\leq n_{\mathcal L}}(\mathfrak{s}_j^2+|\rho_j|)=n_{\mathcal{L}}-j^*+\OO(1)$,  and $|v_2-q|\leq 100(n_{\mathcal{L}}-j^*)$.
Likewise, $\text{D}_{j^{\star}}(h, v_1 + q)$ is
$$
\ll \sum_{\substack{L_{n_{\mathcal{L}}} - 2 \leq v_1 + v_3 \leq U_{n_{\mathcal{L}}} + 2}} \frac{(U_{j^{\star}} - v_1 - q + 1) (U_{n_{\mathcal{L}}} - v_1 - v_3 - q + 1)}{(n_{\mathcal{L}} - j^{\star})^{3/2}} \cdot e^{- \frac{(v_3 + q)^2}{n_{\mathcal{L}} - j^{\star}}}
$$
\subsubsection{Putting it together}
The above estimates give, after summing over $q \in \mathbb{Z}$,
\begin{multline} \label{eq:MASD}
\mathbb{P} \Big ( \mathfrak{S}(h) \cap \mathfrak{S}(h') \Big )\ll\sum_{\substack{L_{j^{\star}} - 1 \leq v_1 \leq U_{j^{\star}} + 1 \\ L_{n_{\mathcal{L}}} - 2 \leq v_1+v_2, v_1+v_3 \leq U_{n_{\mathcal{L}}} + 2}} e^{- \frac{v_1^2}{j^{\star} - n_0} - \frac{v_2^2}{n_{\mathcal{L}} - j^{\star}} - \frac{v_3^2}{n_{\mathcal{L}} - j^{\star}}} 
\\
\times \frac{U_{n_0} (U_{j^{\star}} - v_1 + 1)^3 (U_{n_{\mathcal{L}}} - v_1 - v_2 + 1) (U_{n_{\mathcal{L}}} - v_1 - v_3 + 1)}{(n_{\mathcal{L}} - j^{\star})^3 \cdot (j^{\star} - n_0)^{3/2}}.
\end{multline}
We change the variables to $\overline{v_1}=v_1-\alpha(j^\star-n_0)$, $\overline{v_2}=v_2-\alpha(n_\mathcal L-j^\star)$ and $\overline{v_3} = v_3 - \alpha(n_{\mathcal{L}} - j^{\star})$ so that $\overline{v_1} + \overline{v_2} \in [L_{n_0}, U_{n_0}]$ and $\overline{v_1} + \overline{v_3} \in [L_{n_0}, U_{n_0}]$,   giving
\begin{equation*}
 \frac{e^{ - \frac{v_1^2}{j^{\star} - n_0} - \frac{v_2^2}{n_{\mathcal{L}} - j^{\star}} - \frac{v_3^2}{n_{\mathcal{L}} - j^{\star}}}}{(n_{\mathcal{L}} - j^{\star})^3 (j^{\star} - n_0)^{3/2}} 
 \ll e^{-2(n_\mathcal L-j^\star)-(j^\star-n_0) + 2 \overline{v_1} - 2 (\overline{v_1} + \overline{v_2}) - 2 (\overline{v_1} + \overline{v_3}) } \cdot \frac{n^{\frac{3}{2}\frac{j^\star}{n}} n^{3(1-\frac{j^\star}{n})}}{(j^\star-n_0)^{3/2}(n_{\mathcal{L}}-j^\star)^3}.
\end{equation*}
The contribution of the integral over $\overline{v_1}+\overline{v_2}\in [L_{n_0},U_{n_0}]$ is
$$
\int_{[L_{n_0},U_{n_0}]} (U_{n_0}- z+1) e^{-2z}\rd z\ll |L_{n_0}| e^{2|L_{n_0}|}.
$$
 The same bound holds for the integral over $\overline{v_1}+\overline{v_3}$. The integral over $\overline{v_1}$ is for $\mathcal B_{j^{\star}}=U_{n_0}-10\log (( j^{\star}-n_0)\wedge (n_\mathcal L-j^{\star}))$
 $$
\ll \int_{-\infty}^{\mathcal B_{j^\star}} ( \mathcal B_{j^\star}-\overline{v_{1}}+1)^3 e^{2\overline{v_{1}}} \rd \overline{v_{1}}\ll e^{2\mathcal B_{j^\star}}.
$$
Combining these estimates for the  $\OO(e^{2(n_\mathcal L-n_0)-j^\star})$ pairs with $\log |h-h'|^{-1}\geq j^{\star}$,  we obtain
\begin{multline}\label{lala}
\sum_{e^{-(n_\mathcal L - n_0)} \leq |h-h'|\leq e^{-n_0}}\PP(\mathfrak S(h)\cap \mathfrak S(h'))\ll 
e^{-n_0}U_{n_0} L_{n_0}^2 e^{4|L_{n_0}|}\sum_{j^\star} e^{2\mathcal B_{j^\star}} 
\frac{n^{\frac{3}{2}\frac{j^\star}{n}} n^{3(1-\frac{j^\star}{n})}}{(j^\star-n_0)^{3/2}(n_{\mathcal{L}}-j^\star)^3}\\
 \ll e^{-n_0}U_{n_0} L_{n_0}^2 e^{4|L_{n_0}|}\cdot  e^{2U_{n_0}}.
\end{multline}
On the other hand,   from Proposition \ref{prop:barrier} we have a simple lower bound for $\mathbb{E}[\# \mathfrak{G}^{-}_{\mathcal L}]$:
\begin{equation} \label{eq:firstestim}
\E[\#\mathfrak{G}^{-}_{\mathcal L}]=e^{n_{\mathcal L}-n_0}\cdot \mathbb{P}(\mathcal{G}_{k}\in [L_k + 1,U_k - 1],  n_0<k\leq n_{\mathcal L}) \gg U_{n_0} |L_{n_0}| e^{2 |L_{n_0}|}.
\end{equation}
We conclude from \eqref{lala} and \eqref{eq:firstestim} that
\begin{equation*}
\sum_{e^{-(n_\mathcal L - n_0)} \leq |h-h'|\leq e^{-n_0}}\PP(\mathfrak S(h)\cap \mathfrak S(h')) \ll U_{n_0}^{-1}e^{2U_{n_0}-n_0}(\E[\#\mathfrak{G}^{-}_{\mathcal L}])^2\ll e^{-y/10}(\E[\#\mathfrak{G}^{-}_{\mathcal L}])^2
\end{equation*}
by the choice of $n_0$ and $U_{n_0}$.

\subsection{Conclusion.} When $|h - h'| \leq e^{-(n_{\mathcal{L}} - n_0)}$,
because of the spacing constraint we necessarily have $h = h'$,  and the contribution from such trivial pairs admits the same upper bound as for $j^*=n_{\mathcal L}$ above. All together, we have obtained
$$
\E[(\#\mathfrak{G}^{+}_{\mathcal L})^2]\leq (1+\OO(y^{-c}))(\E[\#\mathfrak{G}_{\mathcal L}^{-}])^2,
$$
which concludes the proof of Proposition \ref{prop: comp second}.
\end{proof}

  \section{Proof of Theorem \ref{thm: right tail}}
\label{sect: right tail}

The proof of the theorem follows the same structure as the one of Theorem 1. The parameters need to be picked differently. 
We take for the times
$$
n_0=\lfloor y/100\rfloor, \qquad n_\mathcal L=\log\log (T^{1/100})=n-\log 100.
$$
The partial sums on primes are now starting from $p=2$ and not $\exp e^{n_0}$
 \begin{equation}
\label{eqn: Sk thm 3}
S_j(h)  = \sum_{\log\log  p \leq j} \re \Big( p^{-(1/2 + \ii \tau + \ii h)} + \frac 12 p^{- 2(1/2 + \ii \tau + \ii h)} \Big ), \quad j\in\mathbb N.
\end{equation}
The set of good points are
$$
G_0=[-\tfrac 12,\tfrac 12]\cap e^{-n_\mathcal L}\mathbb Z,\qquad G_j=\{h\in G_0: S_j\in [L_j,U_j], n_0\leq j\leq n_{\mathcal L}\},
$$
where the barriers are now for $j\geq n_0$
\begin{equation}
\begin{aligned}
\label{eqn: barriers thm3}
U_j&=y+\alpha j -10\log (j\wedge (n-j)),\\
L_j&=-10+(\alpha +\frac{y}{n_\mathcal L})j -(j\wedge (n-j))^{3/4}.
\end{aligned}
\end{equation}
The slope $\alpha$ is $1-\frac{3}{4}\frac{\log n}{n}$ as before. Both barriers are convex, which is crucial. Note that the final interval for $S_{n_\mathcal L}$ is $[L_{n_\mathcal L}, U_{n_{\mathcal L}}]$ where
$$
U_{n_\mathcal L}=n-\frac{3}{4}\log n +y\qquad L_{n_\mathcal L}=n-\frac{3}{4}\log n +y-10.
$$
The reason for the  slightly larger slope in $L_j$, i.e., $(\alpha +\frac{y}{n_\mathcal L})$ instead of $\alpha$,  is to ensure that the width of the final interval is order one. The factor $y/n_\mathcal L$ will not affect the proof.

It is necessary to take  $U_{n_0}=y+\alpha n_0$, as this is the origin of the factor $y$ in front of the exponential decay in Theorem \ref{thm: right tail}. 
For $y$ of order one, it would be possible to take $n_0=\OO(1)$. However, for larger $y$, the spread $U_j-L_j$ could be quite large for small $j$. This prevents a Gaussian comparison for small primes. 
For this reason,   the barrier starts  at $n_0$, a multiple of $y$. For these times, the spread is proportional to variance and the Gaussian comparison goes through. 

Unlike the left tail, we do need to include the small primes in the partial sums. Dropping the first $\exp e^{n_0}$ primes would give a lower bound $ye^{-2y}e^{-n_0}e^{-y^2/n}$, which is suboptimal for $n_0\asymp y$. 
A more involved analysis of the small primes would probably allow to improve the result range of Theorem \ref{thm: right tail}  to $y=\oo(n)$, matching the branching Brownian motion estimate.

For the proof of  Theorem \ref{thm: right tail}, we first need the analogue of Proposition \ref{prop:zetalocal}.

\begin{prop} \label{prop:analogue}
 We have, for any fixed $C > 10$, uniformly in $1 \leq y = \oo(n)$
  $$
  \mathbb{P} \Big ( \max_{|h| \leq 1} \log |\zeta(\tfrac 12 + \ii  \tau + \ii  h)| > n - \frac{3}{4} \log n + y - 10 C \Big ) \geq \mathbb{P} \Big ( \exists h \in G_{\mathcal{L}} \Big ) - \OO(e^{-50 C} y e^{- 2y} e^{-y^2 / n}).
  $$
\end{prop}
Therefore, upon taking $C$ large enough but fixed,  the estimate
\begin{equation}\label{eqn:enough!}
\PP\Big(\#G_{n_\mathcal L}\geq 1 \Big )\gg y e^{-2y} e^{-y^2/n}
\end{equation}
will imply  Theorem \ref{thm: right tail}.
Equation \eqref{eqn:enough!} follows directly from the Paley-Zygmund inequality from the propositions below.
\begin{prop}
\label{prop: 1 thm3}
Uniformly in $10\leq y\leq C\frac{\log\log T}{\log\log\log T}$,
$$
\E[\#G_{n_\mathcal L}]\gg y e^{-2y}e^{-y^2/n}.
$$
\end{prop}

\begin{prop}
\label{prop: 2 thm3}
Uniformly in $10\leq y\leq C\frac{\log\log T}{\log\log\log T}$,
$$
\E[(\#G_{n_\mathcal L})^2]\ll y e^{-2y}e^{-y^2/n}.
$$
\end{prop}

Unlike the left tail, the dominant term in the second moment will come from the pairs $h,h'$ that are very close, i.e., $|h-h'|\ll e^{-n_\mathcal L}$.

\subsection{Proof of Proposition \ref{prop:analogue}}

First we have the following easy variant of Proposition \ref{prop:max}

 \begin{prop} \label{prop:NewMax}
   We have, for $1000 < y < n^{1/10}$, 
   $$
\mathbb{P} \Big ( \max_{|h| \leq 1} \log |\zeta(\tfrac 12 + \ii  \tau + \ii h)| \geq \max_{h \in G_0} \min_{|u| \leq 1} (S_{n_{\mathcal{L}}}(h + u) + \sqrt{|u| e^{n_{\mathcal{L}}}}) - 2 C \Big ) \geq 1 - \OO(e^{-n}),
$$
with $C > 0$ an absolute constant. 
 \end{prop}
 \begin{proof}
   This is the same proof as Proposition \ref{prop:max}, the only difference is that this time we do not need to bound the contribution of the primes $p$ with $\log p \leq e^{n_0}$ and therefore there is no additional term $-20 y$.  Because of this, the exceptional set is also better, i.e., $e^{-n}$ instead of $e^{-y}$. 
 \end{proof}

We highlight the changes needed in Lemma \ref{le:Bound} and Proposition \ref{prop:Low}, with the following two variants.

\begin{lemma} \label{lem:NewBound}  
     Let $1 \leq \ell \leq n_{\mathcal{L}}$.
     Let $v \geq 1$ and $0 \leq k \leq n$ be given.
     Let $\mathcal{Q}$ be a Dirichlet polynomial as defined in \eqref{eqn:Qdef},  such that $e^{\ell} \leq \log p \leq e^{n_{\mathcal{L}}}$ and of length $\leq \exp(\tfrac{1}{200 v} e^n)$. Denote by $a(p)$ the coefficients of $\mathcal{Q}$. Then
     \begin{align} \label{eq:to bound 10101}
       \mathbb{E} \Big [ & \sup_{\substack{|h| \leq 1 \\ |u| \leq e^{-k + 1}}} |\mathcal{Q}(\tfrac 12 + \ii  \tau + \ii  h + \ii  u)  - \mathcal{Q}(\tfrac 12 + \ii  \tau + \ii  h)|^{2v} \cdot \mathbf{1}_{h \in G_{\ell}} \Big ] \\ \nonumber & \ll e^{n_{\mathcal{L}}} \ \mathbb{P}(\mathcal{G}_{\ell}) \cdot 2^{2v} v! \cdot \Big ( \Big ( e^{-2k + 4} \sum_{e^{\ell} \leq \log p \leq e^k} \frac{|a(p)|^2 \log^2 p}{p} \Big )^{v}  + \Big ( 16 \sum_{e^k \leq \log p} \frac{|a(p)|^2}{p} \Big )^{v} \cdot e^{n_{\mathcal{L}} - k} \Big ).
       \end{align}
Moreover, we simply bound $\mathbb{P}(\mathcal{G}_{\ell})\leq 1$ if $\ell<n_0$, and otherwise  
$$
\mathbb{P}(\mathcal{G}_{\ell}) \ll \frac{y}{\ell^{3/2}} \exp \Big ( - \ell + \frac{3}{2} \cdot \frac{\ell \log n_{\mathcal{L}}}{n_{\mathcal{L}}} - \frac{2 y \ell}{n_{\mathcal{L}}} - \frac{y^2 \ell}{n_{\mathcal{L}}^2} + 10 (\ell \wedge (n - \ell))^{3/4} \Big ).
$$
\end{lemma}
\begin{proof}
  The proof is the same as for  Lemma \ref{le:Bound}, the only differences being the  different bound for $\mathbb{P}(\mathcal{G}_{\ell})$ (when $\ell\geq n_0$) that arises from a different barrier.
Note that for $1\leq \ell\leq n_0$ there is no barrier so the proof does  not rely on Proposition \ref{prop: 1 point}, which requires $\ell\geq n_0$.
\end{proof}

 \begin{prop} \label{prop:NewLow}
   We have, for any $C > 10$, and for $y = o(n)$, 
   $$
   \mathbb{P} \Big (\forall h \in G_{\mathcal{L}} \ \forall |u| \leq 1 : |S_{n_{\mathcal{L}}}(h + u) - S_{n_{\mathcal{L}}}(h)| \leq C + \sqrt{|u| e^{n_{\mathcal{L}}}} \Big ) = 1 + \OO \Big ( e^{-50 C} y e^{-2 y} e^{-y^2 / n} \Big ). 
   $$
 \end{prop}
 \begin{proof}
   The proof is very similar to Proposition \ref{prop:Low} but we still find it worthwhile to include the details. 
If there exists an $h \in G_{\mathcal{L}}$ and $|u| \leq 1$ such that
     \begin{equation} \label{eq:stuff2}
     |S_{n_{\mathcal{L}}}(h + u) - S_{n_{\mathcal{L}}}(h)| > C + \sqrt{|u| e^{n_{\mathcal{L}}}}
     \end{equation}
     then there exists a $0 \leq k < n_{\mathcal{L}}' := n_{\mathcal{L}} - \lfloor 2 \log C \rfloor$ such that, 
     $$
     \sup_{\substack{|h| \leq 1 \\ |u| \leq e^{-k + 1}}} |S_{n_{\mathcal{L}}}(h + u) - S_{n_{\mathcal{L}}}(h)| \cdot \mathbf{1}_{h \in G_{\mathcal{L}}} \geq e^{(n_{\mathcal{L}} - k) / 2},
     $$
where considering the case $k \leq n_{\mathcal{L}}'$ is enough thanks to the term $C$ in (\ref{eq:stuff2}).  
It now suffices to bound (\ref{eq:MAINBOUND})
through a bound for the right-hand side of (\ref{eq:finalTerm}),  but with our new definitions for $(S_j)_{j\geq 1}$,  $n_{\mathcal L}$,  $n'_{\mathcal L}$ and $G_k$.
For any  $0 \leq j < k$,  we have the following analogue of \ref{Balotprecis}, which is also obtained by  Lemma \ref{le:Bound}:
     \begin{align*} 
       \mathbb{P} & \Big ( \sup_{\substack{|h| \leq 1 \\ |u| \leq e^{- k + 1}}} |(S_{j + 1} - S_j)(h + u) - (S_{j + 1} - S_{j})(h)|\mathbf{1}_{h \in G_{j}} \geq \frac{e^{(n_{\mathcal{L}} - k) / 2}}{4 (k - j)^2} \Big ) \\ 
& \ll (k - j)^{4v} \cdot e^{n_{\mathcal{L}} - j + 10 (j \wedge (n - j)^{3/4}}  e^{- v (n_{\mathcal{L}} - k)} \cdot v! \cdot C^{v} \cdot e^{2v(j - k)} \cdot \frac{y}{j^{3/2}} \cdot e^{\frac{3}{2} \frac{j \log n}{n} - \frac{2 y j}{n} - \frac{y^2 j}{n^2}}.
     \end{align*}
     Pick $v = 100$. 
     Summing over $0 \leq j < k$ we see that the sum is dominated by the contribution of the last term $j = k - 1$, indeed, the sum is
     $$
     \ll e^{- (v - 1) (n_{\mathcal{L}} - k) + 10 (k \wedge (n - k))^{3/4} } \cdot \frac{y}{k^{3/2}} \exp \Big ( \frac{3}{2}\cdot \frac{k \log n}{n} - \frac{2 y k}{n} - \frac{y^2 k}{n^2} \Big ).
     $$
     The contribution of the second term in (\ref{eq:finalTerm}) is bounded similarly to (\ref{eqn:bidule}), and we obtain
     \begin{align*} 
       \mathbb{P} & \Big ( \sup_{\substack{|h| \leq 1 \\ |u| \leq e^{- k + 1}}} |(S_{n_{\mathcal{L}}} - S_k)(h + u) - (S_{n_{\mathcal{L}}} - S_{k})(h)|\mathbf{1}_{h \in G_{k}} \geq \frac{e^{(n_{\mathcal{L}} - k) / 2}}{4} \Big ) \\ 
 & \ll 4^{4v} \cdot e^{n_{\mathcal{L}} - k + 10 (k \wedge (n - k))^{3/4}} \cdot e^{- v (n_{\mathcal{L}} - k)} \cdot v! \cdot C^{v} \cdot (n_{\mathcal{L}} - k)^{v} \cdot \frac{y}{j^{3/2}} \cdot \exp \Big ( \frac{3}{2} \frac{j \log n}{n} - \frac{2 y j}{n} - \frac{y^2 j}{n^2} \Big ).
     \end{align*}
     Choosing $v = 100$ we see that this is also
     $$
     \ll e^{- (v - 2) (n_{\mathcal{L}} - k) + 10 (k \wedge (n - k))^{3/4} } \cdot \frac{y}{k^{3/2}} \exp \Big ( \frac{3}{2}\cdot \frac{k \log n}{n} - \frac{2 y k}{n} - \frac{y^2 k}{n^2} \Big ).
     $$
     Therefore with the new definitions for $(S_j)_{j\geq 1}$,  $n_{\mathcal L}$,  $n'_{\mathcal L}$ and $G_k$,   \eqref{eq:MAINBOUND} is bounded with
     $$
     \ll \sum_{0 \leq k \leq n_{\mathcal{L}'}} e^{- 100 (n_{\mathcal{L}} - k)} \cdot \frac{y}{k^{3/2}} \exp \Big ( \frac{3}{2} \cdot \frac{k \log n}{n} - \frac{2 y k}{n} - \frac{y^2 k}{n^2} \Big ) \ll e^{-50 C} \cdot y e^{-2 y} e^{-y^2 / n}
     $$
     as needed, and where the final gain $e^{-50 C}$ comes from $k \leq n_{\mathcal{L}}' = n_{\mathcal{L}} - \lfloor 2 \log C \rfloor$.  
\end{proof}

\begin{proof}[Proof of Proposition \ref{prop:analogue}]
If there exists an $h \in G_{\mathcal{L}}$, then from Proposition \ref{prop:NewLow}
 \begin{align*}
 \max_{v \in G_0} \min_{|u| \leq 1} (S_{n_{\mathcal{L}}}(v + u) + \sqrt{|u|e^{n_{\mathcal{L}}}}) & \geq \min_{|u| \leq 1} (S_{n_{\mathcal{L}}}(h + u) + \sqrt{|u| e^{n_{\mathcal{L}}}}) \\ & \geq S_{n_{\mathcal{L}}}(h) - 1 \geq n - \frac{3}{4} \log n - 10C,
 \end{align*}
 outside of a set of probability $\ll e^{-50 C} y e^{-2 y} e^{-y^2 / n}$.
 Proposition \ref{prop:NewMax} then implies that outside of a set of $\tau$ of probability $\ll e^{-n}$,
 $$
 \max_{|h| \leq 1} \log |\zeta(\tfrac 12 + \ii \tau + \ii h)| > n - \frac{3}{4} \log n + y - 10 C.
 $$
 In other words,
 $$
 \mathbb{P} (\exists h \in G_{\mathcal{L}} ) - \OO \Big ( e^{-50 C} y e^{-2y} e^{-y^2 / n} \Big ) \leq \mathbb{P} \Big (\max_{|h| \leq 1} \log |\zeta(\tfrac 12 + \ii \tau + \ii h)| > n - \frac{3}{4} \log n + y - 10 C \Big ),
 $$
 and Proposition  \ref{prop:analogue} follows.
\end{proof}

\subsection{Proof of Proposition \ref{prop: 1 thm3} and \ref{prop: 2 thm3}}

\begin{proof}[Proof of Propositions \ref{prop: 1 thm3}]
Clearly, we have
$$
\E[\#G_{n_\mathcal L}]\gg e^{n_\mathcal L}\cdot \PP(S_j\in [L_j,U_j], n_0\leq j\leq n_{\mathcal L}),
$$
where we write $S_j(0)=S_j$ for simplicity.
By the definition of $U_j,L_j$, we have for $j\geq n_0$
$$
U_j-L_j\ll(y-10)-\frac{y}{n}j +(j\wedge (n-j))^{3/4}\ll \Delta_j^{1/4}.
$$
Therefore, the proof of Proposition \ref{lem: 1 point} applies verbatim for all increments $j\geq n_0$.
For the first $n_0$ increments, the approximation in terms of Dirichlet polynomials still holds up to a multiplicative constant (as in  \cite[Equations (31) and (40)]{ArgBouRad2020}, for example). 
These considerations yield
\begin{multline}
\PP(S_j\in [L_j,U_j], n_0\leq j\leq n_{\mathcal L})\\
\gg \PP(\mathcal S_{n_0}\in [L_{n_0}+1,U_{n_0}-1],\mathcal S_{n_0}+\mathcal G_j \in [L_j+1, U_j-1], n_0<j\leq n_{\mathcal L}),\label{lowerBd}
\end{multline}
where  $(\mathcal G_j)_j$ is defined in  \eqref{eqn: Gk} and is independent of  $\mathcal S_{n_0}$, now defined as 
$$
\mathcal S_{n_0}(h) = \sum_{\log\log p\leq  n_0} \re \Big ( e^{\ii \theta_p} \, p^{-(1/2 +\ii h)} + \tfrac 12 \,  e^{2\ii \theta_p} \, p^{-(1 + 2 \ii h)} \Big ).
$$
Note that it differs from \eqref{eqn: steinhaus walk} as it consists in the first $n_0$ increments.
The $\pm 1$ in the barriers will not contribute to the estimate,  we henceforth drop them to lighten the notations.
We now write $f(z)$ for the density of $\mathcal S_{n_0}$,
we condition on $\mathcal G_{n_\mathcal L}$ and apply the Ballot theorem from Proposition \ref{prop:barrier}: The right-hand side of  \eqref{lowerBd} is lower bounded with
\begin{multline*}\int_{L_{n_0}}^{U_{n_0}}  \PP(\mathcal G_j \in [L_j-z, U_j-z], n_0<j\leq n_{\mathcal L})\ f(z)\rd z\\
\gg \int_{L_{n_0}}^{U_{n_0}}\int_{L_{n_\mathcal L}-z}^{U_{n_\mathcal L}-z} \frac{(U_{n_0}-z)(U_{n_\mathcal L}-z-w)}{(n_\mathcal L-n_0)^{3/2}} e^{-\frac{w^2}{n_\mathcal L-n_0}} f(z)\rd w\rd z.
\end{multline*}
Writing $\bar w=w-\alpha(n_\mathcal L-n_0)$ and $\bar z=z-\alpha n_0$, this becomes
$$
\gg \int_{\bar L_{n_0}}^{\bar U_{n_0}}\int_{y-10-\bar z}^{y-\bar z} \frac{(\bar U_{n_0}-\bar z)(y-\bar z-\bar w)}{(n_\mathcal L-n_0)^{3/2}} e^{-(\bar w+\alpha(n_\mathcal L-n_0))^2/(n_\mathcal L-n_0)} f(\bar z+\alpha n_0)\rd \bar w\rd \bar z.
$$
for $\bar L_{n_0}=L_{n_0}+\frac{y}{n_\mathcal L}n_0-n_0^{3/4}$ and $\bar U_{n_0}=U_{n_0}-10\log n_0^{3/4}$.
Expanding the square gives
\begin{equation}\label{eqn:expandSquare}
\frac{(\bar w+\alpha (n_\mathcal L-n_0))^2}{n_\mathcal L-n_0}
=\alpha^2(n_\mathcal L-n_0) +2\alpha \bar w +\frac{\bar w^2}{n_\mathcal L-n_0}
=(n_\mathcal L-n_0)-\frac{3}{2}\log t +2\alpha \bar w +\frac{\bar w^2}{n_\mathcal L-n_0}+\oo(1).
\end{equation}
The integral in $\bar w$ becomes
$$
\int_{y-10-\bar z}^{y-\bar z} (y-\bar z-\bar w)e^{-2\alpha \bar w}e^{-\frac{\bar w^2}{n_\mathcal L-n_0}} \rd \bar w\gg e^{-2\alpha y} e^{2\alpha \bar z}e^{-y^2/n}\gg e^{-2y}e^{-y^2/n} e^{2\alpha \bar z},
$$
by the assumption on $y$. So far, we have shown
$$
\E[\#G_{n_\mathcal L}]\gg e^{n_0}\cdot e^{-2y}e^{-y^2/n}\cdot \int_{\bar L_{n_0}}^{\bar U_{n_0}} (\bar U_{n_0}-\bar z) e^{2\alpha \bar z} f(\bar z+\alpha n_0)\rd \bar z.
$$
From the proof of \cite[Lemma 18]{ArgBouRad2020}, we have $f(u)\ll e^{-u^2/n_0}/\sqrt{n_0}$ uniformly in  $|u|<100 n_0$. This implies
\begin{multline*}
\E[\#G_{n_\mathcal L}]\gg e^{n_0}\cdot e^{-2y}e^{-y^2/n}\cdot \int_{\bar L_{n_0}}^{\bar U_{n_0}} (\bar U_{n_0}-\bar z) e^{2\alpha \bar z} \frac{e^{(-\bar z+\alpha n_0)^2/n_0}}{\sqrt{n_0}}\rd \bar z\\
\gg e^{-2y}e^{-y^2/n} \int_{\bar L_{n_0}}^{\bar U_{n_0}} (\bar U_{n_0}-\bar z) \frac{e^{-\bar z^2/n_0}}{\sqrt{n_0}}\rd \bar z
\gg ye^{-2y}e^{-y^2/n}
\end{multline*}
since the standard deviation of $\bar z$ is $\sqrt{n_0}\asymp \sqrt{y}$ and $\bar U_{n_0}=y-10\log n_0^{3/4}$.
\end{proof}

\begin{proof}[Proof of Proposition \ref{prop: 2 thm3}]
Proceeding as in Proposition \ref{prop: 1 thm3}, the estimate is reduced to 
$$
\E[(\#G_{n_\mathcal L})^2]\ll \sum_{h,h'\in G_0} \PP(\mathfrak S(h)\cap \mathfrak S(h')), 
$$
where
$$
\mathfrak S(h)=\{\mathcal S_{n_0}(h)\in [L_{n_0}-1,U_{n_0}+1],\mathcal S_{n_0}(h)+\mathcal G_j(h)\in [L_j-1, U_j+1], n_0<j\leq n_{\mathcal L}\}.
$$
Again, since the $\pm 1$ will not contribute to the estimates, we omit them from the notations. 
We write $\mathcal S_{n_0}(h)=\mathcal S_{n_0}$, $\mathcal S_{n_0}(h')=\mathcal S_{n_0}'$ and similarly for $\mathcal G$. 
We condition on the pair $(\mathcal S_{n_0},\mathcal S_{n_0}')$ to get
$$
 \PP(\mathfrak S(h)\cap \mathfrak S(h'))=
 \int_{[L_{n_0},U_{n_0}]^{2}} \PP(\mathcal G_j\leq U_j-z, \mathcal G_j'\leq U_j-z', n_0\leq j\leq n_\mathcal L) f(z,z')\rd z\rd z',
$$
where $f$ now stands for the density of $(\mathcal S_{n_0},\mathcal S_{n_0}')$.
The estimate depends on the {\it branching time} $j^\star=j^\star(h,h')=\lfloor\log |h-h'|^{-1}\rfloor$. We split into two cases ($j^\star\leq n_0$ and $j^\star> n_0$),
contrary to the proof of Proposition (\ref{prop: comp second}) which needs three cases as it requires matching of first and second moments up to $1+\oo(1)$ precision.\\

\noindent{\bf Case $j^\star\leq n_0$}.  
 In this case, the decoupling Lemma \ref{le:gaussian} can be applied to all increments. The probability in the integral is then
$$
 \ll\int_{[L_{n_0},U_{n_0}]^{2}} \PP( \widetilde{\mathcal G}_{j}\leq U_j-z, n_0\leq j\leq n_\mathcal L) \PP(\widetilde{\mathcal G}'_{j}\leq U_j-z', n_0\leq j\leq n_\mathcal L)  f(z,z')\rd z\rd z'
$$
where we recall that $\widetilde{\mathcal G}_{j}=\sum_{i\leq j}\widetilde{\mathcal N}_j$ and the independent  Gaussian centered $\widetilde{\mathcal N}_j$'s have variance $\mathfrak{s}_j^2+|\rho_j|$. 

After conditioning on $(\widetilde{\mathcal G}_{n_\mathcal L},\widetilde{ \mathcal G}'_{n_\mathcal L})$, the Ballot theorem from Proposition \ref{prop:barrier} can be applied to each term. The above becomes
$$
\ll \int_{[L_{n_0},U_{n_0}]^{2}}\int_{[L_{n_\mathcal L},U_{n_\mathcal L}]^{2}}\tfrac{(U_{n_0}-z)(U_{n_\mathcal L}-z-w)(U_{n_0}-z')(U_{n_\mathcal L}-z'-w')}{(n_\mathcal L-n_0)^3}  e^{-\frac{w^2+w'^2}{n_\mathcal L-n_0}}f(z,z')\rd w\rd w'\rd z\rd z'.
$$
The Gaussian density can be expanded as in \eqref{eqn:expandSquare}. The integral in $w, w'$ gives a contribution
$
\OO(e^{-2n_\mathcal L}e^{2n_0}e^{-4y}e^{-2y^2/n}).
$
There are $\OO(e^{2n_\mathcal L})$ pairs $h,h'$ with $j^\star\leq n_0$,  so
$$
\begin{aligned}
&\sum_{h,h': j^\star\leq n_0}  \PP(\mathfrak S(h)\cap \mathfrak S(h))\\
&\hspace{1cm}\ll  e^{2n_0} e^{-4y} \int_{[\bar L_{n_0},\bar U_{n_0}]^{2}}(\bar U_{n_0}-\bar z)(\bar U_{n_0}-\bar z')e^{2\alpha(\bar z+\bar z')}f(\bar z+\alpha n_0,\bar z'+\alpha n_0)\rd \bar z\rd \bar z'\\
&\hspace{1cm}\ll y^2e^{2n_0} e^{-4y} \int_{[\bar L_{n_0},\bar U_{n_0}]^{2}}e^{2\alpha(\bar z+\bar z')}f(\bar z+\alpha n_0,\bar z'+\alpha n_0)\rd \bar z\rd \bar z',
\end{aligned}
$$
where we used the barrier range to bound  $|\bar U_{n_0}-\bar z|\leq y$.  Using  the Cauchy-Schwarz inequality and recalling that  $f(u)\ll e^{-u^2/n_0}/\sqrt{n_0}$,  as $n_0=y/10$, this is
$$
\ll y^2e^{2n_0} e^{-4y}\cdot  \int_{-\infty}^{\infty} e^{4 \bar z} \frac{e^{-\bar z^2/n_0}}{\sqrt{n_0}}\ll y^2e^{2n_0} e^{-4y}\cdot e^{8 n_0}\ll y e^{-2y-y^2/n}.
$$

\noindent{\bf Case $j^\star>n_0$.} 
We proceed similarly to the proof of the left tail and consider the center of mass and the difference between the two Gaussian walks as in \eqref{eqn:orth}.
We index the value of $\overline{\mathcal G}_{j^\star}$ by $v_1$, the values $\mathcal G_{j^\star}^\perp$ by $q$, and  the values of 
two independent copies of
$\widetilde{\mathcal{G}}_{j^{\star},n_{\mathcal L}}$
by $v_2$ and $v_3$. 
Proceeding exactly as for Equation \eqref{eqn: prob 3rd}, i.e.,  using Lemma \ref{le:gaussian} for the increments after $j^*$, we obtain
that $\PP(\mathfrak S(h)\cap \mathfrak S(h'))$ is
$$
\begin{aligned}
&\sum_{q\in \mathbb Z}
\int_{[L_{n_0},U_{n_0}]^{2}} 
\int_{L_{j^\star}}^{U_{j^\star}}
\frac{(U_{n_0}-z_m)(U_{j^\star}-z_m-v_1)}{(j^\star-n_0)^{3/2}}\cdot e^{-cq^2}f(z,z')e^{-\frac{v_1^2}{j^\star-n_0}}\\
&\times \PP(\widetilde{\mathcal G}_{j^\star,j}+v_1+q\leq U_j-z, j>j^\star)\cdot
\PP(\widetilde{\mathcal G}_{j^\star, j}+v_1-q\leq U_j-z', j>j^\star) \rd v_1\rd z\rd z'
\end{aligned}
$$
where we applied Proposition \ref{prop:barrier} for $\bar{\mathcal G}$ between  $n_0$ and $j^\star$ and where we denoted $z_m=\frac{z+z'}{2}=\bar{\mathcal G}_{n_0}$. 
This Ballot theorem can also be applied to the two probabilities in the integral giving, after summing over $q$,
$$
\ll \frac{(U_{j^\star}-z_m-v_1)^2(U_{n_\mathcal L}-z-v_1-v_2)(U_{n_\mathcal L}-z'-v_1-v_3)}{(n_\mathcal L-j^\star)^3} e^{-\frac{v_2^2+v_3^2}{n_\mathcal L -j^\star}}.
$$
After expanding the squares, the densities of $v_1$, $v_2$, $v_3$ become
$$
e^{-2n_\mathcal L+j^\star} e^{n_0} \frac{n^{\frac{3}{2}\frac{j^\star}{t}}n^{3(1-\frac{j^\star}{n})}}{(j^\star-n_0)^{3/2}(n_\mathcal L-j^\star)^3} e^{2\alpha\overline{v_1}-\frac{\bar v_1^2}{j^\star-n_0}-2\alpha(\overline{v_1}+\overline{v_2})-2\alpha(\overline{v_1}+\overline{v_3})},
$$
for $\overline{v_1}=v_1-\alpha(j^\star-n_0)$, $\overline{v_i}=v_i-\alpha(n_\mathcal L-j^\star)$, $i=2,3$. 
The integral over $\overline{v_1}+\overline{v_2}\in [y-10-\bar z, y-\bar z]$ is 
$$
\int_{y-10-\bar z}^{ y-\bar z}(y-\bar z-\overline{v_1}-\overline{v_2})e^{-2\alpha(\overline{v_1}+\overline{v_2})}\rd \overline{v_1}\rd \overline{v_2}\ll e^{2\alpha\bar z-2y}.
$$
The integral over $\overline{v_1}+\overline{v_3}\in [y-10-\bar z', y-\bar z']$ is the same and contributes $\ll e^{2\alpha\bar z'}e^{-2y}$.
The integral over $\overline{v_1}$ is, for $\bar z_m=\frac{\bar z+\bar z'}{2}$,
$$
\ll \int_{\bar L_j^\star}^{\bar U_{j^\star}-\bar z_m} ( \bar U_{j^\star}-\bar z_m -\overline{v_{1}})^2 e^{2\alpha\overline{v_{1}}} e^{-\frac{\bar v_1^2}{j^\star-n_0}}\rd \overline{v_{1}}\ll \frac{e^{2y-2\alpha \bar z_m-y^2/n}}{(j^\star\wedge (n-j^\star))^{10}},
$$
where $\bar U_{j^\star}=y-10 \log (j^\star\wedge (n-j^\star))$ and  $\bar L_{j^\star}=-10- (j^\star\wedge (n-j^\star))^{3/4}$.

We now sum over all $j^*>n_0$ and the $\OO(e^{-2n_\mathcal L+j^\star})$ pairs with a given $j^\star$,  so that
\begin{multline*}
\sum_{h,h': n_0<j^\star<n_{\mathcal L}}  \PP(\mathfrak S(h)\cap \mathfrak S(h))
\ll e^{-2y-y^2/n+n_0}
\int_{\bar L_{n_0}}^{\bar U_{n_0}} (\bar U_{n_0}-\bar z_m) e^{2\alpha\bar z_m}f(\bar z+\alpha n_0, \bar z'+\alpha n_0)\rd z\rd z'\\
\times \sum_{ n_0<j^\star<n_{\mathcal L}} \frac{1}{(j^\star\wedge (n-j^\star))^{10}} \frac{n^{\frac{3}{2}\frac{j^\star}{n}}n^{3(1-\frac{j^\star}{n})}}{(j^\star-n_0)^{3/2}(n_\mathcal L-j^\star)^3}.
\end{multline*}
The integral over $z,z'$ is over a function of $\bar z_m$ only,  which has density $\ll e^{-u^2/n_0}/\sqrt{n_0}$ uniformly in  $|u|<100 n_0$. 
Moreover, we can simply bound $|\bar U_{n_0}-\bar z_m|\leq y$,  hence
$$
\sum_{h,h': n_0<j^\star<n_{\mathcal L}}  \PP(\mathfrak S(h)\cap \mathfrak S(h))\ll ye^{-2y-y^2/n+n_0}\int_{\bar L_{n_0}}^{\bar U_{n_0}}  e^{2\bar z_m}\frac{e^{-(\bar z_m+\alpha n_0)^2/n_0}}{\sqrt{n_0}}\rd {\bar z}_m\ll y e^{-2y}e^{-y^2/n},
$$
which  concludes the proof of Proposition \ref{prop: 2 thm3}. 
\end{proof}

\appendix

\section{Some Auxiliary Results}\label{appendixA}

Let $(Z_p, p \text{ prime})$ a sequence of independent and identically distributed random variables, uniformly distributed on the unit circle $|z| = 1$. For an integer $n$ with prime factorization $n = p_1^{\alpha_1} \ldots p_k^{\alpha_k}$ with $p_1, \ldots, p_k$ all distinct, consider 
$$
Z_n := \prod_{i = 1}^{k} Z_{p_i}^{\alpha_i}. 
$$
Then we have $\mathbb{E}[Z_n \overline{Z}_m] = \mathbf{1}_{n = m}$, and therefore, for an arbitrary sequence $a(n)$ of complex numbers, the following holds
$$
\sum_{n \leq N} |a(n)|^2 = \mathbb{E} \Big [ \Big | \sum_{n \leq N} a(n) Z_n \Big |^2 \Big ]. 
$$
The next lemma shows that the mean value of Dirichlet polynomial is close to the one of the above random model. It follows directly from \cite[Corollary 3]{MonVau1974}.

\begin{lemma}[Mean-value theorem for Dirichlet polynomials] \label{lem: Transition}
  We have,
$$
  \mathbb{E} \Big [ \Big | \sum_{n \leq N} a(n) n^{\ii \tau} \Big |^2 \Big ]  = \Big (1 + \OO \Big ( \frac{N}{T} \Big ) \Big ) \sum_{n \leq N} |a(n)|^2  = \Big ( 1 + \OO \Big ( \frac{N}{T} \Big ) \Big ) \mathbb{E} \Big [ \Big | \sum_{n \leq N} a(n) Z_n \Big |^2 \Big ].
$$
  \end{lemma}

\begin{lemma}[Exponential moments for the probabilistic model,  Lemma 15 in \cite{ArgBouRad2020}]\label{lem:basicFour} Remember the definition \eqref{eqn: steinhaus walk}.
There exists an absolute $C>0$ such that for any $\lambda\in\mathbb{R}$ and $n_0\leq j\leq k$ we have 
$$
\E\Big[\exp\big(\lambda(\mathcal S_{k}(h)-\mathcal S_{j}(h) )\big)\Big]\leq \exp((k-j+C)\lambda^2/4).
$$
\end{lemma}

\begin{lemma}[Gaussian moments of Dirichlet polynomials, Lemma 16 in \cite{ArgBouRad2020}]
\label{lem: Gaussian moments}
For any $h \in [-1,1]$ and integers $k,j,q$ satisfying $n_0\leq j\leq k$,  $2q \leq e^{n-k} $, we have
\begin{align}\label{eqn:mom2}
\mathbb{E} [ |S_k(h) - S_{j}(h)|^{2q} ] &\ll \frac{(2q)!}{2^q q!} \, \Big(\frac{k - j}{2}\Big)^{q}.
\end{align}
Moreover, there exists $C>0$ such that for any $0 \leq j \leq k$,  $2q \leq e^{n-k} $, we have
\begin{equation}\label{eqn:mom3}
\mathbb{E} [ |S_k(h) - S_{j}(h)|^{2q} ] \ll q^{1/2} \frac{(2q)!}{2^q q!} \, \Big(\frac{k - j+C}{2}\Big)^{q}. 
\end{equation}
\end{lemma}

\begin{lemma}[Gaussian moments of Dirichlet polynomials,  Lemma 3 of \cite{Sou2009}]\label{lem:Sound}
Let $2\leq x\leq T$ and $q\in \mathbb N$ with $x^q\leq T/\log T$. For any complex numbers $a(p)$, we have
$$
\E\Big[\Big| \sum_{p\leq x} \frac{a(p)}{p^{1/2+\ii \tau}}\Big|^{2q} \Big]\ll q! \Big(\sum_{p\leq x} \frac{|a(p)|^2}{p}\Big)^q.
$$
\end{lemma}

\begin{lemma}
\label{lem:compProb}
Let $h,h'\in[-1,1]$. Consider the increments $(\mathcal Y_k(h),  \mathcal Y_k(h'))$ for $1\leq k\leq n_\mathcal L$, and the corresponding Gaussian vector $(\mathcal N_k(h),  \mathcal N_k(h'))$,
of mean $0$ and with the covariance given by (\ref{eqn:sk}), (\ref{eqn:rhok}).
There exists a constant $c>0$ such that, for any intervals $A, B$ and $k\geq 1$,
$$
{\mathbb P}\Big((\mathcal Y_k(h),  \mathcal Y_k(h')) \in A\times B\Big)
=\mathbb P\Big((\mathcal N_k(h),  \mathcal N_k(h'))\in A\times B\Big)+\OO(e^{-c e^{k/2}}).
$$
\end{lemma}
\begin{proof}
This follows similaly to \cite[Lemma 20]{ArgBouRad2020}, based on the Berry-Esseen estimate as stated in \cite[Lemma 19]{ArgBouRad2020}. The proof is actually more immediate because the covariances of $(\mathcal Y,\mathcal Y')$ and $(\mathcal N,\mathcal N')$ exactly coincide.
\end{proof}

  \begin{lemma} \label{lem: discretization}
  Let $D$ be a Dirichlet polynomial of length $\leq N$.
  Then, for any $1\leq k\leq \log\log N$, we have
   \begin{equation}
 \begin{aligned}
  \label{eqn: discretize}
  \max_{|h| \leq e^{-k}} |D(\tfrac 12 + \ii  t + \ii h)|^2\ll
 & \sum_{|j| \leq 16 e^{-k}\log N}  \Big | D \Big (\frac 12 +\ii t + \frac{2\pi \ii j}{8 \log N} \Big ) \Big |^{2}  \\ & + \sum_{|j| > 16 e^{-k} \log N} \frac{1}{1 + |j|^{100}} \Big |D \Big ( \frac 12 + \ii t + \frac{2\pi \ii j}{8 \log N} \Big ) \Big |^{2}.
\end{aligned}
\end{equation}
\end{lemma}

\begin{proof}
We proceed similarly to the proof of \cite[Lemma 27]{ArgBouRad2020}, but now with maxima on intervals of general length $e^{-k}$.  With the notations from \cite[Lemma 25]{ArgBouRad2020},  we have
$$
D(\tfrac 12 + \ii t + \ii h_0)^2=\frac{1}{2+\varepsilon}\sum_{h\in\frac{2\pi\mathbb{Z}}{(2+\varepsilon)\log N}}D(\tfrac 12 +\ii t+\ii h)^2 \widehat{V}\big(\frac{(h-h_0)\log N}{2\pi}\big).
$$
Using the triangle inequality and the decay $\widehat{V}(x)\ll_A(1+|x|)^{-A}$ we obtain the result.
\end{proof}

\section{Ballot Theorem} \label{se:GaussianRW}

\subsection{Results}\ Most ideas for the results in this section are due to Bramson.  As we could not find the exact barrier estimates needed in our setting,  this section gives a self-contained and quantitative analogues of some technical results in  \cite{Bra1978,Bra1983} in the setting of Gaussian random walk with arbitrary, comparable,  variance of the increments.\\

\noindent Let $\kappa>0$ be fixed in all this section,  and $(X_i)_{i\geq 1}$ be independent,  real, centered Gaussian randon variables such that $\kappa<\E[|X_i|^2]<\kappa^{-1}$ for all $i$.  For $k\in\mathbb{N}$ we denote $S_k=\sum_{i\leq k} X_i$.

We denote $\PP_{(s,x)}$ for the distribution of the process $(S_k)_{k}$ starting at time $s$ from $x$,  $\PP_{x}=\PP_{(0,x)}$,  $\PP=\PP_0$, and $\PP_{(s,x)}^{(t,y)}$ for the distribution for $(S_k)_{k}$ starting at time $s$ from $x$, and conditioned to end at time $t$ at point $y$. \\

\begin{prop}\label{prop:barrier}
Let $\delta>1/2>\alpha>0$.    Then there exists $c=c(\alpha,\delta,\kappa)$ such that uniformly in the time
$t\geq 1$, $10\leq y\leq t^{1/10}$, 
$a,b\in [1,y-1]$ and uniformly in the functions 
$v_s\geq y+\min(s,t-s)^{\delta}$,  $|u_s|\leq \min(s,t-s)^{\alpha}$,
 we have
$$
\PP_{(0,a)}^{(t,b)}\left( \cap_{0\leq k\leq t}\{u_k\leq S_k\leq v_k\}\right)=
\frac{2ab}{\sigma}\cdot \left(1+\OO_{\alpha,\delta,\kappa}(d^{-c})\right)
$$
where $d=\min(|y-a|,|y-b|,|a|,|b|)$  and $\sigma=\sum_{k\leq t}\E[X_k^2]$.
\end{prop}

\subsection{Preliminaries on Brownian motion}
We denote $\PP_{(s,x)}$ for the distribution of the Brownian motion starting at time $s$ from $x$,  $\PP_{x}=\PP_{(0,x)}$,  $\PP=\PP_0$, and $\PP_{(s,x)}^{(t,y)}$ for the distribution for the Brownian bridge starting at time $s$ from $x$, ending at time $t$ at point $y$.  Context will avoid confusion with the notation $\PP$ from Proposition \ref{prop:barrier} as the Gaussian random walk will always be denoted $S$, and the Brownian motion $B$.

For such a trajectory $B$, let  $M_t=\max_{0\leq s\leq t}B_s$,  $m_t=\min_{0\leq s\leq t}B_s$.

\begin{lemma}\label{lem:explicit}
Let $x,y>0$. Then 
$$
\PP_{(0,x)}^{(t,y)}\left(m_t\geq 0\right)=1-e^{-\frac{2xy}{t}}.
$$
\end{lemma}

\begin{proof}
From the reflection principle,  for any measurable $A\subset(0,\infty)$,
$$
\PP_x\left(m_t\geq 0,B_t\in A\right)=\PP_x\left(B_t\in A\right)-\PP_x\left(B_t\in -A\right).
$$
This implies
$$
\PP_{(0,x)}^{(t,y)}\left(m_t\geq 0\right)=1-\lim_{\varepsilon\to 0}\frac{\PP_x\left(B_t\in -[y,y+\e]\right)}{\PP_x\left(B_t\in [y,y+\e]\right)}=1-e^{-\frac{2xy}{t}},
$$
concluding the proof.
\end{proof}

\begin{lemma}\label{lem:reflection}
Let $a,c>0$ and $A\subset [-c,a]$ be measurable. Then
$$
\PP\left(M_t\leq a,m_t\geq -c,B_t\in A\right)\geq \PP\left(m_t\geq -c,B_t\in A\right)-\PP\left(B_t\in A-2a\right).
$$
\end{lemma}

\begin{proof}
The above left-hand side is
$$
\PP\left(m_t\geq -c,B_t\in A\right)-\PP\left(M_t\geq a,m_t\geq -c,B_t\in A\right)
\geq \PP\left(m_t\geq -c,B_t\in A\right)-\PP\left(M_t\geq a,B_t\in A\right).
$$
From the reflection principle,  this last probability  is $\PP\left(B_t\in 2a-A\right)$.
\end{proof}

\begin{lemma}\label{lem:BM1}
Let $\delta>1/2$,  $v_s\geq y+\min(s,t-s)^{\delta}$.  Let $c\in(0,2-\frac{1}{\delta})$.  Then uniformly in
$t\geq 0$, $2\leq y\leq t^{1/10}$,  $a,b\in [1,y-1]$ we have
$$
\PP_{(0,a)}^{(t,b)}\left( \cap_{0\leq s\leq t}\{0\leq B_s\leq v_s\}\right)=\PP_{(0,a)}^{(t,b)}\left(m_t\geq 0\right)\cdot \left(1+\OO(e^{-\min(|y-a|,|y-b|)^c})\right).
$$
\end{lemma}

\begin{proof}
In this proof we abbreviate $B\geq 0$ for $m_t\geq 0$ and start with
\begin{align*} 
&\PP_{(0,a)}^{(t,b)}\left(B\geq 0,\exists s:B_s>v_s\right)\leq \\
& \sum_{k=0}^{t/2}\left(\PP_{(0,a)}^{(t,b)}\left(B\geq 0,\exists s\in[k,k+1]:B_s>v_k\right)
+\PP_{(0,b)}^{(t,a)}\left(B\geq 0,\exists s\in[k,k+1]:B_s>v_k\right).
\right)
\end{align*}
The first probability above is smaller than
$$
\PP_{(0,a)}^{(t,b)}\left(B\geq 0,\exists s\in[0,k+1]:B_s>v_k\right)
=\PP_{(0,a)}^{(t,b)}\left(B\geq 0\right)-\PP_{(0,a)}^{(t,b)}\left(B\geq 0,\max_{[0,k+1]}B<v_k\right).
$$

We write
\begin{multline*}
\PP_{(0,a)}^{(t,b)}\left(B\geq 0,\max_{[0,k+1]}B<v_k\right)\\=
\int_0^{v_k}\PP_{(0,a)}^{(k+1,x)}\left(B\geq 0,\max_{[0,k+1]}B<v_k\right)
\PP_{(k+1,x)}^{(t,b)}\left(B\geq 0\right) \PP_{(0,a)}^{(t,b)}(B_{k+1}\in\rd x).
\end{multline*}
The first probability in this integral is estimated with Lemma \ref{lem:reflection}:
$$
\PP_{(0,a)}^{(k+1,x)}\left(B\geq 0,\max_{[0,k+1]}B<v_k\right)\geq  \PP_{(0,a)}^{(k+1,x)}\left(B\geq 0\right)-e^{-\frac{2(v_k-a)(v_k-x)}{k+1}}.
$$
This gives 
\begin{multline*}
\PP_{(0,a)}^{(t,b)}\left(B\geq 0,\exists s\in[k,k+1]:B_s>v_k\right)\\ \leq
\int_0^{v_k}e^{-\frac{2(v_k-a)(v_k-x)}{k+1}}
\PP_{(k+1,x)}^{(t,b)}\left(B\geq 0\right) \PP_{(0,a)}^{(t,b)}(B_{k+1}\in\rd x).
\end{multline*}
From Lemma \ref{lem:explicit}, we have $\PP_{(k+1,x)}^{(t,b)}\left(B\geq 0\right)\leq 5\frac{xb}{t}$.
Moreover,   $\PP_{(0,a)}^{(t,b)}(B_{s}\in\rd x)=\PP_{(0,0)}^{(t,0)}(B_{s}\in\rd x+x_s)=\frac{e^{-(x-x_s)^2/(2w_s)}}{\sqrt{2\pi w_s}}\rd x$ where $w_s=s(t-s)/t$, $x_s=(1-\frac{s}{t})a+\frac{s}{t}b$.
This gives
$$
\PP_{(0,a)}^{(t,b)}\left(B\geq 0,\exists s\in[k,k+1]:B_s>v_k\right)\leq C
\frac{b}{t}\int_0^{v_k}e^{-\frac{2(v_k-a)(v_k-x)}{k+1}}
x \frac{e^{-\frac{(x-x_{k+1})^2}{2w_{k+1}}}}{\sqrt{w_{k+1}}}
\rd x.
$$
In the above integral,  the contribution from $|x-x_{k+1}|>v_k/3$ is bounded with
\begin{equation}\label{eqn:error1}
\int_{|x-x_{k+1}|>v_k/3}(x_{k+1}+|x-x_{k+1}|)\frac{e^{-\frac{(x-x_{k+1})^2}{2w_{k+1}}}}{\sqrt{w_{k+1}}}\rd x\leq C x_{k+1} e^{-\frac{v_k^2}{100w_{k+1}}}.
\end{equation}
The regime $|x-x_{k+1}|<v_k/3$ gives the error 
\begin{equation}\label{eqn:error2}
\int_{|x-x_{k+1}|<v_k/3}e^{-\frac{(k^\delta+|y-a|)k^\delta}{k+1}}x\frac{e^{-\frac{(x-x_{k+1})^2}{2w_{k+1}}}}{\sqrt{w_{k+1}}}\rd x\leq C x_{k+1} e^{-\frac{(k^\delta+|y-a|)k^\delta}{k+1}}.
\end{equation}
We first bound the sum of the error terms from (\ref{eqn:error2}) as $0\leq k\leq t/2$. 
For $k^\delta<|y-a|$,  from the hypothesis $y<t^{1/10}$ we have $x_{k+1}<a+1<2a$, so that
$$
\sum_{k^\delta<|y-a|} x_{k+1} e^{-\frac{(k^\delta+|y-a|)k^\delta}{k+1}}\leq 2a \sum_{k^\delta<|y-a|}  e^{-\frac{|y-a|k^\delta}{k+1}}\leq 2a |y-a|^{1/\delta}e^{-|y-a|^{2-\frac{1}{\delta}}}=a\,\OO(e^{-|y-a|^c}).
$$
For $k^\delta>|y-a|$, we obtain
\begin{multline*}
\sum_{k\geq |y-a|^{1/\delta}}x_k e^{-k^{2\delta-1}}\leq \sum_{k\geq |y-a|^{1/\delta}}(a+\frac{kb}{t})e^{-k^{2\delta-1}}
\leq (a+1)\sum_{k\geq |y-a|^{1/\delta}}k e^{-k^{2\delta-1}}\\
\leq C_\delta(a+1)\int_{v>|y-a|^{2-\frac{1}{\delta}}}v^{\frac{3-2\delta}{2\delta-1}} e^{-v}\rd v=a\,\OO(e^{-|y-a|^c}).
\end{multline*}
The same estimate can be obtained for the sum over $k$ from (\ref{eqn:error1}).  We have thus obtained
$$
\PP_{(0,a)}^{(t,b)}\left( \cap_{0\leq s\leq t}\{0\leq B_s\leq v_s\}\right)=\PP_{(0,a)}^{(t,b)}\left(m_t\geq 0\right)+\OO\left(\frac{ab}{t}e^{-\min(|y-a|,|y-b|)^c}\right).
$$
The result follows from the above estimate and Lemma \ref{lem:explicit}.
\end{proof}

\begin{lemma}\label{lem:BM2}
Let $\delta>1/2>\alpha>0$.  Then there exists $c=c(\alpha,\delta)$, such that, uniformly in
$t\geq 0$, $y\geq 10$,  $a,b\in [1,y-1]$,  and uniformly in the functions $v_s\geq y+\min(s,t-s)^{\delta}$,  $|u_s|\leq \min(s,t-s)^{\alpha}$, we have
$$
\PP_{(0,a)}^{(t,b)}\left( \cap_{0\leq s\leq t}\{u_s\leq B_s\leq v_s\}\right)\geq \PP_{(0,a)}^{(t,b)}\left(m_t\geq 0\right)\cdot 
\left(1+\OO(d^{-c})\right),
$$
where $d=\min(|y-a|,|y-b|,|a|,|b|)$.
\end{lemma}

\begin{proof}
Without loss of generality we can assume $\alpha+\delta<1$. We also pick $\varepsilon\in(0,1)$ and write $d_0=d^{1-\varepsilon}$. Let $s_1,s_2$ be the solutions of $s_1^\alpha=d_0$, $(t-s_2)^\alpha=d_0$.
Let 
\begin{equation}\label{eqn:ubar}
\bar u_s=(d_0+(s-s_1)\alpha s_1^{\alpha-1})\mathbf{1}_{s^\alpha<d_0}+(d_0-(s-s_2)\alpha (t-s_2)^{\alpha-1})\mathbf{1}_{(t-s)^\alpha<d_0}+\min(s,t-s)^\alpha\mathbf{1}_{s_1<s<s_2}.
\end{equation}
In other words,  $\bar u$ is the function coinciding with $\min(s,t-s)^\alpha$ on $(s_0,s_1)$ and linearly expanded on the complement, with continuous derivative. Note that $\bar u_0=\bar u_t=(1-\alpha) d_0$. We also denote
$\bar v_s=(1-\frac{s}{t})a+\frac{s}{t}b+\min(s,t-s)^{\delta}+d$.  We have
$$
\PP_{(0,a)}^{(t,b)}\left( \cap_{0\leq s\leq t}\{u_s\leq B_s\leq v_s\}\right)\geq \PP_{(0,a)}^{(t,b)}\left( \cap_{0\leq s\leq t}\{\bar u_s\leq B_s\leq \bar v_s\}\right).
$$
The Cameron-Martin formula gives
$$
\PP_{(0,a)}^{(t,b)}\left( \cap_{0\leq s\leq t}\{\bar u_s\leq B_s\leq \bar v_s\}\right)
=
\E_{(0,a)}^{(t,b)}\left[
e^{-\int_0^t \dot {\bar u}_s\rd B_s-\frac{1}{2}\int_0^t\dot {\bar u}_s^2\rd s}\mathbf{1}_{\cap_{0\leq s\leq t}\{d^{1-\varepsilon}\leq B_s\leq \bar v_s-\int_0^u \dot {\bar u}_u\rd u\}}
\right],
$$
where we denote the derivative in $s$ of $f$ by $\dot f$.
We now bound both terms in the measure bias,  deterministically.   First, 
$$
\int_0^t \dot {\bar u}_s^2\rd s\leq C \int_{s>s_1} s^{2\alpha-2}\rd s+C\int_{s<s_1}s_1^{2\alpha-2}\rd s\leq C d^{-(\frac{1}{\alpha}-2)(1-\varepsilon)}.
$$
Moreover,  by integration by parts  we have (using the fact that $\bar u$ has continuous derivative)
$$
-\int_0^t \dot {\bar u}_s\rd B_s=\int_0^t B_s \ddot {\bar u}_s\rd s-B_t \dot {\bar u}_t+B_0 \dot {\bar u}_0=\int_0^t \big(B_s-((1-\frac{s}{t})B_0+\frac{s}{t}B_t)\big) \ddot {\bar u}_s\rd s.
$$
On the set $\cap_{0\leq s\leq t}\{B_s\leq {\bar v}_s\}$, we therefore have the deterministic bound
$$
-\int_0^t \dot {\bar u}_s\rd B_s\geq- \int_{s^\alpha>d^{1-\varepsilon}}(d^{1-\varepsilon}+s^\delta)s^{\alpha-2}\rd s\geq -C (d^{-(\frac{1}{\alpha}-2)(1-\varepsilon)}+d^{(\frac{\delta}{\alpha}+1-\frac{1}{\alpha})(1-\varepsilon)}).
$$
We have therefore proved
\begin{multline*}
\PP_{(0,a)}^{(t,b)}\left( \cap_{0\leq s\leq t}\{\bar u_s\leq B_s\leq \bar v_s\}\right)\\\geq 
\exp\left(\OO(d^{-(\frac{1}{\alpha}-2)(1-\varepsilon)}+d^{-\frac{1-\alpha-\delta}{\alpha}(1-\varepsilon)})\right)
\PP_{(0,a)}^{(t,b)}\left( \cap_{0\leq s\leq t}\{d^{1-\varepsilon}\leq B_s\leq v_s-\int_0^s\dot {\bar u}_u\rd u\}\right).
\end{multline*}
The desired lower bound follows by using Lemma \ref{lem:BM1}, noting that $\frac{(a-d^{1-\varepsilon})(b-d^{1-\varepsilon})}{t}=\frac{ab}{t}(1+\OO(d^{-\varepsilon}))$. 
\end{proof}

\subsection{Proofs of barrier estimates for the random walk, Proposition \ref{prop:barrier}.}
The lower bound is a direct consequence of Lemma \ref{lem:BM2} and Lemma \ref{lem:explicit}.

For the upper bound,  we only need to bound
$\PP_{(0,a)}^{(t,b)}\left( \cap_{0\leq k\leq t}\{-\bar u_k\leq S_k\}\right)$,
where $\bar u$ is defined in \eqref{eqn:ubar}.
By the change of variables $S_k=\tilde S_k-\sum_{0\leq i\leq k-1}(\bar u_{i+1}-\bar u_i)$, and denoting $d_1=(1-\alpha)d^{1-\varepsilon}$, we obtain 
\begin{multline*}
 \PP_{(0,a)}^{(t,b)}\left( \cap_{0\leq k\leq t}\{-\bar u_k\leq S_k\}\right)
 =
  \E_{(0,a)}^{(t,b)}\left[ e^{-\frac{1}{2}\sum_k(\bar u_{k+1}-\bar u_k))^2+\sum_k( S_{k+1}- S_{k})(\bar u_{k+1}-\bar u_k)}\mathbf{1}_{\cap_{0\leq k\leq t}\{-d_1\leq  S_k\}}\right]\\
\leq
  \E_{(0,a)}^{(t,b)}\left[ e^{\sum_k( S_{k+1}- S_{k})(\bar u_{k+1}-\bar u_k)}\mathbf{1}_{\cap_{0\leq k\leq t}\{-d_1\leq  S_k\}}\right].
\end{multline*}
Let $\bar S_s=S_s-(a\frac{t-s}{t}+b\frac{s}{t})$. We have
$$
\sum_k( S_{k+1}- S_{k})(\bar u_{k+1}-\bar u_k)=
\sum_k(\bar S_{k+1}-\bar S_{k})(\bar u_{k+1}-\bar u_k)=\sum_k a_k \bar S_k,
$$
where 
the constants $a_k$ satisfy $ 0\leq a_k\leq 10 \min(k,t-k+1)^{\alpha-2}$ and vanish outside $[d_2,t-d_2]$ where we define $d_2=d^{\frac{1-\varepsilon}{\alpha}}$.
As $ab\leq \max(a^2,b^2)$, we have obtained
$$
\PP_{(0,a)}^{(t,b)}\left( \cap_{0\leq k\leq t}\{-\bar u_k\leq S_k\}\right)\leq
\E_{(0,a)}^{(t,b)}\left[e^{2\sum_{k\leq t/2} a_k\bar S_k}
\mathbf{1}_{\cap_{0\leq k\leq t}\{-d_1\leq  S_k\}}\right].
$$
Let $\varepsilon_0\in(0,\frac{1}{2}-\alpha)$ and define $\kappa=\frac{1-\varepsilon}{2\alpha}(\frac{1}{2}-\alpha-\varepsilon_0)$. Note that for any integer $v\geq 1$, $\sum_{k\leq t/2} a_k\bar S_k>vd^{-\kappa}$ implies that there exists $k$ such that $d_2\leq k\leq t/2$ such that 
$\bar S_k> v  k^{\frac{1}{2}+\varepsilon_0}$. This observation together with the union bound gives
\begin{align}
&\E_{(0,a)}^{(t,b)}\left[e^{2\sum_{k\leq t/2} a_k\bar S_k}
\mathbf{1}_{\cap_{0\leq k\leq t}\{-d_1\leq  S_k\}}\right]-\PP_{(0,a)}^{(t,b)}\left( \cap_{0\leq k\leq t}\{-d_1\leq S_k\}\right)(1+d^{-\kappa})
\notag\\
\leq& \sum_{v\geq 1,d_0\leq k\leq t/2,w\geq v  k^{1/2+\varepsilon_0}}e^{v d^{-\kappa}}\,\PP_{(0,a)}^{(t,b)}\left(\{\bar S_k\in[w,w+1]\}\cap_{j\leq k}\{S_j \geq -d_1)\}\right)\notag\\
\ll&
\sum_{v\geq 1,d_0\leq k\leq t/2,w\geq v  k^{1/2+\varepsilon_0}}e^{v d^{-\kappa}}\,\PP_{(0,a)}^{(t,b)}
\left(\bar S_k\in[w,w+1]\right)\times\\
&\hspace{5cm}
\Big(\sup_{c\in[w,w+1]+a\frac{t-k}{t}+b\frac{k}{t}}\PP_{(0,c)}^{(t-k,b)}\left(\cap_{1\leq j\leq t-k}\{S_j \geq -d_1\}\right)\Big)\label{eqn:interm3}
\end{align}
where we used the Markov property for the second inequality. To bound the first probability above, note that under $\PP_{(0,a)}^{(t,b)}$, the random variable $\bar S_k$ is centered, Gaussian with variance $k-\frac{k^2}{t}\asymp k$. For the second probability, 
from Lemma \ref{lem:upperBallot}, we have  uniformly in all parameters
\begin{equation}
\label{eqn: ballot webb}
\PP_{(0,c)}^{(t-k,b)}\left(\cap_{1\leq j\leq t-k}\{S_j \geq -d_1\}\right)\ll \frac{(w+a\frac{t-k}{t}+b\frac{k}{t})b}{t}.
\end{equation}
This allows to bound the left-hand side of (\ref{eqn:interm3}) with
$$
\sum_{v\geq 1,d_0\leq k\leq t/2,w\geq v  k^{1/2+\varepsilon_0}}
e^{v d^{-\kappa}-c\frac{w^2}{k}}\cdot 
 \frac{(w+a\frac{t-k}{t}+b\frac{k}{t})b}{t}
$$
for some absolute $c>0$. 
The above sum over $w$ and then $v$ is $\ll e^{-c'k^{2\varepsilon_0}}$ for some absolute $c'>0$.
We conclude that uniformly in our parameters,  the left-hand side of (\ref{eqn:interm3})  is bounded with
$$ 
\frac{b}{t}\sum_{d_0\leq k\leq t/2}
e^{-c'k^{2\varepsilon_0}}(1+a\frac{t-k}{t}+b\frac{k}{t})\ll \frac{ab}{t}e^{-c' d_0^{2\varepsilon_0}}\ll d^{-\kappa}\frac{ab}{t}.
$$
Finally, Lemma  \ref{lem:constantBar} yields
$$
\PP_{(0,a)}^{(t,b)}\left( \cap_{0\leq k\leq t}\{-d_1\leq S_k\}\right)=2\frac{ab}{\sigma}(1+\OO(d^{-c})).
$$
This concludes the proof.

\begin{lemma}\label{lem:freeEndBallot} There exists $C=C(\kappa)>0$ such that for any $t\geq 10$,  $a\geq 1$, we have
\begin{equation}\label{eqn:oneSided}
\PP_{(0,a)}\left(\cap_{0\leq k\leq t}\{S_k\geq 0\}\right)\leq C\frac{a}{\sqrt{t}}.
\end{equation}
\end{lemma}

\begin{proof}
By monotonicity in $a$, we can consider $a\in\mathbb{N}$ without loss of generality . Moreover,  we have
$$
\PP_{(0,a)}\left(\cap_{0\leq k\leq t}\{S_k\geq 0\}\right)\leq
\PP_{(0,a)}\left(\cap_{0\leq k\leq \frac{t}{a^2}}\{\frac{S_{k a^2}}{a}\geq 0\}\right)=\PP_{(0,1)}\left(\cap_{0\leq k\leq \frac{t}{a^2}}\{\tilde S_k\geq 0\}\right),
$$
where $\tilde S_k:=S_{ka^2}/a$ has independent Gaussian increments with variance in $[\kappa,\kappa^{-1}]$, like $S$. This proves that (\ref{eqn:oneSided}) only needs to be proved for $a=1$.

Consider $T=\min(t,\min\{k\leq t\mid S_k\leq 0\})$. By the stopping time theorem,
$
\E[S_{T}]=1,
$
so that
$$
\E[|S_t|\mathbf{1}_{T\geq t}]=1+\E[-S_{T}\mathbf{1}_{T<t}].
$$Moreover,  we have the correlation inequality 
\begin{equation}\label{eqn:FKG}
\E[|S_t|\mathbf{1}_{T>t}]\geq \E[\max(0,S_t)]\PP[T>t],
\end{equation}
which is a simple consequence of the Harris inequality. Indeed, consider the random walk $Z^{(n)}_k=1+\frac{1}{\sqrt{n}}\sum_{1\leq j\leq kn}\varepsilon_j$ with independent Bernoulli random variables $\varepsilon_k$, and $I^{(n)}=\{\lfloor n{\rm Var} S_k\rfloor,k\geq 1\}$; denote $U^{(n)}=\min\{k\in I^{(n)}:Z^{(n)}_k\leq 0\}$. Then the functions 
$\mathbf{1}_{U^{(n)}>t}$ and $\max(0,Z^{(n)}_t)$ are non-decreasing functions of $(\varepsilon_k)_{k\leq nt}$, so that $\E[\max(0,Z^{(n)}_t)\mathbf{1}_{U^{(n)}>t}]\geq \E[\max(0,Z^{(n)}_t)]\cdot \PP[U^{(n)}>t]$. This implies \eqref{eqn:FKG}
by taking $n\to\infty$.

We have obtained
\begin{equation}\label{eqn:inter}
\PP[T>t]\leq C\frac{1+\E[-S_{T}\mathbf{1}_{T<t}]}{\sqrt{t}},
\end{equation}
and we will now prove that
\begin{equation}\label{eqn:shift}
\E[-S_{T}\mathbf{1}_{T<t}]\leq C_1,
\end{equation}
for some $C_1>0$ uniform in $t$,  which together with  (\ref{eqn:inter}) will conclude the proof of (\ref{eqn:oneSided}).

To prove (\ref{eqn:shift}),  first note that $\E[-S_{T}\mathbf{1}_{T<t}]\leq \E[-S_{T_0}]$ where $T_0=\min\{k\geq 0\mid S_k\leq 0\}$.  We now consider 
$$Z_n=\sum_{k\geq 0}\mathbf{1}_{S_k\in[n,n+1),k<T_0},\ n\geq 0,$$
the time spent by $S$ in $[n,n+1)$ before it hits $0$. Define $U_{0}=0$ and $U_{i+1}=\min\{u\geq U_i+n^2:S_u\in[n,n+1)\}$. For any $\lambda\geq n^2$ we have the inclusion
$$
\{Z_n\geq \lambda\}\subset\cap_{i\leq \lambda/n^2}\left\{S_{U_i+n^2}-S_{U_i}\geq -(n+1)\right\}.
$$
By the strong Markov property  the events on the right-hand side are independent, and there exists $\alpha=\alpha(\kappa)$ such that each such event has probability at most $1-\alpha$, uniformly in $n$.
This implies
$
\E(Z_n)\leq C n^2
$
for some $C=C(\kappa)$, a key estimate in the last inequality below: for any $\ell\geq 1$ we have
\begin{multline*}
\PP[|S_{T_0}|\geq \ell]\leq\sum_{k,n\geq 0}\PP[S_k\in[n,n+1),|S_{k+1}-S_k|\geq \ell+(n+1),k<T_0]\\
=\sum_{k,n\geq 0}\PP[S_k\in[n,n+1),k<T_0]\cdot\PP[|S_{k+1}-S_k|\geq \ell+(n+1)]
\leq \sum_{n\geq 0}e^{-c(\ell+n)^2}\E[Z_n]\leq Ce^{-c\ell^2},
\end{multline*}
which immediately implies $\E[-S_{T_0}]\leq C_1$ and concludes the proof.
\end{proof}

\begin{lemma}\label{lem:upperBallot} With the same notations as Proposition \ref{prop:barrier},  there exists $C=C(\kappa)>0$ such that for any $t\geq 10$,  $a,b\geq 1$ we have
\begin{equation}\label{eqn:upperBallot}
\PP_{(0,a)}^{(t,b)}\left( \cap_{0\leq k\leq t}\{S_k\geq 0\}\right)\leq C\frac{ab}{t}.
\end{equation}
\end{lemma}

\begin{proof}
We first assume that $a,b\leq \sqrt{t}$.  Let $n_1=\lfloor t/3\rfloor$ and $n_2=\lfloor 2t/3\rfloor$, and $p_t(x)=e^{-x^2/(2t)}/\sqrt{2\pi t}$, and abbreviate $\PP_{(0,a)}^{(t,b)}=\PP_{(0,a)}^{(t,b)}(S>0)$.  Then $ \PP_{(0,a)}^{(t,b)}(S>0)$ is equal to
\begin{multline*}
\frac{1}{p_t(a-b)}\iint_{x_1,x_2>0}p_{n_1}(x_1-a)\PP_{(0,a)}^{(n_1,x_1)}p_{n_2-n_1}(x_2-x_1)\PP_{(n_1,x_1)}^{(n_2,x_2)}p_{t-n_2}(b-x_2)\PP_{(n_2,x_2)}^{(t,b)}\rd x_1\rd x_2\\
\leq C\int_{x_1}p_{n_1}(x_1-a)\PP_{(0,a)}^{(n_1,x_1)} \rd x_1\int_{x_2>0}p_{t-n_2}(b-x_2)\PP_{(n_2,x_2)}^{(t,b)}\rd x_2\leq C\frac{ab}{t},
\end{multline*}
where we have used the trivial bounds $\PP_{(n_1,x_1)}^{(n_2,x_2)}\leq 1$,  $p_{n_2-n_1}(x_2-x_1)\leq Ct^{-1/2}$,  the estimate (valid for $a,b\leq \sqrt{t}$) $(p_t(a-b))^{-1}\leq C\sqrt{t}$, and Lemma \ref{lem:freeEndBallot}.

For the general case,  we can assume $a<\sqrt{t}<b$ and $ab<t$.  Let $B$ be a Brownian bridge from $a$ ($s=0$) to $b$ ($s=\sigma$).  There exists $s_1<\dots<s_t=\sigma$ such that $(S_k)_{k\leq t}$ and $(B_{s_k})_{k\leq t}$ have the same distribution.  
Moreover,  from \cite[pages 21, 22]{Bra1983}, by a simple coupling argument the function 
\begin{equation}\label{eqn:monotone}b\mapsto
\PP_{(a,0)}^{(b,t)}(B_s>0,0\leq s\leq \sigma\mid B_{s_i}>0,0\leq i\leq t)\ 
\mbox{ is non-decreasing}.
\end{equation}
This implies
\begin{multline*}
\PP_{(a,0)}^{(b,t)}(B_{s_i}\geq 0,1\leq i\leq t)=\frac{\PP_{(a,0)}^{(b,t)}(B_{s_i}\geq 0,1\leq i\leq t)}{\PP_{(a,0)}^{(b,t)}(B_{s}\geq 0,0\leq s\leq \sigma)}\PP_{(a,0)}^{(b,t)}(B_{s}\geq 0,0\leq s\leq \sigma)\\
\leq\frac{\PP_{(a,0)}^{(\sqrt{t},t)}(B_{s_i}\geq 0,1\leq i\leq t)}{\PP_{(a,0)}^{(\sqrt{t},t)}(B_{s}\geq 0,0\leq s\leq \sigma)}\PP_{(a,0)}^{(b,t)}(B_{s}\geq 0,0\leq s\leq \sigma).
\end{multline*}
From Lemma \ref{lem:explicit}, the denominator is lower-bounded with $c\frac{a\sqrt{t}}{t}$ and the last probability is upper-bounded with $\frac{ab}{t}$. And from the previously discussed case, the numerator is at most $\frac{a\sqrt{t}}{t}$. This concludes the proof.
\end{proof}

\begin{lemma}\label{lem:constantBar} With the same notations as Proposition \ref{prop:barrier},  there exists $C=C(\kappa)>0$ such that for any $t\geq 10$,   $y\leq t^{1/10}$, $1\leq a,b\leq y-1$ we have
$$
\PP_{(0,a)}^{(t,b)}\left( \cap_{0\leq k\leq t}\{S_k\geq 0\}\right)=
\frac{2ab}{\sigma}\cdot \left(1+\OO_{\alpha,\delta,\kappa}(d^{-c})\right).
$$
\end{lemma}

\begin{proof}
The lower bound follows directly from Lemma \ref{lem:explicit}.  

For the upper bound,  consider first the case $a=b=d$.
Note that for any $k\leq t-1$ and $u,v>0$,  under $\PP_{(k,u)}^{(k+1,v)}$ we can decompose
$$
B_s=(k+1-s)v+(s-k)u+\tilde B_{s-k}-(s-k)\tilde B_{1},
$$
where $\tilde B$ is a standard Brownian motion.
Therefore, if there exists $s\in[k,k+1]$ such that $B_s<0$, we have $\max_{0\leq u\leq 1}|\tilde B_{u}|+|\tilde B_{1}|>\min(u,v)$.
As $|\tilde B_1|+\max_{0\leq u\leq 1}|\tilde B_u|$ is clearly dominated by $|\mathcal{N}|$ with $\mathcal{N}$ a Gaussian random variable with variance $\OO(1)$,
by  a union bound, we obtain 
\begin{align*}
&\PP_{(0,d)}^{(t,d)}\left( \cap_{0\leq k\leq t}\{S_k\geq 0\}\right)-\PP_{(0,d)}^{(t,d)}\left( \cap_{0\leq s\leq t}\{B_s\geq 0\}\right)\\
\leq&
\sum_{u.v\geq 0,k\leq t-1}\max_{x\in[u,u+1]}\PP_{(0,d)}^{(k,x)}(S_i>0,1\leq i\leq k)\cdot \PP(|\mathcal{N}|>\min(u,v))\\
&\cdot\max_{y\in[v,v+1]}\PP_{(k+1,y)}^{(t,d)}(S_i>0,k+1\leq i\leq t)\cdot\PP(B_k\in[u,u+1],B_{k+1}\in[v,v+1]).
\end{align*}
All terms above can be bounded,  giving  the estimate
\begin{multline*}
\sum_{0\leq k\leq t/2,u,v\geq 0}\frac{du}{k+1}\frac{dv}{t-k}e^{-c\min(u,v)^2}\frac{e^{-c\frac{(u-d)^2}{k+1}}}{\sqrt{k+1}}e^{-c(v-u)^2}\\
\ll\frac{d^2}{t}\sum_{0\leq k\leq t/2,u\geq 0}\frac{u^2}{(k+1)^{3/2}}e^{-cu^2-c\frac{(u-d)^2}{k+1}}\\
\ll\frac{d^2}{t}\sum_{0\leq k\leq t/2}\frac{1}{(k+1)^{3/2}}e^{-c\frac{d^2}{k+1}}\leq C\frac{d^2}{t}d^{-1/4+\e},
\end{multline*}
for any arbitrary $\varepsilon>0$, concluding the proof in the case $a=b=d$. 

For the general case,  from (\ref{eqn:monotone}) assuming $b>a$ without loss of generality,  we have
$$
\begin{aligned}
\PP_{(0,a)}^{(t,b)}\left( \cap_{0\leq k\leq t}\{S_k\geq 0\}\right)&\leq \PP_{(0,a)}^{(t,b)}\left( \cap_{0\leq s\leq \sigma}\{B_s\geq 0\}\right)\cdot\frac{\PP_{(0,a)}^{(t,a)}\left( \cap_{0\leq k\leq t}\{S_k\geq 0\}\right)}{\PP_{(0,a)}^{(t,a)}\left( \cap_{0\leq s\leq \sigma}\{B_s\geq 0\}\right)}\\
&\leq \frac{2ab}{\sigma}(1+\OO(d^{-c}))
\end{aligned}
$$
from the previous discussion, concluding the proof.
\end{proof}


\begin{bibdiv}
\begin{biblist}

\end{biblist}
\end{bibdiv}


\begin{bibdiv}
\begin{biblist}


\bib{ArgBelBou2017}{article}{
   author={Arguin, L.-P.},
   author={Belius, D.},
   author={Bourgade, P.},
   title={Maximum of the characteristic polynomial of random unitary
   matrices},
   journal={Comm. Math. Phys.},
   volume={349},
   date={2017},
   number={2},
   pages={703--751}
}


\bib{ArgBelHar2017}{article}{
   author={Arguin, L.-P.},
   author={Belius, D.},
   author={Harper, A. J.},
   title={Maxima of a randomized Riemann zeta function, and branching random
   walks},
   journal={Ann. Appl. Probab.},
   volume={27},
   date={2017},
   number={1},
   pages={178--215}
}


\bib{ArgBelBouRadSou2019}{article}{
   author={Arguin, L.-P.},
   author={Belius, D.},
   author={Bourgade, P.},
   author={Radziwi\l \l , M.},
   author={Soundararajan, Kannan},
   title={Maximum of the Riemann zeta function on a short interval of the
   critical line},
   journal={Comm. Pure Appl. Math.},
   volume={72},
   date={2019},
   number={3},
   pages={500--535}
}


\bib{ArgBouRad2020}{article}{
  title={The Fyodorov-Hiary-Keating Conjecture. I},
   author={Arguin, L.-P.},
   author={Bourgade, P.},
   author={ Radziwi{\l}{\l}, M.},
  journal={preprint arXiv:2007.00988},
  year={2020}
}



\bib{Bra1983}{article}{
   author={Bramson, M.},
   title={Convergence of solutions of the Kolmogorov equation to travelling
   waves},
   journal={Mem. Amer. Math. Soc.},
   volume={44},
   date={1983},
   number={285},
   pages={iv+190}
}




\bib{Bra1978}{article}{
   author={Bramson, M.},
   title={Maximal displacement of branching Brownian motion},
   journal={Comm. Pure Appl. Math.},
   volume={31},
   date={1978},
   number={5},
   pages={531--581}
}


\bib{BraDinZei2016}{article}{
   author={Bramson, M.},
   author={Ding, J.},
   author={Zeitouni, O.},
   title={Convergence in law of the maximum of the two-dimensional discrete
   Gaussian free field},
   journal={Comm. Pure Appl. Math.},
   volume={69},
   date={2016},
   number={1},
   pages={62--123}
}


\bib{BraDinZei2016bis}{article}{
   author={Bramson, M.},
   author={Ding, J.},
   author={Zeitouni, O.},
   title={Convergence in law of the maximum of nonlattice branching random
   walk},
   journal={Ann. Inst. Henri Poincar\'{e} Probab. Stat.},
   volume={52},
   date={2016},
   number={4},
   pages={1897--1924}
}




\bib{CarLed01}{article}{
   author={Carpentier,  L.},
   author={Le Doussal, P.},
   title={Glass transition of a particle in a random potential, front selection in nonlinear renormalization group, and entropic phenomena in liouville and sinh-gordon models},
   journal={Physical review. E, Statistical, nonlinear, and soft matter physics},
   volume={63},
   date={2001}
}

\bib{ChhMadNaj2018}{article}{
   author={Chhaibi, R.},
   author={Madaule, T.},
   author={Najnudel, J.},
   title={On the maximum of the ${\rm C}\beta {\rm E}$ field},
   journal={Duke Math. J.},
   volume={167},
   date={2018},
   number={12},
   pages={2243--2345}
}



\bib{Din2013}{article}{
   author={Ding, J.},
   title={Exponential and double exponential tails for maximum of
   two-dimensional discrete Gaussian free field},
   journal={Probab. Theory Related Fields},
   volume={157},
   date={2013},
   number={1-2},
   pages={285--299}
}



\bib{DinZei2014}{article}{
   author={Ding, J.},
   author={Zeitouni, O.},
   title={Extreme values for two-dimensional discrete Gaussian free field},
   journal={Ann. Probab.},
   volume={42},
   date={2014},
   number={4},
   pages={1480--1515}
}




\bib{FyoHiaKea2012}{article}{
   author={Fyodorov, Y.},
   author={Hiary, G.},
   author={Keating, J.},
   title={Freezing Transition, Characteristic Polynomials of Random Matrices, and the Riemann Zeta Function},
   journal={Physical Review Letters},
   volume={108},
   date={2012}
}


\bib{FyoKea2014}{article}{
  title={Freezing transitions and extreme values: random matrix theory, and disordered landscapes},
  author={Fyodorov, Y.}
  author={Keating, J.},
  journal={Philosophical Transactions of the Royal Society A},
  volume={372},
  number={2007},
  pages={20120503},
  year={2014},
  publisher={The Royal Society Publishing.}
}




\bib{Har2019}{article}{
   author={Harper, A. J.},
   title={On the partition function of the Riemann zeta function, and the Fyodorov–Hiary– Keating conjecture},
   journal={preprint arXiv:1906.05783},
   date={2019}
}


\bib{KatSar1999}{book}{
   author={Katz, N. M.},
   author={Sarnak, P.},
   title={Random matrices, Frobenius eigenvalues, and monodromy},
   series={American Mathematical Society Colloquium Publications},
   volume={45},
   publisher={American Mathematical Society, Providence, RI},
   date={1999},
   pages={xii+419}
}



\bib{KeaSna2000}{article}{
   author={Keating, J. P.},
   author={Snaith, N. C.},
   title={Random matrix theory and $\zeta(1/2+it)$},
   journal={Comm. Math. Phys.},
   volume={214},
   date={2000},
   number={1},
   pages={57--89}
}

\bib{Mon1972}{article}{
   author={Montgomery, H. L.},
   title={The pair correlation of zeros of the zeta function},
   conference={
      title={Analytic number theory},
      address={Proc. Sympos. Pure Math., Vol. XXIV, St. Louis Univ., St.
      Louis, Mo.},
      date={1972},
   },
   book={
      publisher={Amer. Math. Soc., Providence, R.I.},
   },
   date={1973},
   pages={181--193}
}

\bib{MonVau1974}{article}{
   author={Montgomery, H. L.},
   author={Vaughan, R. C.},
   title={Hilbert's inequality},
   journal={J. London Math. Soc. (2)},
   volume={8},
   date={1974},
   pages={73--82}
}

\bib{Naj2018}{article}{
   author={Najnudel, J.},
   title={On the extreme values of the Riemann zeta function on random
   intervals of the critical line},
   journal={Probab. Theory Related Fields},
   volume={172},
   date={2018},
   number={1-2},
   pages={387--452}
}




\bib{PaqZei2018}{article}{
   author={Paquette, E.},
   author={Zeitouni, O.},
   title={The maximum of the CUE field},
   journal={Int. Math. Res. Not. IMRN},
   date={2018},
   number={16},
   pages={5028--5119}
}



\bib{PaqZei2022}{article}{
   author={Paquette, E.},
   author={Zeitouni, O.},
   title={The extremal landscape for the C$\beta$E ensemble},
   journal={preprint arXiv:2209.06743},
   date={2022}
}


\bib{Sou2009}{article}{
   author={Soundararajan, K.},
   title={Moments of the Riemann zeta function},
   journal={Ann. of Math. (2)},
   volume={170},
   date={2009},
   number={2},
   pages={981--993}
}







\end{biblist}
\end{bibdiv}
\end{document}